\documentclass[12pt]{amsart}
\usepackage{stmaryrd} 
\usepackage{url}
\usepackage[all]{xy}

\newtheorem{theorem}{Theorem}[subsection]
\newtheorem{lemma}[theorem]{Lemma}
\newtheorem{cor}[theorem]{Corollary}

\newtheorem{prop}[theorem]{Proposition}
\theoremstyle{definition}
\newtheorem{defn}[theorem]{Definition}
\newtheorem{hypothesis}[theorem]{Hypothesis}
\newtheorem{example}[theorem]{Example}
\newtheorem{remark}[theorem]{Remark}
\newtheorem{convention}[theorem]{Convention}
\newtheorem{notation}[theorem]{Notation}

\newcommand{\CC}{\mathbb{C}}

\newcommand{\DD}{\mathbb{D}}
\newcommand{\FF}{\mathbb{F}}

\newcommand{\QQ}{\mathbb{Q}}
\newcommand{\RR}{\mathbb{R}}

\newcommand{\ZZ}{\mathbb{Z}}

\newcommand{\be}{\mathbf{e}}
\newcommand{\bv}{\mathbf{v}}
\newcommand{\bw}{\mathbf{w}}

\newcommand{\calC}{\mathcal{C}}

\newcommand{\calH}{\mathcal{H}}

\newcommand{\calM}{\mathcal{M}}
\newcommand{\calO}{\mathcal{O}}

\newcommand{\calR}{\mathcal{R}}

\newcommand{\gothm}{\mathfrak{m}}

\newcommand{\gotho}{\mathfrak{o}}

\newcommand{\dual}{\vee}

\DeclareMathOperator{\alg}{alg}
\DeclareMathOperator{\an}{an}
\DeclareMathOperator{\Aut}{Aut}
\DeclareMathOperator{\bd}{bd}
\DeclareMathOperator{\coker}{coker}

\DeclareMathOperator{\End}{End}
\DeclareMathOperator{\Frac}{Frac}
\DeclareMathOperator{\Gal}{Gal}
\DeclareMathOperator{\GL}{GL}
\DeclareMathOperator{\Hom}{Hom}
\DeclareMathOperator{\inte}{int}

\DeclareMathOperator{\rank}{rank}

\DeclareMathOperator{\trdeg}{trdeg}

\numberwithin{equation}{theorem}
 
\begin{document}
\title[Structure of connections on nonarchimedean curves]{Local and global structure of connections on
nonarchimedean curves}
\author{Kiran S. Kedlaya}
\date{29 August 2014}
\thanks{The author was supported by NSF (grant DMS-1101343), DARPA (grant HR0011-09-1-0048),
MIT (NEC Fund, Cecil and Ida Green professorship), UCSD (Stefan E. Warschawski professorship).}

\begin{abstract}
Consider a vector bundle with connection on a $p$-adic analytic curve in the sense of Berkovich.
We collect some improvements and refinements of recent results on the structure of such connections,
and on the convergence of local horizontal sections. This builds on work from the author's 2010 book
and on subsequent improvements by Baldassarri and Poineau--Pulita. One key result exclusive to this paper is that the convergence polygon of a connection is locally constant around every type 4 point.
\end{abstract}

\maketitle

\section*{Introduction}

The theory of $p$-adic ordinary differential equations 
has been an active part of number theory ever since the pioneering
work of Dwork, starting with his $p$-adic analytic proof of the rationality aspect of the Weil conjectures circa 1960
(predating the development of \'etale cohomology). The subsequent half-century saw slow but substantial progress
on the question of convergence of solutions of $p$-adic differential equations; in that time, new spheres of application
(rigid cohomology, $p$-adic Hodge theory, numerical computation of zeta functions, $p$-adic dynamical systems) have attracted additional attention to the area. 
A broad survey of the theory of $p$-adic differential equations
has been given recently by the author in the book \cite{kedlaya-book}.

At about the time that \cite{kedlaya-book} was written, it was observed by Baldassarri
\cite{baldassarri, baldassarri-divizio}
that the classical theory
of $p$-adic differential equations could be rearticulated much more clearly using Berkovich's
language of analytic geometry over complete nonarchimedean fields. 
That is because the classical theory is heavily concerned with the convergence of local solutions
of $p$-adic differential equations around certain \emph{generic points}, which appear naturally
in Berkovich's framework on an equal footing with rigid analytic points.
In this language, one can also naturally treat general $p$-adic curves, not just subspaces
of the affine line, by using semistable models to obtain scaling parameters;
Baldassarri demonstrated this in \cite{baldassarri} by establishing continuity of the
radius of convergence for a differential module over a semistable $p$-adic curve.

The radius of convergence function for a differential module over a curve measures only the joint radius
of convergence of all local horizontal sections around a point. A finer invariant is the \emph{convergence
polygon}, a Newton polygon whose slopes record the extent to which there exist subspaces of the local horizontal
sections which converge on larger discs. Building on the results of \cite{kedlaya-book},
it has been shown recently by
Poineau and Pulita \cite{pulita-poineau2, pulita-poineau} 
that the convergence polygon is again a continuous function which factors through the retraction onto
some finite skeleton (as in the work of Payne \cite{payne}).
Informally, this means that the convergence of local horizontal sections is controlled by
finitely many numerical invariants. Another proof is included in this paper, while a simplified version of our proof will appear in \cite{baldassarri-kedlaya}. While formally different, these proofs share many common ingredients; for instance, our key Lemma~\ref{L:variation0 convexity} is materially equivalent to \cite[Proposition~7.5]{pulita-poineau}. In fact, the main difference between the arguments here and those in \cite{pulita-poineau} is that the combinatorial argument there is replaced by a compactness argument.

The purpose of this paper is to collect some results about differential modules on nonarchimedean analytic curves
over fields of characteristic 0 which refine and extend the aforementioned results as well as some
other results from \cite{kedlaya-book}. Here is a partial list of the new results of the present paper.
\begin{itemize}
\item
We make a finer analysis of refined differential modules over a field of analytic functions
than is made in \cite{kedlaya-book}; see \S\ref{subsec:fields2}.
This leads to results about refined differential modules on open annuli; see
\S\ref{subsec:annuli refined}.

\item
We provide more detailed discussion of the theory of exponents for differential modules on annuli satisfying
the Robba condition (existence of horizontal sections over any open disc); see \S\ref{sec:Robba}
and \S\ref{sec:Robba2}.

\item
We generalize the $p$-adic local monodromy theorem to arbitrary differential modules over an open annulus at one boundary, with no hypotheses on Frobenius structures or $p$-adic exponents; see \S\ref{sec:solvable}.

\item
We show that the convergence polygon of a differential module on a curve is constant locally around any
point of type 4; see \S\ref{subsec:solvable}. This strengthens the continuity theorem of \cite{pulita-poineau2, pulita-poineau3, pulita-poineau}. 

\item
We show that every curve admits a triangulation such that locally at any interior point, the connection decomposes
into a particularly simple form; see \S\ref{subsec:clean}. Such triangulations and decompositions can be used to give a global version
of the Christol-Mebkhout index formula; see \cite{pulita-poineau4, pulita-poineau5}
for some arguments along these lines.
\end{itemize}

As in \cite{kedlaya-book}, we have made an effort to maintain as much parity as possible between the cases
of zero and positive residual characteristic. One unavoidable complication in the latter case 
is the existence of some pathologies
in the theory of regular singularities caused by the existence of $p$-adic Liouville numbers (numbers which are
not integers but which admit extremely good integer approximations).
These complications generally emerge when considering cohomology; in this paper, we primarily limit ourselves to
statements of a ``precohomological''
nature, for which one can skirt these complications with some extra work.

Note that while many of the interesting applications of $p$-adic differential equations involve spaces of
dimension greater than 1, in this paper we follow the model of \cite{kedlaya-book} 
and confine attention to \emph{ordinary} $p$-adic differential equations. It is of course natural to consider
also higher-dimensional spaces; in so doing, one should be able to obtain a unification of some existing
work. Such work would include the study of good formal structures for formal flat meromorphic connections
\cite{kedlaya-goodformal1, kedlaya-goodformal2} in the case of zero residual characteristic and 
semistable reduction for overconvergent $F$-isocrystals \cite{kedlaya-semistable1, kedlaya-semistable2,
kedlaya-semistable3, kedlaya-semistable4} in the case of positive residual characteristic.

\subsection*{Acknowledgments}
Thanks to Francesco Baldassarri for arranging a visit to Padova in September 2012 during which some of this work
was completed. Thanks also to Matt Baker, Bruno Chiarellotto, Xander Faber, Andrea Pulita, and J\'er\^ome Poineau for helpful discussions.

\section{Preliminaries}

We begin with some assorted preliminary definitions and arguments.
This also provides an opportunity to set running notation for the whole paper.

\setcounter{theorem}{0}
\begin{notation}
Throughout the paper, let $K$ denote an \emph{analytic field} (a field equipped with a nonarchimedean multiplicative norm $|\cdot|$ with respect to which it is complete) of characteristic $0$.
Let $\gotho_K$, $\gothm_K$, and $\kappa_K$ denote the valuation subring, maximal ideal, and residue field of $K$,
respectively. Let $p$ denote the characteristic of $\kappa_K$;
put $\omega = 1$ if $p=0$ and $\omega = p^{-1/(p-1)}$ if $p>0$.
Let $\CC$ denote a completed algebraic closure of $K$.
\end{notation}

\subsection{A lemma on linear groups}

We need a bit of elementary analysis of linear groups in the spirit of Andr\'e's abstract analysis of filtrations
\cite{andre-filtrations}. This will be used to analyze the structure of the automorphism groups of certain Tannakian categories,
especially those generated by refined differential modules over fields
(\S \ref{subsec:fields2})
and solvable differential modules over annuli
(\S \ref{sec:solvable}).
For the formalism of Tannakian categories, including the Tannaka-Krein duality theorem,
see \cite{tannakian}.

\begin{lemma} \label{L:Jordan}
Let $F$ be a field of characteristic $0$. Fix a positive integer $n$ and let $G_0 \subseteq G_1 \subseteq \cdots$
be an increasing sequence of finite subgroups of $\GL_n(F)$ such that $G_i$ is normal in $G_j$ whenever $i \leq j$.
\begin{enumerate}
\item[(a)]
The union $G = \bigcup_{i=0}^\infty G_i$ contains an abelian normal
subgroup $H$ of finite index.
\item[(b)]
There exists an index $i$ such that $G/G_i$ is isomorphic to a subgroup of $(\QQ/\ZZ)^n$,
and in particular is abelian.
\end{enumerate}
\end{lemma}
\begin{proof}
By Jordan's theorem on finite linear groups \cite[Chapter~36]{curtis-reiner}, 
there exists a constant $f(n)$ such that each $G_i$ contains an abelian normal
subgroup of index at most $f(n)$. Let $S_i$ be the set of abelian normal subgroups of $G_i$ of index at most $f(n)$.
For each $H_j \in S_j$ and each $i \leq j$, the map $G_i / (S_i \cap H_j) \to G_j/H_j$ is injective,
so $G_i \cap H_j \in S_i$.
We may thus assemble the sets $S_i$ into an inverse system via restriction, and the inverse limit is necessarily nonempty by Tikhonov's theorem. This proves (a).

Given (a), let $\overline{F}$ be an algebraic closure of $F$. Then $H$ is an abelian torsion group which embeds into $(\overline{F}^*)^n$. This implies that $H$ is isomorphic to a subgroup of
$(\QQ/\ZZ)^n$, as then is any quotient of $H$.
Note also that since the group $G/H$ is finite and is the union of its subgroups $G_i/(G_i \cap H)$, there must exist an index $i$ 
for which the inclusion $G_i/(G_i \cap H) \to G/H$ is bijective.
The group $G/G_i$ is then isomorphic to the abelian group $H/(G_i \cap H)$.
This proves (b).
\end{proof}

\begin{prop} \label{P:finite Tannakian}
Let $F$ be a field of characteristic $0$. Let $V$ be a finite-dimensional $F$-vector space.
Let $G$ be an algebraic subgroup of $\GL(V)$.
Let $\{G^r\}_{r \in \RR}$ be a family of normal algebraic subgroups of $G$.
For $r \geq -\infty$, put $G^{r+} = \cup_{s>r} G^s$. Assume also the following conditions.
\begin{enumerate}
\item[(a)]
For every $r,s \in \RR$ with $r \leq s$, $G^s$ is a normal subgroup of $G^r$.
\item[(b)]
For every $s \in \RR$, there exists $r < s$ such that $G^r = G^s$.
\item[(c)]
There exists $r \in \RR$ such that $G^r$ is the trivial group.
\item[(d)]
For every $r \in \RR$ for which $G^{r+}$ is finite 
and all nonnegative integers $g,h$, the $G^{r+}$-invariant subspace of $(V^\dual)^{\otimes g} \otimes V^{\otimes h}$
admits a direct sum decomposition into $G$-stable subspaces, each of which restricts to an
isotypical representation of $G^r/G^{r+}$.
\item[(e)]
In (d), the isotypical representations of $G^r/G^{r+}$ that occur all have finite image.
\item[(f)]
For all nonnegative integers $g,h$ and every one-dimensional
$G$-stable subspace $W$ of $(V^\dual)^{\otimes g} \otimes V^{\otimes h}$,
the image of $G^{-\infty+}$ in $\GL(W)$ is finite.
\end{enumerate}
Then $G^{-\infty+}$ is itself finite.
\end{prop}
\begin{proof}
Let $S$ be the set of $r \in \RR$ for which $G^r$ is finite.
By (a), the set $S$ is up-closed. By (b), the set $S$ does not contain its infimum.
By (c), the set $S$ is nonempty.

Put $r = \inf S$; by the previous paragraph, $r \notin S$. Suppose by way of contradiction that $G^{r+}$ is infinite.
By Lemma~\ref{L:Jordan}, there exists $s_0>r$ such that $G^{r+}/G^{s_0}$ 
embeds into a product of finitely many copies of $\QQ/\ZZ$.
By Tannaka-Krein duality, we can choose $g,h$ so that $(V^\dual)^{\otimes g} \otimes V^{\otimes h}$
contains a $G$-stable subspace $X$ on which $G^{s_0}$ acts trivially but $G^s$ acts nontrivially for
some $s \in (r,s_0)$. 
By applying (d) finitely many times (with $r$ replaced by varying choices of $s \in (r,s_0)$), we can split $X$ as a direct sum of $G$-stable summands,
each of which is $G^{r+}$-isotypical. Since $G^{r+}$ is not finite,
we can choose a $G$-stable summand $Y$ of $X$ such that $G^{r+}$ has image in $\GL(Y)$
isomorphic to an infinite subgroup of $\QQ/\ZZ$. 
Put $W = \wedge^{\dim(Y)} Y$; this space occurs as a $G$-invariant subspace of 
$(V^\dual)^{\otimes g} \otimes V^{\otimes h}$ for some possibly different values of $g$ and $h$.
However, the image of $G^{r+}$ in $\GL(W)$ is again isomorphic to an infinite subgroup of $\QQ/\ZZ$, contradicting (f).

We conclude that $G^{r+}$ is finite. 
Suppose now that $r \in \RR$.
By Tannaka-Krein duality, the action of $G^r$ on the direct sum of the 
$G^{r+}$-invariant subspaces of $(V^\dual)^{\otimes g} \otimes V^{\otimes h}$ over all
nonnegative integers $g,h$ is a faithful representation of $G^r/G^{r+}$.
However, by (e), the action on each individual summand factors through a finite group; since $G^r$
is algebraic, this implies that $G^r$ is finite. But then $r \in S$, a contradiction.
We must thus have $r = -\infty$, which yields the desired result.
\end{proof}

We will apply Proposition~\ref{P:finite Tannakian} via the following Tannakian interpretation.
\begin{remark} \label{R:finite Tannakian}
Let $F$ be a field of characteristic $0$.
Let $\calC$ be a Tannakian category equipped with a fibre functor $\omega$ to the category of finite-dimensional
$F$-vector spaces. Assign to each nonzero element $V \in \calC$ an element $r = r(V) \in \RR \cup \{-\infty\}$ depending only on the isomorphism class of $V$, subject to the following conditions.
\begin{enumerate}
\item[(a)]
For any $V \in \calC$, $r(V^\dual) = r(V)$.
\item[(b)]
For any short exact sequence $0 \to V_1 \to V \to V_2 \to 0$ in $\calC$, $r(V) = \max\{r(V_1), r(V_2)\}$.
\item[(c)]
For any $V_1, V_2 \in \calC$, $r(V_1 \otimes V_2) \leq \max\{r(V_1), r(V_2)\}$.
\end{enumerate}
For $V \in \calC$, let $[V]$ denote the Tannakian subcategory of $\calC$ generated by $V$;
note that $r(W) \leq r(V)$ for all $W \in [V]$.
Let $G(V) \subseteq \GL(\omega(V))$ denote the automorphism group of the restriction of $\omega$ to $[V]$;
this is an algebraic group over $F$, so all of its pro-algebraic quotients are also algebraic.
For $r \in \RR$, let $G^r(V)$ be the subgroup of $G(V)$ acting trivially on $\omega(W)$ for all
$W \in [V]$ with $r(W) < r$. For $r \in \RR \cup \{-\infty\}$, put $G^{r+}(V) = \cup_{s>r} G^s(V)$;
if this group is finite,
then it equals the subgroup of $G(V)$ acting trivially on $\omega(W)$ for all $W \in [V]$ with $r(W) \leq r$
(because there exists $s>r$ for which $G^s(V) = G^{r+}(V)$ and hence $r(W) \notin (r,s)$ for all $W \in [V]$).

The groups $G^r(V)$ then satisfy conditions (a),(b),(c)
of Proposition~\ref{P:finite Tannakian}. This is evident for (a) and (c). For (b), note that
the objects $W \in [V]$ for which $G^s(V)$ acts trivially on $\omega(W)$
form a Tannakian category which is finitely generated (because restricting $\omega$ gives a fibre functor
whose automorphism group $G(V)/G^s(V)$ is algebraic, not just pro-algebraic).

To enforce conditions (d),(e),(f) of Proposition~\ref{P:finite Tannakian}, it would suffice to have the following
additional information about $\calC$.
\begin{enumerate}
\item[(i)]
Every $V \in \calC$ with $r(V) > -\infty$
admits a direct sum decomposition $V = \bigoplus_i V_i$ in which each summand
$V_i$ satisfies $r(V_i^\dual \otimes V_i) < r(V)$. (This implies (d).)
\item[(ii)]
For every $V \in \calC$ with $r(V^\dual \otimes V) < r(V)$, there exists a positive integer $n$ such that
$r(V^{\otimes n}) < r(V)$. (Given (i), this implies (e).)
\item[(iii)]
For every $V \in \calC$ with $\dim_F \omega(V) = 1$, there exists a positive integer $n$ such that
$r(V^{\otimes n}) = -\infty$. (This implies (f).)
\end{enumerate}
Note also that if in (ii) and (iii) the integer $n$ can always be taken to be a power of a fixed prime $p$,
then the group $G^{-\infty+}(V)$ is then forced to be not only finite but also a $p$-group.
\end{remark}

\begin{lemma} \label{L:isolate character}
Suppose the conditions of Remark~\ref{R:finite Tannakian} hold
and that in (ii) and (iii) the integer $n$ can always be taken to be a power of a fixed prime $p$. Then for any $V \in \calC$ with 
$r(V) > -\infty$, there exists $W \in \calC$ such that the
action of $G^{-\infty +}(V)$ on $W$ is $\tau$-isotypical
for some character $\tau: G^{-\infty +}(V) \to \GL_1(F)$ of order $p$.
\end{lemma}
\begin{proof}
By Remark~\ref{R:finite Tannakian}, $G^{-\infty +}(V)$ is a finite $p$-group, which must be nontrivial since $r(V) > -\infty$. The group $G^{-\infty +}(V)$ thus admits a character $\tau: G^{-\infty +}(V) \to \GL_1(F)$ of order $p$.
Let $r(\tau) > -\infty$ be the smallest value of $r$ for which $G^{r+}(V) \subseteq \ker(\tau)$, and choose $\tau$ to minimize $r(\tau)$.

By Tannaka-Krein duality for $G^{-\infty +}(V)$, we may choose nonnegative integers $g,h$ such that $\tau$ occurs in the action of $G^{-\infty +}(V)$ on $(V^\dual)^{\otimes g} \otimes V^{\otimes h}$.
Then $\tau$ also occurs in the action of $G^{-\infty +}(V)$ on some irreducible subquotient $W$ of $(V^\dual)^{\otimes g} \otimes V^{\otimes h}$.

Since  $G^{-\infty +}(V)$ is a finite group, its action on $W$ is completely reducible and thus admits an isotypical decomposition. Since $W$ is irreducible, all of its isotypical components must correspond to conjugates of $\tau$ by the action of $G(V)$ on its normal subgroup $G^{-\infty +}(V)$. In particular, each of these conjugates $\tau'$ is a character of order $p$ with $r(\tau') = r(\tau)$.

It follows that $r(W) = r(\tau')$. By property (i) of Remark~\ref{R:finite Tannakian}, we have $r(W^\dual \otimes W) < r(W)$;
however, the irreducible representations of $G^{-\infty +}(V)$ appearing in $W^\dual \otimes W$ are characters of order dividing $p$,
so by our minimization of $r(\tau)$ these characters must be trivial.
That is, $r(W^\dual \otimes W) = -\infty$, which implies that $W$ is $\tau$-isotypical.
\end{proof}

\subsection{A lemma on local fields}
\label{subsec:local fields}

We introduce an auxiliary calculation concerning local fields in positive characteristic.
This is needed for the study of solvable differential modules at type 4 points (\S \ref{subsec:solvable}).
We use without comment some basic facts about higher ramification of local fields, for which see
\cite[Chapter~3]{kedlaya-book} for a brief summary or
\cite{serre-local-fields} for a complete treatment.

\begin{hypothesis} \label{H:local fields}
Throughout \S\ref{subsec:local fields}, assume that $p>0$
and let $k$ be an algebraically closed field of characteristic $p$.
\end{hypothesis}

\begin{defn}
Let $N$ be the pro-unipotent pro-algebraic group over $k$ whose $k$-points are identified
with the $t$-adically continuous $k$-linear automorphisms $\psi$ of $k((t))$ fixing $t$ modulo $t^2$.
The group $N$ is filtered by the pro-algebraic subgroups
\[
N_m = \ker(N \to \Aut(k \llbracket t \rrbracket/t^{m+1})) \qquad (m=1,2,\dots)
\]
for which $N_1 = N$ and each successive quotient $N_m/N_{m+1}$ is isomorphic to the additive group (though not canonically).
We will write $N_t$ and $N_{m,t}$ instead of $N$ and $N_m$ when it is necessary to specify the series variable
$t$ in the notation.
(The analogous construction with $k = \FF_p$ is sometimes called the \emph{Nottingham group}.)
\end{defn}

\begin{hypothesis} \label{H:local fields2}
For the remainder of \S\ref{subsec:local fields}, let $m$ be a positive integer,
and let $\Gamma_m$ be a copy of the additive group over $k$ equipped with a homomorphism $\Gamma_m \to N_m$ 
of pro-algebraic groups over $k$
such that the composition $\Gamma_m \to N_m \to N_m/N_{m+1}$ is surjective and separable.
\end{hypothesis}

\begin{example}
The key case of Hypothesis~\ref{H:local fields2} for our intended applications is the one in which
$m=1$ and $\Gamma_m$ is the group of translations $t^{-1} \mapsto t^{-1} + c$. However, we will need
the full generality of Hypothesis~\ref{H:local fields2} in order to make certain inductive arguments
in towers of field extensions.
\end{example}

\begin{lemma} \label{L:local fields2}
Let $E$ be a $\ZZ/p\ZZ$-extension of $k((t))$ equipped with an extension of the action of $\Gamma_m$.
\begin{enumerate}
\item[(a)]
The ramification number $e$ of $E$ is a positive integer no greater than $m$ and not divisible by $p$.
\item[(b)]
Put $m' = (m-e)p + e$.
For any $k$-linear homeomorphism $E \cong k((u))$, the action of $\Gamma_m$ on $E$ induces a homomorphism
$\Gamma_m \to N_{m',u}$ of pro-algebraic groups such that the composition $\Gamma_m \to N_{m',u} \to N_{m',u}/N_{m'+1,u}$ is surjective
and separable.
\end{enumerate}
\end{lemma}
\begin{proof}
Let $\varphi$ denote the $p$-power Frobenius endomorphism of $k((t))$.
Write $E$ as an Artin-Schreier extension $k((t))[z]/(z^p - z - x)$
with the $t$-adic valuation of $x$ as large as possible. We then have
$x = at^{-e} + \cdots$ for some nonzero $a \in k$, where $e$ is the ramification number of $E$.
In particular, $e$ is a positive integer not divisible by $p$ (it cannot be $0$ because $k$ has been
assumed to be algebraically closed).

For each $c \in k$, the element $\psi_c \in \Gamma_m$ corresponding to $c$ has the property that
$\psi_c(x)$ defines the same Artin-Schreier extension of $k((t))$ as does $x$,
and so the elements $x$ and $\psi_c(x)$ must generate the same $\FF_p$-subspace of $\coker(\varphi-1, k((t)))$.
Since $x$ and $\psi_c(x)$ both have the form $at^{-e} + \cdots$ and $e$ is not divisible by $p$,
the images of $x$ and $\psi_c(x)$ in $\coker(\varphi-1, k((t)))$ must in fact coincide.

Write $\psi_c(t) = t + \sum_{i=m+1}^\infty P_i(c) t^i$ for certain polynomials
$P_i(T) \in k[T]$. Because $\Gamma_m \to N_m/N_{m+1}$ is separable, $P_{m+1}$ is not a $p$-th power.
Moreover, the map $c \mapsto P_{m+1}(c)$ must be additive in order to come from a group action.

Suppose that $e > m$, and write $x = \sum_{j \geq -e} a_j t^j$ with $a_{-e} = a$. We then have
\[
\psi_c(x)-x \equiv \sum_{j=m-e}^{-1} Q_j(c) t^j  \pmod{k \llbracket t \rrbracket}
\]
for certain polynomials $Q_j(T) \in k[T]$, and in particular $Q_{m-e}(T) = -e a P_{m+1}(T)$.
Since $\psi_c(x) - x \in \coker(\varphi-1, k((t)))$, we must have
\begin{equation} \label{eq:local fields1}
\sum_{i=0}^\infty Q_{(m-e)/p^i}(c)^{p^i} = 0 \qquad (c \in k).
\end{equation}
However, in the sum $\sum_{i=0}^\infty Q_{(m-e)/p^i}(T)^{p^i}$, the
$i=0$ term is not a $p$-th power whereas all of the other terms are.
Consequently, \eqref{eq:local fields1} asserts that a nonzero polynomial over $k$ vanishes at all $c \in k$,
a contradiction. We conclude that $e \leq m$, proving (a).

To prove (b), note that it is sufficient to check the claim for a single $k$-linear homeomorphism 
$E \cong k((u))$. We will check the claim with $u = z^r t^s$ for an arbitrary pair of integers $r,s$ satisfying
$-re + ps = 1$ (which exist because $e$ is not divisible by $p$). To begin with, we have
$z = A u^{-e} + \cdots, t = B u^p + \cdots$ for some $A,B \in k$; using the equalities
\[
u = z^r t^s, \qquad z^p = at^{-e} + \cdots,
\]
we can solve for $A$ and $B$ to obtain
\[
z= a^s u^{-e} + \cdots, \qquad t = a^{-r} u^p + \cdots.
\]
By (a), we have $m - e \geq 0$. For $c \in k$, we thus have
\begin{align*}
(\psi_c(z)-z)^p - (\psi_c(z)-z) &= \psi_c(x)-x \\
&= (\psi_c-1)(at^{-e} + \cdots) \\
&= -ea P_{m+1}(c) t^{m-e} + \cdots \in k \llbracket t \rrbracket.
\end{align*}
In case $e<m$, this implies that $\psi_c(z) = z + e a P_{m+1}(c) t^{m-e} + \cdots$.
Since the $u$-adic valuation of $(\psi_c(z)-z)/z$ is $(m-e)p + e = m'$
while the valuation of $(\psi_c(t)-t)/t$ is the strictly larger value $mp$,
we obtain
\begin{align*}
\psi_c(u) &= \psi_c(z)^r \psi_c(t)^s \\
&= z^r t^s + rea P_{m+1}(c) t^{m-e+s}z^{r-1} + \cdots \\
&= u + re P_{m+1}(c) a^{1-r(m-e)-s} u^{m'+1} + \cdots.
\end{align*}
In case $e=m$, we instead have $\psi_c(z) = z + d + \cdots$ for some $d \in k$
satisfying $d - d^p = e a P_{m+1}(c)$. Computing as before, we obtain
\[
\psi_c(u) = u + rd a^{-s} u^{m'+1} + \cdots.
\]
In both cases, we obtain (b).
\end{proof}

\begin{prop} \label{P:local fields}
Let $E$ be a finite Galois extension of $k((t))$ equipped with an extension of the action of $\Gamma_m$.
Then the ramification breaks of $E/k((t))$ in the upper numbering are all less than or equal to $m$.
\end{prop}
\begin{proof}
Note that $E$ is totally ramified because we assumed that $k$ is algebraically closed.
Also, by replacing $m$ by a multiple, we may reduce to the case where $E$ is totally wildly ramified.
In this case, we induct on the degree of $E$, the case $E = k((t))$ serving as a trivial base case.

Suppose $E \neq k((t))$. Let $e$ be the least ramification break of $E$ in the upper numbering,
and let $F_e$ be the corresponding subfield of $E$.
Since the definition of the ramification filtration is invariant under automorphisms of $k((t))$,
we obtain an action of $\Gamma_m$ on $F_e$. Moreover, $\Gamma_m$ acts on $H = \Gal(F_e/k((t)))$
via a discrete quotient, but the additive group has no
nontrivial discrete quotients. Consequently, if we pick any $\ZZ/p\ZZ$-subextension $F$ of $F_e$,
then $\Gamma_m$ acts on $F$.

By Lemma~\ref{L:local fields2}(a), we have $e \leq m$. 
In addition, if we put $m' = (m-e)p + e$ and choose a homeomorphism $F \cong k((u))$, 
then by Lemma~\ref{L:local fields2}(b), we obtain
a homomorphism $\Gamma_m \to N_{m',u}$ such that the composition $\Gamma_m \to N_{m',u} \to N_{m',u}/N_{m'+1,u}$ is surjective and separable. This last fact allows us to invoke the induction hypothesis,
which implies that the ramification breaks of $E/F$ in the upper numbering
are all less than or equal to $m'$. By Herbrand's rule for
transferring ramification breaks from a group to a subgroup \cite[\S IV.3]{serre-local-fields},
this in turn implies that the ramification breaks of $E/k((t))$ for the upper numbering are 
all less than or equal to $m$, as desired.
\end{proof}

\section{Differential modules over complete fields}

We next recall some definitions and results from \cite{kedlaya-book}
concerning the spectral behavior of differential modules over complete fields. We then make a few additional
calculations leading to a finiteness result concerning the Tannakian automorphism group of a differential module.

\setcounter{theorem}{0}
\begin{convention}
For a matrix over a ring equipped with a norm, we will always interpret the norm of the matrix to be the supremum norm over entries of the matrix.
\end{convention}

\subsection{Differential rings and modules}

We need some general terminology concerning differential rings and modules.

\begin{defn}
By a \emph{differential ring}, we will mean a pair $(R,d)$ in which $R$ is a commutative unital ring
and $d$ is a derivation on $R$. By a \emph{differential module} over $(R,d)$, we will mean a pair $(M,D)$
in which $M$ is a finite projective $R$-module and $D$ is a differential operator on $M$ with respect to $d$.
For example, for each nonnegative integer $n$, $R^{\oplus n}$ may be viewed as a differential operator
by setting $D(r_1,\dots,r_n) = (d(r_1),\dots,d(r_n))$; any differential module isomorphic to one of this form 
is said to be \emph{trivial}. 
We will often omit mention of $d$ and/or $D$ when they may be inferred from context.
\end{defn}

\begin{remark}
Let $M$ be a differential module over a differential ring $R$ which is freely generated by the basis
$\be_1,\dots,\be_n$. Then the action of $D$ on $M$ can be reconstructed from the matrix $N$
defined by $D(\be_j) = \sum_i N_{ij} \be_i$ (the \emph{matrix of action} of $D$ on the basis).
Any other basis $\be'_1,\dots,\be'_n$ is uniquely determined by the invertible matrix
$U$ over $R$ defined by $\be'_j = \sum_i U_{ij} \be_i$ (the \emph{change-of-basis matrix} from
the $\be_i$ to the $\be'_i$); the matrix of action of $D$ on this new basis has the form
$U^{-1} N U + U^{-1} d(U)$.
\end{remark}

\begin{defn}
The differential modules over a given differential ring form a tensor category.
For $M$ a differential module, we write $\End(M)$ as shorthand for $M^\dual \otimes M$;
there is a natural composition morphism $ - \circ - : \End(M) \otimes \End(M) \to \End(M)$.
%For $N$ a differential submodule of $\End(M)$, the \emph{commutator} of $N$ is the differential submodule
%of $\End(M)$ consisting of those $\bv \in \End(M)$ for which $\bv \circ \bw = \bw \circ \bv$ for all
%$\bw \in N$.
\end{defn}

\begin{defn} \label{D:cyclic vector}
Let $(M,D)$ be a differential module of rank $n$ over a differential ring $(R,d)$. A \emph{cyclic vector} for $M$ is an element
$\bv \in M$ such that $\bv, D(\bv), \dots, D^{n-1}(\bv)$ form a basis of $M$ as an $R$-module. 
\end{defn}

\begin{lemma}[Cyclic vector theorem] \label{L:cyclic vector}
Let $(R,d)$ be a differential ring such that $R$ is a field
of characteristic $0$ and $d$ is nonzero. Then every differential module over $R$ admits a cyclic vector.
\end{lemma}
\begin{proof}
See for instance
\cite[Theorem~5.4.2]{kedlaya-book}.
\end{proof}
\begin{cor} \label{C:cyclic vector}
Let $(R,d)$ be a differential ring such that $R$ is a domain
of characteristic $0$ and $d$ is nonzero. Then every differential
module $M$ over $(R,d)$ contains a cyclic vector for $M \otimes_R \Frac(R)$.
\end{cor}

\begin{defn}
For $(M,D)$ a differential module, write $H^0(M)$ and $H^1(M)$ for $\ker(D)$ and $\coker(D)$, respectively.
Note that $H^1(M)$ may be interpreted as a Yoneda extension group.
\end{defn}

\subsection{Differential modules over fields}
\label{sec:fields}

We next review some of the theory of differential modules over completed rational function fields
(also known as \emph{fields of analytic elements}) as presented in \cite[Chapters~9--10]{kedlaya-book}.

\begin{hypothesis} \label{H:differential field}
Throughout \S\ref{sec:fields}, choose $\rho>0$, let $F_\rho$ be the completion of $K(t)$ for the $\rho$-Gauss norm,
and let $E$ be a finite tamely ramified extension of $F_\rho$.
View $F_\rho$ as a differential field for the derivation $d = \frac{d}{dt}$, which extends uniquely
to $E$.
\end{hypothesis}

\begin{defn} \label{D:intrinsic radius}
Let $(V,D)$ be a differential module over $E$.
For $V$ nonzero, let $IR(V)$ denote the \emph{intrinsic radius} of $V$ in the sense of 
\cite[Definition~9.4.7]{kedlaya-book}. That is, $\omega /(\rho IR(V))$ equals the spectral radius of $D$
as a $K$-linear endomorphism of $V$ for any $E$-Banach norm on $V$. The following properties are easily derived
(see \cite[Lemma~6.2.8]{kedlaya-book}).
\begin{enumerate}
\item[(a)]
For any $V$, $IR(V^\dual) = IR(V)$.
\item[(b)]
For any short exact sequence $0 \to V_1 \to V \to V_2 \to 0$, $IR(V) = \min\{IR(V_1), IR(V_2)\}$.
\item[(c)]
For any $V_1, V_2$, $IR(V_1 \otimes V_2) \geq \min\{IR(V_1), IR(V_2)\}$,
with equality if $IR(V_1) \neq IR(V_2)$.
\end{enumerate}

As in \cite[Definition~9.8.1]{kedlaya-book},
the multiset of \emph{intrinsic subsidiary radii} of $V$
is constructed as follows: for each Jordan-H\"older constituent $W$ of $V$, include
$IR(W)$ with multiplicity $\dim_{F_\rho}(W)$.
This multiset is invariant under arbitrary extensions of the constant field and under finite tamely ramified
extensions of $E$ \cite[Proposition~10.6.6]{kedlaya-book},
and its maximum element equals $IR(V)$. 

Let $s_1 \leq \cdots \leq s_n$ be the intrinsic subsidiary radii of $V$.
The \emph{spectral polygon} of $V$, denoted $P(V)$, is then defined to be the convex polygonal curve
starting at $(-n,0)$ and consisting of segments of width 1 and slopes $\log s_1, \dots, \log s_n$ 
in that order.
\end{defn}

\begin{defn}
Let $V$ be a nonzero differential module over $E$.
We say $V$ is \emph{pure} if its intrinsic subsidiary radii are all equal.
We say that $V$ is \emph{refined} if $IR(\End(V)) > IR(V)$; this condition implies that $IR(V) < 1$,
and also that $V$ is pure (using \cite[Lemma~9.3.4]{kedlaya-book}). Consequently, this definition of refinedness agrees with that of \cite[Definition~6.2.12]{kedlaya-book}.

We say that two refined differential modules $V_1, V_2$ over $E$
are \emph{equivalent} if $IR(V_1) = IR(V_2) < IR(V_1^\dual \otimes V_2)$.
As the terminology suggests, this is an equivalence relation \cite[Lemma~6.2.14]{kedlaya-book}.
\end{defn}
\begin{lemma} \label{L:refined comparison}
Let $V_1, V_2$ be nonzero differential modules over $F_\rho$ such that
$IR(V_1), IR(V_2) < IR(V_1^\dual \otimes V_2)$. Then $IR(V_1) = IR(V_2)$ and
\[
IR(\End(V_1)), IR(\End(V_2)) \geq IR(V_1^\dual \otimes V_2);
\]
consequently, $V_1$ and $V_2$ are both refined of the same intrinsic radius.
\end{lemma}
\begin{proof}
The first claim holds because $V_2$ is a direct summand of $V_1 \otimes (V_1^\dual \otimes V_2)$
and $V_1$ is a direct summand of $V_2 \otimes (V_1^\dual \otimes V_2)^\dual$.
The second claim holds because $V_1^\dual \otimes V_1$ is a direct summand of $V_1^\dual \otimes V_1 \otimes V_2^\dual \otimes V_2
\cong (V_1^\dual \otimes V_2)^\dual \otimes (V_1^\dual \otimes V_2)$. 
\end{proof}

\begin{remark} \label{R:intuition}
The intrinsic subsidiary radii of $V$ behave for many purposes like the reciprocal norms of the eigenvalues of some linear transformation associated to $V$. In this model, 
a refined differential module (resp.\ two equivalent refined modules)
over $E$ corresponds to a linear transformation (resp.\ two linear transformations) whose eigenvalues
 all have a single image in the graded ring associated to an algebraic closure of $F_\rho$.
 
For radii in the range $(0, \omega)$ (called the \emph{visible range} in \cite{kedlaya-book}),
this intuition can be made precise using cyclic vectors;
see Proposition~\ref{P:christol-dwork} below.
When $p>0$, one must use pullback and pushforward along Frobenius to access radii in the range $[\omega,1)$,
as described in \cite[Chapter~10]{kedlaya-book}. We will see these techniques in action in
\S\ref{subsec:fields2}.
\end{remark}

\begin{prop}[Christol-Dwork] \label{P:christol-dwork}
Let $V$ be a differential module over $E$ of rank $n$, let $\bv$ be a cyclic vector of $V$,
and write $D^n(\bv) = a_0 \bv + \dots + a_{n-1} D^{n-1}(\bv)$ with $a_0,\dots,a_{n-1} \in E$. Then the multiset of slopes of the spectral polygon of $V$
less than $\log \omega$ consists of $\log \omega - \log \rho + s$ 
for $s$ running over the multiset of slopes of the Newton polygon of the polynomial
$T^n - a_{n-1} T^{n-1} - \cdots - a_0 \in E[T]$ less than $\log \rho$.
\end{prop}
\begin{proof}
See \cite[Corollary~6.5.4]{kedlaya-book}.
\end{proof}

\begin{cor} \label{C:christol-dwork}
For any $s < \omega$ and any positive integers $n_1,n_2,m$, 
there exists $\delta \in (s, \omega)$ for which the following statements hold.
For $i=1,2$, let $V_i$ be a differential module over $E$ of rank $n_i$ which is pure of intrinsic radius $s$.
Let $\bv_i$ be a cyclic vector of $V_i$,
write $D^{n_i}(\bv_i) = a_{0,i} \bv_i + \dots + a_{n_i-1,i} D^{n_i-1}(\bv_i)$ with
$a_{0,i},\dots,a_{n_i-1,i} \in E$,
and define the polynomial $P_i(T) = T^{n_i} - a_{n_i-1,i} T^{n_i-1} - \cdots - a_{0,i} \in E[T]$.
\begin{enumerate}
\item[(a)]
Let $P(T) \in E[T]$ be the monic polynomial of degree $n_1n_2$ with roots $\alpha_2 - \alpha_1$
where $\alpha_i$ runs over the roots of $P_i$.
Then the multiset of slopes of the spectral polygon of $V_1^\dual \otimes V_2$ less than $\log \delta$
consists of $\log \omega - \log \rho + c$ 
for $c$ running over the multiset of slopes of the Newton polygon of $P(T)$ less than $\log \delta - \log \omega + \log \rho$.
\item[(b)]
Let $Q(T) \in E[T]$ be the monic polynomial of degree $n_1^m$ with roots $\alpha_1 + \cdots + \alpha_m$
where $\alpha_i$ runs over the roots of $P_1$.
Then the multiset of slopes of the spectral polygon of $V_1^{\otimes m}$ less than $\log \omega$
consists of $\log \omega - \log \rho + s$ 
for $c$ running over the multiset of slopes of the Newton polygon of $Q(T)$ less than $\log \delta - \log \omega + \log \rho$.
\end{enumerate}
\end{cor}
\begin{proof}
We describe only (a) in detail, as the proof of (b) is similar.
Equip $V_i$ with a norm as in the proof of 
\cite[Theorem~6.5.3]{kedlaya-book}; by enlarging $K$ if necessary, we may ensure that this norm 
is the supremum norm defined by a basis. Equip $V_1^\dual$ with the dual basis,
then equip $V_1^\dual \otimes V_2$ with the product basis and the resulting supremum norm.
The claim then follows by applying \cite[Theorem~6.7.4]{kedlaya-book} .
\end{proof}

\begin{defn}
Let $V$ be a differential module over $E$.
A \emph{spectral decomposition} of $V$ is a direct sum decomposition $V  = \bigoplus_{s \in (0,1]} V_s$ 
such that the intrinsic subsidiary radii of $V_s$ are all equal to $s$.
A \emph{refined decomposition} of $V$ is a direct sum decomposition of $V$ refining a spectral
decomposition in which $V_1$ remains whole, but each $V_s$ with $s<1$ is split into inequivalent refined
summands.
%
%A \emph{potentially refined decomposition} (resp.\ \emph{refined decomposition}) of $V$ is a direct sum decomposition of $V$ refining a spectral
%decomposition, in which each $V_s$ with $s<1$ is split into inequivalent potentially refined
%(resp.\ refined) summands.
\end{defn}

\begin{prop} \label{P:field decomposition1}
Let $V$ be a differential module over $E$.
\begin{enumerate}
\item[(a)]
There exists a unique spectral decomposition of $V$.
\item[(b)]
A refined decomposition of $V$ is unique if it exists. Moreover,
there exists a finite tamely ramified extension $E'$ of $E$
such that $V \otimes_{E} E'$ admits a refined decomposition.
%\item[(c)]
%There exists a unique potentially refined decomposition of $V$.
\end{enumerate}
\end{prop}
\begin{proof}
By restriction of scalars, we may reduce to the case $E = F_\rho$.
In this case, see \cite[Theorem~10.6.2, Theorem~10.6.7]{kedlaya-book}
for (a) and (b), respectively. 
%From (b), (c) follows immediately.
\end{proof}
\begin{cor} \label{C:indecomposable radii}
Let $V$ be a differential module over $E$ such that $IR(V) < 1$.
\begin{enumerate}
\item[(a)] If $V$ is indecomposable, then $V$ is pure. % and potentially refined.
\item[(b)] 
If $V \otimes_{E} E'$ is indecomposable for every finite tamely ramified extension $E'$
of $E$, then $V$ is refined.
\end{enumerate}
\end{cor}

In case $p=0$, one can state an even stronger version of Corollary~\ref{C:indecomposable radii},
closely related to the classical Turrittin-Levelt-Hukuhara decomposition theorem for formal meromorphic connections (see for instance \cite[Chapter~7]{kedlaya-book}).
\begin{prop} \label{P:Turrittin analogue}
Assume $p=0$. Let $V$ be a differential module over $E$ such that $IR(V) < 1$.
%\begin{enumerate}
%\item[(a)] If $V$ is indecomposable, then $V$ is pure and potentially refined.
%\item[(b)] 
If $V \otimes_{E} E'$ is indecomposable for every finite tamely ramified extension $E'$
of $E$, then there exists a differential module $W$ over $E$ of dimension $1$ such that $IR(W^\dual \otimes V) = 1$.
\end{prop}
\begin{proof}
Put $n = \dim_E(V)$. Let $\bv$ be a generator of $\wedge^n V$ and define $f \in E$ by the formula 
$D(\bv) = f\bv$. Let $W$ be the differential module
of dimension 1 over $E$ on the generator $\bw$ for which $D(\bw) = (f/n) \bw$;
then $W^{\otimes n} \cong \wedge^n V$, so $\wedge^n (W^\dual \otimes V)$ is trivial.
If $IR(W^\dual \otimes V) < 1$, then by Corollary~\ref{C:indecomposable radii}, $W^\dual \otimes V$
would be refined; however, \cite[Proposition~6.8.4]{kedlaya-book} would then imply that
$IR(W^\dual \otimes V) = IR(\wedge^n (W^\dual \otimes V)) = 1$, a contradiction.
Hence $IR(W^\dual \otimes V) = 1$ as desired.
\end{proof}

\subsection{More on refined modules}
\label{subsec:fields2}

Proposition~\ref{P:Turrittin analogue} gives a fairly precise description of the indecomposable differential
modules over finite tamely ramified extensions of $F_\rho$ in case $p=0$. We next turn to the situation where $p>0$,
in which case things are more complicated.

\begin{hypothesis}
Throughout \S\ref{subsec:fields2},
retain Hypothesis~\ref{H:differential field},
but assume in addition that $p>0$. Let $\mu_p$ denote the group of $p$-th roots of unity in some algebraic closure
of $K$.
\end{hypothesis}

We recall the basic formalism of Frobenius pullback and pushforward, as in \cite[Chapter~10]{kedlaya-book}.
\begin{defn}
For each $\zeta \in \mu_p$, the $K$-linear substitution $t \mapsto \zeta t$ induces a continuous automorphism $\zeta^*$ of
$E(\mu_p)$. Let $E'$ be the fixed subfield of $E(\mu_p)$ under the group generated by $\Gal(E(\mu_p)/E)$
and the automorphisms $\zeta^*$ for $\zeta \in \mu_p$; we may then view $E'$
as a differential field for the derivation $d = \frac{d}{dt^p}$, and thus define the intrinsic radius
of a nonzero differential module $(V', D')$ over $E'$ so that $\rho^p/(\omega IR(V'))$ equals the spectral radius of $D'$.

For $m=0,\dots,p-1$, let $(W_m,D')$ denote the differential module over $E'$ on the single generator $\bv$
given by $D'(\bv) = (m/p) t^{-p} \bv$. By Proposition~\ref{P:christol-dwork},
$IR(W_m) = \omega^p$ for $m\neq 0$
(see also \cite[Definition~10.3.3]{kedlaya-book}).

For $(V',D')$ a differential module over $E'$, define the differential module $\varphi^* V'$
over $E$ to have underlying module $V' \otimes_{E'} E$ and derivation given by
$D = D' \otimes pt^{p-1}$.

For $(V,D)$ a differential module over $E$, define the differential module $\varphi_* V$
over $E'$ to have underlying module $V$ and derivation given by $D' = p^{-1} t^{1-p} D$.
\end{defn}

\begin{lemma} \label{L:weak pullback}
For any nonzero differential module $V'$ over $E'$,
$IR(\varphi^* V') \geq \min\{IR(V')^{1/p}, p IR(V')\}$.
\end{lemma}
\begin{proof}
See \cite[Lemma~10.3.2]{kedlaya-book}
\end{proof}

\begin{prop} \label{P:antecedent}
Let $V$ be a nonzero differential module over $E$ such that $IR(V) > \omega$.
Then there exists a unique differential module $V'$ over $E'$ 
(called the \emph{Frobenius antecedent} of $V$)
such that
$IR(V') > \omega^p$ and $\varphi^* V' \cong V$; moreover, this module satisfies
$IR(V') = IR(V)^p$.
\end{prop}
\begin{proof}
See \cite[Theorem~10.4.2]{kedlaya-book}.
\end{proof}

\begin{prop} \label{P:descendant}
Let $V$ be a differential module over $E$ with intrinsic subsidiary radii $s_1, \dots, s_n$.
Then the intrinsic subsidiary radii of $\varphi_* V$ (called the \emph{Frobenius descendant} of $V$) comprise the multiset
\[
\bigcup_{i=1}^n \begin{cases} \{s_i^p\} \cup \{ \omega^p \, \mbox{($p-1$ times)}\} & \mbox{if } s_i > \omega \\
\{p^{-1} s_i \, \mbox{($p$ times)}\} & \mbox{if } s_i \leq \omega.
\end{cases}
\]
\end{prop}
\begin{proof}
See \cite[Theorem~10.5.1]{kedlaya-book}.
\end{proof}

\begin{lemma} \label{L:push pull}
\begin{enumerate}
\item[(a)]
For $V$ a differential module over $E$, there are
canonical isomorphisms 
\[
\iota_m: (\varphi_* V) \otimes W_m \cong \varphi_* V
\qquad (m=0,\dots,p-1).
\]
\item[(b)]
For $V$ a differential module over $E$, 
a submodule $U$ of $\varphi_* V$ has the form $\varphi_* X$ for some differential submodule $X$
of $V$ if and only if $\iota_m(U \otimes W_m) = U$
for $m=0,\dots,p-1$.
\item[(c)]
For $V'$ a differential module over $E'$, there is a canonical
isomorphism
\[
\varphi_* \varphi^* V' \cong \bigoplus_{m=0}^{p-1} (V' \otimes W_m).
\]
%\item[(d)]
%For $V_1, V_2$ two differential modules over $E$, $\varphi_*(V_1 \otimes V_2)$ is a direct summand
%of $(\varphi_* V_1) \otimes (\varphi_* V_2)$.
\end{enumerate}
\end{lemma}
\begin{proof}
See \cite[Lemma~10.3.6(a,b,c)]{kedlaya-book}.
%For (d), see \cite[Lemma~10.3.6(f)]{kedlaya-book} plus the errata to \cite{kedlaya-book}.
\end{proof}

\begin{lemma} \label{L:boundary radius}
Let $V'$ be an indecomposable differential module over $E'$ of intrinsic radius $\omega^p$
such that $IR(\varphi^* V') > \omega$.
Then there exists a unique $m \in \{0,\dots,p-1\}$ such that $IR(V' \otimes W_m) > \omega^p$.
\end{lemma}
\begin{proof}
By Proposition~\ref{P:descendant}, at least one of the intrinsic subsidiary radii of $\varphi_* \varphi^* V'$
is greater than $\omega^p$.
By Lemma~\ref{L:push pull}(c), we have $\varphi_* \varphi^* V' \cong \bigoplus_{m=0}^{p-1} (V'
\otimes W_m)$, so for some $m$, at least one of the intrinsic subsidiary radii of $V' \otimes W_m$
is greater than $\omega^p$.
Since $V' \otimes W_m$ is indecomposable, this implies $IR(V' \otimes W_m)  > \omega^p$
by Corollary~\ref{C:indecomposable radii}.
This proves the existence of $m$; uniqueness holds because $IR(W_m) = \omega^p$ for $m \neq 0$.
\end{proof}
\begin{cor} \label{C:boundary radius}
Let $V'$ be a nonzero differential module over $E'$ of intrinsic radius $\omega^p$
such that $IR(\varphi^* V') > \omega$.
Then there exists a unique direct sum decomposition $V' = \bigoplus_{m=0}^{p-1} V'_m$ such that
$IR(V'_m \otimes W_m) > \omega^p$ for $m=0,\dots,p-1$.
\end{cor}

\begin{lemma} \label{L:boundary radius2}
Let $V'_1, V'_2$ be nonzero differential modules over $E'$ of intrinsic radius $\omega^p$
such that $V'_1$ is refined, $V'_2$ is indecomposable, and $IR(\varphi^* ((V'_1)^\dual \otimes V'_2)) > \omega$.
Then there exists a unique $m \in \{0,\dots,p-1\}$ such that $IR((V'_1)^\dual \otimes V'_2 \otimes W_m) > \omega^p$.
\end{lemma}
\begin{proof}
By Corollary~\ref{C:boundary radius}, we have a decomposition $(V'_1)^\dual \otimes V'_2 = \bigoplus_{m=0}^{p-1} X_m$ such that
$IR(X_m \otimes W_m) > \omega^p$ for $m=0,\dots,p-1$. Contracting with $V'_1$ produces an inclusion
$V'_2 \to \bigoplus_{m=0}^{p-1} (V_1' \otimes X_m)$; since $V'_2$ is indecomposable,
we have $V'_2 \subseteq V_1' \otimes X_m$ for some $m$.
Therefore 
\begin{align*}
IR((V'_1)^\dual \otimes V'_2 \otimes W_m) &\geq IR((V'_1)^\dual \otimes V'_1 \otimes X_m \otimes W_m) \\
&\geq \min\{IR((V'_1)^\dual \otimes V'_1), IR(X_m \otimes W_m)\} \\
& > \omega^p.
\end{align*}
Again, $m$ is unique because $IR(W_m) = \omega^p$ for $m \neq 0$.
\end{proof}

\begin{remark} \label{R:refined construction}
Let $V$ be a refined differential module over $E$ 
of intrinsic radius $\omega$ such that $\varphi_* V$ admits a refined decomposition
$\bigoplus_i X_i$.
By Lemma~\ref{L:push pull}(a), there are canonical isomorphisms
$\psi_m: (\varphi_* V) \otimes W_m \cong \varphi_* V$ for $m=0,\dots,p-1$;
we may view these as an action of $\ZZ/p\ZZ$ on $\varphi_* V$, which induces an action on the collection of the
$X_i$. Since $IR(W_m) = \omega^p$ for $m \neq 0$, any two distinct $X_i$ in the same orbit are refined and pairwise inequivalent. By Lemma~\ref{L:push pull}(b), the $X_i$ in a single orbit constitute the pushforward
of a direct summand of $V$.
\end{remark}

\begin{lemma}\label{L:descend refined}
Let $V$ be a refined differential module over $E$ 
of intrinsic radius $s \geq \omega$. Then 
for some finite unramified extension $E_1'$ of $E'$,
there exists a refined differential module $V'$ over $E_1'$
of intrinsic radius $s^p$ such that $\varphi^* V' \cong V \otimes_{E'} E'_1$.
\end{lemma}
\begin{proof}
In case $s > \omega$, we may take $V'$ to be the Frobenius antecedent of $V$
(Proposition~\ref{P:antecedent}); we thus assume $s =\omega$ hereafter.
Suppose first that $V$ is indecomposable.
By Proposition~\ref{P:descendant}, the intrinsic subsidiary radii of $\varphi_* V$
are all equal to $\omega^p$. We may thus apply Proposition~\ref{P:field decomposition1} to produce a finite Galois
tamely ramified extension $E'_1$ of $E'$
such that $\varphi_* V \otimes_{E'} E'_1$ admits a refined decomposition $\bigoplus_i X_i$.

Define an action of $\ZZ/p\ZZ$ on the collection of the $X_i$ as in Remark~\ref{R:refined construction}.
Since we assumed $V$ is indecomposable, it follows that the $X_i$ form a single orbit
under $\ZZ/p\ZZ$.

The group $G = \Gal(E'_1/E')$ also acts on the set of the $X_i$; since $W_m$ is defined over $E'$,
this action defines a homomorphism $G \to \ZZ/p\ZZ$. 
We may replace $E'_1$ with the fixed field of the kernel of this homomorphism;
this field has tame degree over $E'$ dividing $p$ and so must be unramified.

Let $V'$ be any of the $X_i$. By adjunction, the inclusion $V' \to \varphi_* V \otimes_{E'} E'_1$ corresponds to
a map $\varphi^* V' \to V  \otimes_{E'} E'_1$. Pushing forward gives a new map $\varphi_* \varphi^* V' \to \varphi_* V  \otimes_{E'} E'_1$;
using Lemma~\ref{L:push pull} again, we may rewrite the left side as 
$\bigoplus_{m=0}^{p-1} (V' \otimes W_m)$ and match up the actions of $\ZZ/p\ZZ$.
This shows that 
each composition $V' \otimes W_m \to \varphi_* \varphi^* V' \to \varphi_* V  \otimes_{E'} E'_1$ is injective;
since distinct terms $V' \otimes W_m$ cannot have isomorphic submodules (because they are refined and inequivalent),
the map $\varphi_* \varphi^* V' \to \varphi_* V  \otimes_{E'} E'_1$ must be injective. By counting dimensions, this map is also surjective;
hence $\varphi^* V' \to V  \otimes_{E'} E'_1$ is also bijective. This proves the claim in case $V$ is indecomposable.

For general $V$ (still assuming $s = \omega$), we may split $V$ as a direct sum $\bigoplus_{i=0}^r V_i$ of indecomposable summands.
For some $E'_1$, by the previous arguments there exist differential modules $V'_i$ over $E'_1$
which are refined of intrinsic radius $s^p$ such that $\varphi^* V'_i \cong V_i  \otimes_{E'} E'_1$.
By Lemma~\ref{L:boundary radius2}, for each $i$ we can find $m_i \in \{0,\dots,p-1\}$
such that $IR((V'_0)^\dual \otimes V'_i \otimes W_{m_i})> \omega^p$.
We may thus take $V' = \bigoplus_{i=0}^r V'_i \otimes W_{m_i}$.
\end{proof}

\begin{lemma} \label{L:push forward refined decomposition}
Let $V$ be a pure differential module over $E$ 
of intrinsic radius $s \geq \omega$ such that $\varphi_* V$ admits a refined decomposition.
Group summands in this decomposition according to their $\ZZ/p\ZZ$-orbit as per
Remark~\ref{R:refined construction}. Then the resulting decomposition descends to a refined
decomposition of $V$.
\end{lemma}
\begin{proof}
We may use Proposition~\ref{P:antecedent} to check the claim when $s > \omega$, so we may assume
$s = \omega$ hereafter.
The claim may be checked after enlarging $E$, so by Proposition~\ref{P:field decomposition1} we may
ensure that $V$ itself admits a refined decomposition $\bigoplus_i V_i$.
After enlarging $E$ again, by Lemma~\ref{L:descend refined}
we may ensure that each $V_i$ can be written as $\varphi^* V'_i$ for some
refined differential module $V'_i$ over $E'$.
By Lemma~\ref{L:push pull}(c), we then have $\varphi_* V_i \cong \bigoplus_{m=0}^{p-1} (V'_i \otimes W_m)$.
For $i,j$ distinct and $m \in \{0,\dots,p-1\}$,
we cannot have $IR((V'_i)^\dual \otimes V'_j \otimes W_m) > \omega^p$
or else Proposition~\ref{P:descendant} would imply $IR(V_i^\dual \otimes V_j) > \omega$.
It follows that the $V'_i \otimes W_m$ are refined and pairwise inequivalent, so they form the
refined decomposition of $\varphi_* V$. This proves the claim.
\end{proof}

\begin{prop} \label{P:refined power}
Let $V$ be a refined differential module over $E$.
Then $IR(V^{\otimes p}) > IR(V)$.
\end{prop}
\begin{proof}
It is sufficient to prove that for each nonnegative integer $h$, the claim holds when $IR(V) < \omega^{p^{-h}}$.
For $h=0$, this follows by Corollary~\ref{C:christol-dwork}(a,b) with the parameter $m$ in (b) taken
to be $p$. Given this assertion for some $h$,
we may check it for $h+1$ by forming a module $V'$ as in Lemma~\ref{L:descend refined},
applying the known case to deduce that $IR((V')^{\otimes p}) > IR(V') = IR(V)^p$,
then observing that $\varphi^* ((V')^{\otimes p}) = V^{\otimes p}$
and invoking Lemma~\ref{L:weak pullback}
to deduce that $IR(V^{\otimes p}) > IR(V)$.
\end{proof}

When $V$ has dimension 1, we can prove an even stronger assertion.
\begin{lemma} \label{L:hop1}
Let $V$ be a differential module over $E$ of dimension $1$.
Then
\[
\min\{\omega, IR(V^{\otimes p})\} = \min\{\omega, p IR(V)\}.
\]
\end{lemma}
\begin{proof}
This is immediate from Proposition~\ref{P:christol-dwork}.
\end{proof}

\begin{lemma} \label{L:hop2}
Let $V$ be a differential module over $E$ of dimension $1$ such that $\omega^p \leq IR(V) \leq \omega$.
Then $IR(V^{\otimes p}) \geq IR(V)^{1/p}$.
\end{lemma}
\begin{proof}
By enlarging $K$ and rescaling, we may reduce to the case $\rho = 1$.
Put $d = \frac{d}{dt}$ and $s = IR(V)$.
Choose a generator $\bv$ of $V$ and write $D(\bv) = n \bv$ with $n \in E$.
By Proposition~\ref{P:christol-dwork}, $\left| n \right| = \omega/s$.
The differential module $V^{\otimes p}$ is generated by $\bv^{\otimes p}$ and
$D(\bv^{\otimes p}) = pn \bv^{\otimes p}$.
Since $\left|d\right|_E = 1$ and $\left| pn \right| = p^{-1} \omega/s = \omega^p/s \leq 1$, for any $a \in E$ we have $D^p(a \bv^{\otimes p}) = b \bv^{\otimes p}$ 
for some $b \in E$ with $\left| b - d^p(a) \right| \leq p^{-1} (\omega/s) \left| a \right|$.
Since $\left|d^p\right|_E = p^{-1} \leq p^{-1} \omega/s$, we conclude that the operator norm of $D^p$
on $V$ is at most $p^{-1} \omega/s = \omega^p/s$, so the spectral norm of $D$ on $V$ is at most $\omega/s^{1/p}$. This implies the desired inequality.
\end{proof}

\begin{prop} \label{P:rank 1 power}
Let $V$ be a differential module over $E$ of dimension $1$ such that $IR(V) < 1$.
Then $IR(V^{\otimes p}) \geq \min\{IR(V)^{1/p}, p IR(V)\}$.
\end{prop}
\begin{proof}
The claim is trivial if $IR(V) = 1$, so we may assume $IR(V) < 1$.
If $IR(V) \leq \omega^p$, then $\min\{IR(V)^{1/p}, p IR(V)\} = p IR(V)$, 
and in this case the claim follows from Lemma~\ref{L:hop1}.
To complete the proof, it suffices to check the claim in case $\omega^{p^{-h+1}} \leq IR(V) < \omega^{p^{-h}}$ for some nonnegative integer $h$. We prove this by induction on $h$, with the base case $h=0$
following from Lemma~\ref{L:hop2}.
Given the claim for $h-1$, we may deduce the claim for $h$ by forming
$V'$ as in Lemma~\ref{L:descend refined} (after enlarging $E$ if necessary),
applying the induction hypothesis to $V'$,
and then applying Lemma~\ref{L:weak pullback}.
\end{proof}

We are now ready to deduce a finiteness theorem for Tannakian automorphism groups.
\begin{theorem} \label{T:finite levels fields}
Let $V$ be a differential module over $E$.
Let $[V]$ be the Tannakian category of differential modules over $E$ generated by $V$.
Let $\omega$ be the fibre functor on $[V]$ which extracts underlying $E$-vector spaces.
Let $G$ be the automorphism group of $\omega$. 
For $s<1$, let $G^s$ be the subgroup of $G$ acting trivially on $\omega(W)$ for every $W \in [V]$
with $IR(W) > s$. Then $G^s$ is a finite $p$-group.
\end{theorem}
\begin{proof}
Instead of working with differential modules over $E$, we work with the direct limit
of the categories of differential modules over all finite tamely ramified extensions of $E$;
this does not change the groups $G^s$ except for a base extension.
In this larger category, 
we may apply Proposition~\ref{P:finite Tannakian} using Remark~\ref{R:finite Tannakian}:
conditions (i), (ii), (iii) of the remark may be verified using
Proposition~\ref{P:field decomposition1}(b),
Proposition~\ref{P:refined power},
Proposition~\ref{P:rank 1 power}, respectively.
\end{proof}

\begin{remark} \label{R:bad tannakian example}
The group $\bigcup_{s<1} G^s$ need not be finite in general. For example,
if $V$ is free on one generator $\bv$ and $D(\bv) = \lambda t^{-1} \bv$ for some
$\lambda \in K \setminus \QQ_p$, then $IR(V^{\otimes n}) < 1$ for all positive integers $n$
\cite[Example~9.5.2]{kedlaya-book} and so $\bigcup_{s<1} G^s \cong \QQ_p/\ZZ_p$.

In order to obtain finiteness for some class of differential modules, 
one must impose additional hypotheses to ensure that
when $V$ is of dimension 1, there exists a nonnegative integer $m$ for which $IR(V^{\otimes p^m}) = 1$.
For an example of such hypotheses, see Theorem~\ref{T:finite levels}.
\end{remark}

\begin{remark} \label{R:Tannakian easy}
If we assume $p=0$ but otherwise set notation as in Theorem~\ref{T:finite levels fields},
then the group $\bigcup_{s<1} G^s$ becomes a torus, as one may deduce easily from
Proposition~\ref{P:Turrittin analogue}.
\end{remark}

\section{Differential modules over discs and annuli}
\label{sec:rings discs annuli}

We next continue in the vein of \cite{kedlaya-book}, treating differential modules on discs and annuli.
In this section, we maintain continuity with \cite{kedlaya-book}
by phrasing everything in the language of modules over rings of convergent power series.
Starting in \S\ref{sec:Berkovich discs annuli},
we will switch to the language of Berkovich spaces in order to articulate more precise and general results.

\subsection{Rings of convergent power series}

We first introduce the relevant rings of convergent power
series on a disc or annulus, modifying the notation somewhat from that used in \cite[Chapter~8]{kedlaya-book}.

\begin{defn} \label{D:annulus ring}
For $\rho \in [0, +\infty)$, let $\left|\cdot\right|_\rho$ denote the $\rho$-Gauss seminorm on $K[t]$,
defined by the formula $\left| \sum_n c_n t^n \right|_\rho = \max\{\left|c_n\right| \rho^n\}$.
For $I$ a subinterval of $[0, +\infty)$, let
$R_I$ denote the Fr\'echet completion of $K[t]$ (if $0 \in I$) or $K[t,t^{-1}]$ (if $0 \notin I$)
for the seminorms $\left|\cdot\right|_\rho$ for $\rho \in I$.
View $R_I$ as a differential ring for the derivation $\frac{d}{dt}$.
We will occasionally write $R_{I,K}$ instead of  $R_I$ when it is necessary to specify $K$.
\end{defn}

\begin{remark}
Let us briefly recall how the rings $R_I$ appear in the notation of \cite{kedlaya-book}.
\begin{itemize}
\item If $I = [0, \beta]$, then $R_I$ appears as $K \langle t/\beta \rangle$, the ring of analytic functions on
the closed disc $\left|t\right| \leq \beta$.
\item If $I = [0, \beta)$, then $R_I$ appears as $K \{ t/\beta \}$, the ring of analytic functions on
the open disc $\left|t\right| < \beta$.
\item If $I = [\alpha, \beta]$ with $\alpha > 0$, 
then $R_I$ appears as $K \langle \alpha/t, t/\beta \rangle$, the ring of analytic functions on
the closed annulus $\alpha \leq \left|t\right| \leq \beta$.
\item If $I = (\alpha, \beta)$ with $\alpha > 0$, 
then $R_I$ appears as $K \{ \alpha/t, t/\beta \}$, the ring of analytic functions on
the open annulus $\alpha < \left|t\right| < \beta$.
\end{itemize}
\end{remark}

\begin{remark}
Suppose that $I$ is a closed interval. Then $R_I$ is an affinoid algebra for the norm
$\left|\cdot\right|_I = \sup\{\left|\cdot\right|_\rho: \rho \in I\}$. By the log-convexity of $\left|\cdot\right|_\rho$ 
\cite[Proposition~8.2.3]{kedlaya-book} (see also Lemma~\ref{L:approximate using distance}),
one has $\left|\cdot\right|_{[0,\beta]} = \left|\cdot\right|_\beta$ and $\left|\cdot\right|_{[\alpha,\beta]} = \max\{\left|\cdot\right|_\alpha, \left|\cdot\right|_\beta\}$ for $\alpha>0$.
In addition, the ring $R_I$ is a principal ideal domain \cite[Proposition~8.3.2]{kedlaya-book},
so the underlying module of any differential module over $R_I$ is automatically finite free.

Now let $I$ be arbitrary. In this case, $R_I$ is a \emph{Fr\'echet-Stein algebra}
in the sense of \cite[Section~3]{schneider-teitelbaum};
this means that every coherent sheaf on the associated analytic space is generated by its module of global sections.
Moreover, any coherent locally free sheaf of rank $n$ is uniformly finitely generated (because exactly $n$
generators are needed over any closed disc or annulus), and so corresponds to a finite projective module over $R_I$
by \cite[Proposition~2.1.15]{kpx} or \cite[Corollary~2.2.5]{bellovin}.
\end{remark}

\begin{defn}
For $x \in \RR$, let $\langle x \rangle$ denote the distance from $x$ to the nearest integer, that is,
$\langle x \rangle = \min\{x - \lfloor x \rfloor, -x - \lfloor -x \rfloor\}$. We will frequently use the fact that for $m$ a positive integer,
$m \langle x/m \rangle$ is the distance from $x$ to the nearest multiple of $m$.
\end{defn}

\begin{lemma} \label{L:approximate using distance}
Choose $\eta > 1$ and $\alpha, \alpha', \beta, \beta' \in [0, +\infty)$ such that
\[
\alpha' < \beta', \qquad \alpha' = \alpha \eta, \qquad \beta' = \beta/\eta.
\]
Choose a positive integer $m$, an element $h \in \ZZ$, and an element $f \in R_{[\alpha,\beta]}$
whose terms all have exponents
congruent to $h$ modulo $m$.
\begin{enumerate}
\item[(a)]
Put $h' = m \langle h/m \rangle$. Then
\[
\left|f\right|_{[\alpha', \beta']}
\leq
\eta^{-h'} \left|f\right|_{[\alpha, \beta]}.
\]
\item[(b)]
Assume $h=0$. Let $f_0$ be the constant coefficient of $f$. Then
\[
\left|f-f_0\right|_{[\alpha', \beta']} \leq
\eta^{-m} \left|f\right|_{[\alpha, \beta]}.
\]
\end{enumerate}
\end{lemma}
\begin{proof}
Both assertions reduce at once to the case $f = t^n$ for some $n \in h + m \ZZ$, for which the claim is evident.
\end{proof}

We will also need the construction of rings of analytic elements.
\begin{defn}
Let $J$ be the closure of $I$.
Let $R_I^{\an}$ be the Fr\'echet completion of the ring of rational functions in $K(t)$
with no poles in the region $\left|t\right| \in I$ 
for the norms $\left|\cdot \right|_\rho$ for $\rho \in J$. 
This is called the ring of \emph{analytic elements} in the region $\left|t\right| \in I$;
it is a principal ideal domain \cite[Proposition~8.5.2]{kedlaya-book}.
\begin{itemize}
\item
If $I$ is closed, then $R_I^{\an}  = R_I$.
\item
If $I = [0, \beta)$, then $R_I^{\an}$ appears in \cite{kedlaya-book} as $K \llbracket t/\beta \rrbracket_{\an}$.
\item If $I = (\alpha, \beta)$ with $\alpha > 0$, 
then $R_I$ appears in \cite{kedlaya-book} as $K \llbracket \alpha/t, t/\beta \rrbracket_{\an}$.
\end{itemize}
\end{defn}

\subsection{The Robba condition}
\label{sec:Robba}

We now introduce a special class of differential modules over annuli; this class is closely related
to the class of 
\emph{regular} meromorphic differential modules on a Riemann surface.

\begin{hypothesis} \label{H:robba}
Throughout \S\ref{sec:Robba}, 
let $I$ be an open subinterval of $[0, +\infty)$
and let $M$ be a differential module of rank $n$ over $R_I$ for the derivation $t \frac{d}{dt}$.
For $\rho \in I \setminus \{0\}$, put $M_\rho = M \otimes_{R_I} F_\rho$;
for $J$ a closed subinterval of $I$ of positive length, put $M_J = M \otimes_{R_I} R_J$.
\end{hypothesis}

\begin{defn}
We say that $M$ satisfies the \emph{Robba condition} if $IR(M_\rho) = 1$ for all $\rho \in I - \{0\}$.
In this case, we may define an action of the multiplicative group $1 + \gothm_K$ on $M$ by the formula
\[
\lambda(\bv) = \sum_{i=0}^\infty (\lambda - 1)^i \binom{D}{i}(\bv) \qquad
(\lambda \in 1 + \gothm_K, \bv \in M),
\]
since the Taylor series on the right is guaranteed to converge.
(Note that this formula is given incorrectly in \cite[Definition~13.5.2]{kedlaya-book}; it differs from the analogous formula in
\cite[Definition 5.8.1]{kedlaya-book} because the latter is adapted to differential modules for the derivation $\frac{d}{dt}$.)
We may also interpret the action of $\lambda \in 1+\gothm_K$ as an isomorphism $\lambda^*(M) \cong M$,
where $\lambda^*$ is the pullback along the substitution $t \mapsto \lambda t$.
\end{defn}

\begin{example} \label{exa:Robba}
For $\lambda \in K$, let $M_\lambda$ denote the differential module over $R_I$ on a single generator $\bv$ satisfying $D(\bv) = \lambda \,dv$. If $p=0$, then $M_\lambda$ satisfies the Robba condition
whenever $\left|\lambda\right| \leq 1$, and is trivial if and only if $\lambda \in \ZZ$.
By contrast, if $p>0$, then $M_\lambda$ satisfies the Robba condition if and only if
$\lambda \in \ZZ_p$ \cite[Example~9.5.2]{kedlaya-book}, and is again trivial if and only if
$\lambda \in \ZZ$ \cite[Proposition~9.5.3]{kedlaya-book}.
\end{example}

\begin{defn} \label{D:exponent1}
For $A$ a finite multisubset of $\gotho_{K^{\alg}}$,
we say $A$ is \emph{prepared} if no two elements $a_1, a_2$ of $A$ 
have the property that $\left|a_1-a_2-m\right| < 1$ for some nonzero integer $m$.
For $A,B$ two finite multisubsets of $\gotho_{K^{\alg}}$ of the same cardinality $n$,
we say $A$ and $B$ are \emph{equivalent} if there exist orderings $a_1,\dots,a_n$
and $b_1,\dots,b_n$ of $A$ and $B$, respectively, such that $a_i - b_i \in \ZZ$ for $i=1,\dots,n$;
this indeed defines an equivalence relation.
\end{defn}

\begin{defn}
We say $M$ is of \emph{cyclic type} if $\End(M)$ satisfies the Robba condition.
For example, if there exists a differential module $N$ over $R_I$ of positive rank such that
$N^\dual \otimes M$ satisfies the Robba condition, then $M$ is of cyclic type by Lemma~\ref{L:refined comparison}.
Note that the tensor product of modules of cyclic type is again of cyclic type.
\end{defn}

\begin{lemma} \label{L:projector}
Suppose that $M$ is of cyclic type. For each $\lambda \in 1+\gothm_K$,
view the Taylor isomorphism $T_\lambda: \lambda^*(\End(M)) \cong \End(M)$ as a horizontal element of 
\begin{align*}
\lambda^*(M^\dual \otimes M) \otimes (M^\dual \otimes M) &\cong
\lambda^*(M^\dual) \otimes \lambda^*(M) \otimes M^\dual \otimes M \\
&\cong \lambda^*(M) \otimes M^\dual \otimes \lambda^*(M^\dual) \otimes M \\
&\cong (\lambda^*(M^\dual) \otimes M)^\dual \otimes \lambda^*(M^\dual) \otimes M \\
&\cong \End(\lambda^*(M^\dual) \otimes M).
\end{align*}
Then the corresponding endomorphism of $\lambda^*(M^\dual) \otimes M$ is a projector of rank $1$.
\end{lemma}
\begin{proof}
The construction of the Taylor isomorphism on modules satisfying the Robba condition is functorial,
so the diagram
\[
\xymatrix@C3cm{
\lambda^*(\End(M)) \otimes \lambda^*(\End(M)) \ar^{- \circ -}[r] \ar^{T_\lambda \otimes T_\lambda}[d]
& \lambda^*(\End(M)) \ar^{T_\lambda}[d] \\
\End(M) \otimes \End(M) \ar^{- \circ -}[r] & \End(M)
}
\]
commutes. From this, it follows formally that the endomorphism of $\lambda^*(M^\dual) \otimes M$
is a projector. The trace of this projector is an analytic function of $\lambda$, but is also equal
to the rank of the projector and so always belongs to $\{0,\dots,\rank(M)\}$. It is thus a constant function;
moreover, the constant value must equal 1 because for
 $\lambda = 1$, the endomorphism of $\lambda^*(M^\dual) \otimes M \cong \End(M)$ in question
is the projector onto the trace component. This proves the claim.
\end{proof}

\subsection{The Robba condition: residue characteristic $0$}
\label{sec:Robba2a}

We continue to study the Robba condition
in the case of residue characteristic 0. The methods used are familiar, but the
exact result seems to be inexplicably missing from the literature.

\begin{hypothesis}
Throughout \S\ref{sec:Robba2a}, 
retain Hypothesis~\ref{H:robba},
but also assume that $p=0$ and that $M$ satisfies the Robba condition.
\end{hypothesis}

\begin{defn}
An \emph{exponent} for $M$ is a finite multisubset of $\gotho_{K^{\alg}}$ such that
$M[t^{-1}] \otimes_K K^{\alg}$ admits a basis on which $D$ acts via a matrix over $\gotho_{K^{\alg}}$
with multiset of eigenvalues equal to $A$.
\end{defn}

\begin{lemma} \label{L:shearing}
Assume that $0 \in I$ and that the eigenvalues of $D$ on $M/tM$ belong to $\gotho_{K^{\alg}}$.
Then there exists a differential module $M'$ over $R_I$ with $M[t^{-1}] \cong M'[t^{-1}]$ such that 
the eigenvalues of $D$ on $M'/tM'$ belong to $\gotho_{K^{\alg}}$ and are prepared.
\end{lemma}
\begin{proof}
This is an example of the use of \textit{shearing transformations} \cite[Proposition~7.3.10]{kedlaya-book}.
Split $M/tM$ as a direct sum in which each summand consists of the generalized eigenspaces for a single
Galois orbit of eigenvalues for the action of $D$. If we consider the differential submodule $M'$ of $M$
consisting of those elements whose images in $M/tM$ project to zero in a particular summand, the eigenvalues of $D$
on $M'/tM'$ are the same as on $M/tM$ except that one Galois orbit has been shifted by 1.

It thus suffices to establish the existence of a sequence of shifts having the desired property.
This follows from the following two observations (both of which require $p=0$).
\begin{enumerate}
\item[(a)]
If $\lambda_1, \lambda_1' \in \kappa_K^{\alg}$ are Galois conjugate,
$\lambda_2, \lambda'_2 \in \kappa_K^{\alg}$ are Galois conjugate,
and $\lambda_1 - \lambda_2, \lambda'_1 - \lambda'_2 \in \ZZ$, then
$\lambda_1 - \lambda_2 =  \lambda'_1 - \lambda'_2$. 
(This follows by taking traces from some finite extension of $\kappa_K$ containing $\lambda_1, \lambda_1', \lambda_2, \lambda_2'$.)
\item[(a)]
If $\lambda, \lambda' \in \kappa_K^{\alg}$ are Galois conjugate and differ by an integer, then they are equal.
(This follows from (a) by taking $\lambda_1 = \lambda'_1 = \lambda_2 = \lambda$, $\lambda'_2 = \lambda'$.)
\end{enumerate}
\end{proof}

\begin{lemma} \label{L:sheared solve}
Assume that $0 \in I$ and the eigenvalues of $D$ on $M/tM$ belong to $\gotho_{K^{\alg}}$ and are prepared.
Then there exists a basis of $M$ on which $D$ acts via a matrix over $\gotho_K$.
\end{lemma}
\begin{proof}
Let $P(T) \in \gotho_K[T]$ be the characteristic polynomial of the action of $D$ on $M/tM$. 
Since the roots of $P$ are prepared,
for each positive integer $j$ there exists a unique polynomial $Q_j(T) \in \gotho_K[T]$ of degree at most $n-1$
such that $P(T-j) Q_j(T) \equiv 1 \pmod{P(T)}$.

It is straightforward to check (see for example \cite[Proposition~7.3.6]{kedlaya-book}) that
there exists a basis of $M \otimes_{R_I} K \llbracket t \rrbracket$ on which $D$ acts via a matrix over
$\gotho_K$. We may reconstruct this basis by starting with any elements $\be_1,\dots,\be_n \in M$
which lift a basis of $M/tM$ and forming the $t$-adic limits of the sequences
\[
\be_{i,m} = \left(\prod_{j=1}^m P(D-j) Q_j(D) \right)\be_i \qquad (i=1,\dots,n; m=1,2,\dots).
\]
For any given $\rho \in I - \{0\}$, these sequences are bounded for the norm induced by $|\cdot|_\rho$
using a basis of $M_{[0,\rho]}$ (because $M$ satisfies the Robba condition);
since these sequences also converge $t$-adically, they converge under
$|\cdot|_{\rho'}$ for any $\rho' \in (0, \rho)$
by Lemma~\ref{L:approximate using distance}(b) (with $m=1$).
This proves the existence of the desired basis.
\end{proof}

\begin{lemma} \label{L:char 0 solve}
Assume that for some $\rho \in I$, $M$ admits a 
basis $\be_1,\dots,\be_n$ on which $D$ acts via a matrix $N = \sum_{i \in \ZZ} N_i t^i$
with $|N_0| \leq 1$ and $|N-N_0|_{\rho} < 1$ for all $\rho \in I$.
Then there exists a basis of $M$ on which $D$ acts via a matrix over $\gotho_K$.
\end{lemma}
\begin{proof}
By applying Lemma~\ref{L:shearing} with $K$ replaced by $\kappa_K$ (equipped with the trivial norm)
and using the fact that $\kappa_K[t^{\pm}]$ is a principal ideal domain (so every
invertible square matrix over it factors as a product of elementary matrices),
we may ensure that the eigenvalues of $N_0$ are prepared.
In this case, for any nonzero $i \in \ZZ$, the eigenvalues of the linear operator 
$X \mapsto \overline{N_0} X - X \overline{N_0} + iX$ on $n \times n$ matrices over $\kappa_K$ are all nonzero
(because each of them has the form $\lambda - \lambda' + i$
for some eigenvalues $\lambda, \lambda'$ of $\overline{N_0}$). Consequently,
this linear operator is invertible; it follows that for any $n \times n$ matrix $X$ over $K$ and any nonzero $i \in \ZZ$,
$\left|N_0 X - X N_0 + iX\right| = \left|X\right|$.

We next produce a sequence $U_0,U_1,\dots$ of invertible matrices over $R_I$ such that
$\left|U_l-I_n\right|_\rho < 1$ for all $l \in \{0,1,\dots\}$ and $\rho \in I$. Start with $U_0 = I_n$. Given 
$U_l$ for some $l$, put $N_l = U_l^{-1} N U_l + U_l^{-1} D(U_l)$. Write $N_l = \sum_{i \in \ZZ} N_{l,i} t^i$
and apply the previous paragraph to construct $X_l$ so that $\left|X_l\right|_\rho = \left|N_l - N_{l,0}\right|_\rho$ for all $\rho \in I$
and $N_{l,0} - N_{l} = X_l N_0 - N_0 X_l + D(X_l)$.
Then put $V_l = I_n + X_l$ and $U_{l+1} = U_l V_l$; note that $N_{l+1} = V_l^{-1} N_l V_l + V_l^{-1} D(V_l)$.

Suppose that for $\rho \in I$ and $\epsilon > 0$, we have
$\left|N-N_0\right|_\rho \leq \epsilon$ and $\left|N_l-N_{l,0}\right|_\rho \leq \epsilon^{l+1}$.
We then have
$\left|V_l-I_n\right|_\rho \leq \epsilon^{l+1}$, so
\[
\left|N_{l+1} - N_l + X_l N_0 - N_0 X_l - D(X_l)\right|_\rho \leq \epsilon^{l+2}.
\]
However, the matrix on the left side is exactly $N_{l+1} - N_{l,0}$,
so we must have $\left|N_{l+1} - N_{l+1,0}\right|_\rho \leq \epsilon^{l+2}$.

{}From the previous paragraph, it follows that the $U_l$ converge to an invertible matrix $U$ over $R_I$.
The elements $\be'_1,\dots,\be'_n$ of $M$ given by $\be'_j = \sum_i U_{ij} \be_i$ then form a basis with the desired
property.
\end{proof}

\begin{theorem} \label{T:CM char 0}
Assume that $p=0$ and that $M$ satisfies the Robba condition.
\begin{enumerate}
\item[(a)]
There exists a Galois-invariant exponent for $M$.
\item[(b)]
Any two exponents of $M$ are equivalent.
\end{enumerate}
\end{theorem}
\begin{proof}
Apply Corollary~\ref{C:cyclic vector}
to choose $\bv \in M$ which is a cyclic vector for $M \otimes_{R_I} \Frac(R_I)$.
For any closed subinterval $J$ of $I$ of positive length, the quotient of $M_J$
by the span of $\bv, D(\bv), \dots, D^{n-1}(\bv)$ is killed by some nonzero element of $R_J$; 
since the slopes of the Newton polygon of this element form a discrete subset of $J$, we can shrink $J$
so as to force this element to become a unit. That is, we may choose $J$ so that $\bv, D(\bv), \dots, D^{n-1}(\bv)$ form a basis of $M_J$. 

Let $N$ be the matrix of action of $D$ on the basis $\bv, D(\bv), \dots, D^{n-1}(\bv)$ of $M_J$.
By Proposition~\ref{P:christol-dwork}, we have
$\left|N\right|_J \leq 1$. In particular, if we write $N = \sum_{i \in \ZZ} N_i t^i$, then $\left|N_0\right| \leq 1$. Since $J$ has positive
length and $\left|N-N_0\right|_J \leq 1$, by shrinking $J$ and applying Lemma~\ref{L:approximate using distance}(b)
(with $m=1$)
we may ensure that $\left|N-N_0\right|_J < 1$. We may then apply
Lemma~\ref{L:char 0 solve} to obtain the conclusion of (a) for $M_J$. We may then use
Lemma~\ref{L:shearing} and Lemma~\ref{L:sheared solve} to extend the convergence from $J$ to $I$.
This yields (a). 
Given (a), (b) follows from the fact that $M_\lambda$ is trivial if and only if $\lambda \in \ZZ$.
\end{proof}

\subsection{The Robba condition: residue characteristic $p>0$}
\label{sec:Robba2}

When $p>0$, the structure of modules satisfying the Robba condition is more complicated; it is best understood
using the Christol-Mebkhout theory of $p$-adic exponents. Here we follow and refine
the exposition in \cite[Chapter~13]{kedlaya-book}.

\begin{hypothesis}
Throughout \S\ref{sec:Robba2},
retain Hypothesis~\ref{H:robba},
but also assume that $p>0$ and that $M$ satisfies the Robba condition. 
\end{hypothesis}

\begin{defn}
We say $a \in \ZZ_p$ is a \emph{$p$-adic Liouville number} if $a \notin \ZZ$ and
\begin{equation} \label{eq:p-adic Liouville number}
\liminf_{m \to \infty} \frac{p^m}{m} \left\langle \frac{a}{p^m} \right\rangle < +\infty.
\end{equation}
Otherwise, we say $a$ is a \emph{$p$-adic non-Liouville number}.

For $A$ a multisubset of $\ZZ_p$, we say that $A$ is \emph{$p$-adic non-Liouville} if it contains
no $p$-adic non-Liouville number. We say that $A$ has \emph{$p$-adic non-Liouville differences}
if the \emph{difference multiset}
of $A$, defined as
\[
A-A = \{a_1 -a_2: a_1, a_2 \in A\},
\]
is $p$-adic non-Liouville.
\end{defn}

\begin{defn}
Let $A = \{a_1,\dots,a_n\}$ and $B = \{b_1,\dots,b_n\}$ be two finite multisubsets of $\ZZ_p$ of the same cardinality $n$. We say that $A$ and $B$ are \emph{weakly equivalent} if there exist a constant $c>0$
and a sequence $\sigma_1, \sigma_2, \dots$ of permutations of $\{1,\dots,n\}$
such that
\[
p^m \left\langle \frac{a_{\sigma_m(i)} - b_{i}}{p^m} \right\rangle  \leq cm \qquad (m=1,2,\dots; \,i =1,\dots,n).
\]
This is evidently an equivalence relation.
Note that $A,B$ are weakly equivalent if they are equivalent 
in the sense of Definition~\ref{D:exponent1}; the converse is false in general (see \cite[Example~13.4.6]{kedlaya-book})
but is true for $n=1$ (see Corollary~\ref{C:singleton weakly equivalent} below).
\end{defn}

All of the key properties of weak equivalence can be expressed in terms of the following construction.
\begin{defn}
Let $A,A_1,\dots,A_k$ be multisubsets  of $\ZZ_p$
such that $A$ is the multiset union of $A_1, \dots, A_k$. We say that $A_1,\dots,A_k$ form an \emph{integer partition}
(resp.\ a \emph{Liouville partition})
of $A$ if there do not exist distinct values $g,h \in \{1,\dots,k\}$ and elements
$a_g \in A_g, a_h \in A_h$ such that $a_g - a_h$ is an integer (resp.\ an integer or a $p$-adic Liouville number).
This implies in particular that $A_g$ and $A_h$ are disjoint, so $A_1,\dots,A_k$ 
form a partition of $A$.

Note that $A$ always admits a maximal integer partition, namely the partition into $\ZZ$-cosets.
This partition is a Liouville partition if and only if $A$ has $p$-adic non-Liouville differences.
\end{defn}

\begin{prop} \label{P:weakly equivalent}
Let $A$ be a finite multisubset of $\ZZ_p$ and let $A_1,\dots,A_k$ be a Liouville partition of $A$.
\begin{enumerate}
\item[(a)]
Let $B_1, \dots, B_k$ be multisubsets of $\ZZ_p$ such that $B_g$ is weakly equivalent to $A_g$
for $g=1,\dots,k$. Then $B_1,\dots,B_k$ form a Liou\-ville partition of $B = B_1 \cup \cdots \cup B_k$;
in particular, $B_1,\dots,B_k$ are pairwise disjoint.
\item[(b)]
Suppose $B$ is a multisubset of $\ZZ_p$ weakly equivalent to $A$.
Then $B$ admits a Liouville partition $B_1,\dots,B_k$ such that $B_g$ is weakly equivalent to $A_g$
for $g=1,\dots,k$.
\end{enumerate}
\end{prop}
\begin{proof}
By the conditions on $A$, for each $c>0$, there exists $m_0 = m_0(c)$ such that 
for all $m \geq m_0$, $g, h \in \{1,\dots,k\}$ with $g \neq h$, $a_g \in A_g$, $a_h \in A_h$,
\begin{equation}  \label{eq:weakly equivalent}
p^m \left\langle \frac{a_g - a_h}{p^m} \right\rangle > (3c+1)m.
\end{equation}
Assume now the hypotheses of (a). Suppose by way of contradiction that there exist 
$g, h \in \{1,\dots,k\}$ with $g \neq h$, $b_g \in B_g$, $b_h \in B_h$ such that
$b_g - b_h$ is an integer or a $p$-adic Liouville number.
Then there exists $c>0$ such that for each $m$, on one hand
\[
p^m \left\langle \frac{b_g - b_h}{p^m} \right \rangle \leq cm
\]
and on the other hand there exist $a_g \in A_g$, $a_h \in A_h$ such that
\[
p^m \left\langle \frac{a_g - b_g}{p^m} \right \rangle,
p^m \left\langle \frac{a_h - b_h}{p^m} \right \rangle \leq cm.
\]
But then
\[
p^m \left\langle \frac{a_g - a_h}{p^m} \right \rangle \leq 3cm,
\]
which combined with \eqref{eq:weakly equivalent} yields the desired contradiction.

Assume now the hypotheses of (b). Label the elements of $A$ and $B$ as $a_1,\dots,a_n$
and $b_1,\dots,b_n$, respectively.
Then there exists $c>0$ such that for each $m$, there exists a permutation 
$\sigma_m$ of $\{1,\dots,n\}$ such that
\[
p^m \left\langle \frac{a_{\sigma_m(i)} - b_{i}}{p^m} \right\rangle  \leq cm \qquad (i =1,\dots,n).
\]
In particular,
\[
p^m \left\langle \frac{a_{\sigma_m(i)} - a_{\sigma_{m+1}(i)}}{p^m} \right\rangle  \leq (2c+1)m \qquad (i =1,\dots,n),
\]
which by \eqref{eq:weakly equivalent}
yields that for $m \geq m_0(c)$, $\sigma_m^{-1} \circ \sigma_{m+1}$ must respect the partition of $A$.
Define $B_1,\dots,B_k$ so that $B_g$ consists of those $b_i$ for which $a_{\sigma_{m)}(i)} \in A_g$
for $m \geq m_0(c)$;
by the above argument, $B_g$ is weakly equivalent to $A_g$.
By (a), $B_1,\dots,B_k$ is a Liouville partition of $B$, as desired.
\end{proof}

\begin{cor} \label{C:transmit integer}
Let $A,B$ be two finite multisubsets of $\ZZ_p$ which are weakly equivalent.
Then $A$ contains an integer or a $p$-adic non-Liouville number if and only if $B$ does.
\end{cor}
\begin{proof}
Note that $A$ contains an integer or $p$-adic non-Liouville number if and only if
$\{0\}$ and $A$ fail to form a Liouville partition of $\{0\} \cup A$. 
The claim thus follows by applying Proposition~\ref{P:weakly equivalent}(a)
to $\{0\} \cup A$ and $\{0\} \cup B$.
\end{proof}

The following corollary reproduces \cite[Lemma~13.4.3]{kedlaya-book}.
\begin{cor} \label{C:singleton weakly equivalent}
For $a,b \in \ZZ_p$, the singleton multisets $\{a\}, \{b\}$ are weakly equivalent if and only if 
$a- b \in \ZZ$.
\end{cor}
\begin{proof}
By translating both $a$ and $b$, we may assume $b=0$. If $a \in \ZZ$, then $\{a\}$ and $\{0\}$ are equivalent and hence weakly equivalent.
Conversely, if $\{a\}$ and $\{0\}$ are weakly equivalent, then 
$a$ satisfies \eqref{eq:p-adic Liouville number} and so must be either an integer or a $p$-adic Liouville
number, but the latter case is ruled out by Corollary~\ref{C:transmit integer}.
\end{proof}

The following corollary reproduces \cite[Proposition~13.4.5]{kedlaya-book}.
\begin{cor} \label{C:weakly equivalent}
Let $A,B$ be two finite multisubsets of $\ZZ_p$ which are weakly equivalent.
Suppose that $A$ has $p$-adic non-Liouville differences.
Then $A$ and $B$ are equivalent.
\end{cor}
\begin{proof}
By partition $A$ into $\ZZ$-cosets and applying Proposition~\ref{P:weakly equivalent}(b),
we may reduce to the case where $A$ is a multisubset of $\ZZ$. In this case, for each $b \in B$, the 
singleton multisets $\{0\}$ and $\{b\}$ are weakly equivalent,
so Corollary~\ref{C:singleton weakly equivalent} implies that $b \in \ZZ$. This proves the claim.
\end{proof}

\begin{cor} \label{C:weakly equivalent2}
Let $A,B$ be two finite multisubsets of $\ZZ_p$ which are weakly equivalent.
Suppose that $A$ is $p$-adic non-Liouville. Then there exist Liouville partitions $A_1, A_2$ of $A$ and $B_1, B_2$ of $B$ satisfying
the following conditions.
\begin{enumerate}
\item[(a)]
The multisets $A_1, B_1$ consist entirely of integers.
\item[(b)]
The multisets $A_2, B_2$ are weakly equivalent and contain no integers or $p$-adic Liouville numbers.
\end{enumerate}
In particular, $B$ is also $p$-adic non-Liouville.
\end{cor}
\begin{proof}
Partition $A$ into  two parts $A_1, A_2$ so that $A_1$ consists precisely of the integers appearing in $A$;
by hypothesis, this is a Liouville partition of $A$.
By Proposition~\ref{P:weakly equivalent}(b), $B$ admits a Liouville partition $B_1, B_2$ in which
$B_i$ is weakly equivalent to $A_i$ for $i=1,2$. Since $A_1$ consists only of integers, by 
Corollary~\ref{C:weakly equivalent}, $B_1$ also consists only of integers.
Since $A_2$ does not contain any integer or $p$-adic Liouville number, neither does $B_2$ by Corollary~\ref{C:transmit integer}.
This proves the desired results.
\end{proof}

\begin{cor} \label{C:weakly equivalent3}
Let $A$ be a finite multisubset of $\ZZ_p$ such that $A-A$ is weakly equivalent to a
$p$-adic non-Liouville multiset. Then $A$ has $p$-adic non-Liouville differences.
\end{cor}
\begin{proof}
By Corollary~\ref{C:weakly equivalent2}, $A-A$ is $p$-adic non-Liouville.
\end{proof}

\begin{defn} \label{D:exponent}
Recall that we are assuming that $M$ satisfies the Robba condition.
Let $J$ be a closed subinterval of $I$ of positive length.
We say that the multisubset $A = \{a_1,\dots,a_n\}$ of $\ZZ_p$ is an \emph{exponent}
for $M$ over $J$ if there exist elements $\bv_{m,A,j} \in M_J[t^{-1}]$ for $m=1,2,\dots$ and $j=1,\dots,n$ satisfying the following
conditions. (For this definition, we fix an ordering of $A$, but this choice is manifestly immaterial.)
\begin{enumerate}
\item[(a)]
For all $m,j$, for all $\zeta \in K^{\alg}$ with $\zeta^{p^m} = 1$, 
we have $\zeta^* (\bv_{m,A,j}) = \zeta^{a_j} \bv_{m,A,j}$
as an equality in $M_J[t^{-1}] \otimes_{K} 
K(\zeta)$.
\item[(b)]
For some (and hence any) basis $\be_1,\dots,\be_n$ of $M_J$,
there exists $k>0$ such that the $n \times n$ matrices $S_{m,A}$ over $R_J$ 
defined by $\bv_{m,A,j} = \sum_i (S_{m,A})_{ij} \be_i$ are invertible and satisfy
\[
\left|S_{m,A}\right|_J, \left|S_{m,A}^{-1}\right|_J
\leq p^{km}\qquad (m=1,2,\dots).
\]
\end{enumerate}
Note that if $A$ is an exponent for $M$ over $J$, then so is any multiset equivalent to $A$
(but not necessarily any multiset weakly equivalent to $A$).
\end{defn}

\begin{remark} \label{R:change hypothesis}
In \cite[Definition~13.5.2]{kedlaya-book}, the hypotheses on the matrix $S_{m,A}$ are slightly different:
it is assumed that $S_{m,A}$ is invertible and satisfies
$\left|S_{m,A}\right|_J \leq p^{km}$ and $\left|\det(S_{m,A})\right|_J \geq 1$.
It is easy to see that this hypothesis is equivalent to the one given in
Definition~\ref{D:exponent} modulo
rescaling the vectors $\bv_{m,A,j}$ and rechoosing the constant $k$; we may thus safely quote results from
\cite{kedlaya-book} in what follows.
\end{remark}

\begin{example} \label{exa:rank 1 Robba}
As noted in Example~\ref{exa:Robba},
for any $\lambda \in \ZZ_p$, the differential module $M_\lambda$
generated by a single element $\bv$ satisfying $D(\bv) =\lambda \bv$ satisfies the Robba condition
\cite[Example~9.5.2]{kedlaya-book}. This module admits the singleton multiset $\{\lambda\}$ as an exponent.
\end{example}

\begin{remark} \label{R:exponent functoriality}
If $M_1, M_2$ are two differential modules over $R_I$ for the derivation $t \frac{d}{dt}$
admitting respective exponents $A_1, A_2$ over some $J$, we then have the following.
\begin{enumerate}
\item[(a)] If there exists an exact sequence $0 \to M_1 \to M \to M_2 \to 0$ of differential modules over $R_I$, then $M$ admits the multiset union $A_1 \cup A_2$ as an exponent over $J$.
\item[(b)] The differential module $M_1 \otimes M_2$ admits the multiset$ A_1 + A_2 = \{a_i + a_j: a_i \in A_1, a_j \in A_2\}$
as an exponent over $J$.
\item[(c)] The differential module $M_1^\dual$ admits the multiset $-A_1 = \{-a: a \in A_1\}$ as an exponent over $J$.
\end{enumerate}
\end{remark}

\begin{remark} \label{R:contains 0}
If $0 \in I$, then it is straightforward to check the following by
imitating the proof of \cite[Theorem~13.2.2]{kedlaya-book}.
\begin{enumerate}
\item[(a)]
Let $A$ be the set of eigenvalues of $D$ on $M/tM$. Then $A$ belongs to $\ZZ_p^n$ and is an exponent for $M$.
\item[(b)]
Any Liouville partition of $A$ corresponds to a unique direct sum decomposition of $M$.
\item[(c)]
If $A$ is a multisubset of $\lambda + \ZZ$, then there exists another differential module $M'$ over $R_I$ with $M[t^{-1}] \cong M'[t^{-1}]$
such that $D$ acts on $M'/tM'$ via a matrix with all eigenvalues equal to $\lambda$.
(This again follows from the use of shearing transformations as in Lemma~\ref{L:shearing}.)
\item[(d)]
If $A$ is a multisubset of $\{\lambda\}$,
then there exists a basis of $M$
on which $D$ acts via a matrix over $K$ with all eigenvalues equal to $\lambda$.
\end{enumerate}
We may thus safely assume $0 \notin I$ in what follows.
\end{remark}

\begin{theorem} \label{T:exponents}
\begin{enumerate}
\item[(a)]
For any closed subinterval $J$ of $I$ of positive length not containing $0$,
there exists an exponent for $M$ over $J$.
\item[(b)]
Any two exponents for $M$ (possibly over different intervals) are weakly equivalent.
\end{enumerate}
\end{theorem}
\begin{proof}
For (a), see \cite[Theorem~13.5.5]{kedlaya-book}.
For (b), let $J_1, J_2$ be two closed subintervals of $I$ of positive length not containing $0$.
Let $A_1, A_2$ be exponents for $M$ over $J_1, J_2$. If $J_1 = J_2$, we may apply \cite[Theorem~13.5.6]{kedlaya-book}
to deduce that $A_1$ is weakly equivalent to $A_2$. Otherwise, 
let $J$ be a third such interval containing both $J_1$ and $J_2$.
By (a), there exists an exponent $A$ for $M$ over $J$, which then restricts to an exponent for $M$
over $J_1$ and over $J_2$. By \cite[Theorem~13.5.6]{kedlaya-book} again,
$A$ is weakly equivalent to both $A_1$ and $A_2$, so $A_1$ and $A_2$ are weakly equivalent to each other.
\end{proof}

\begin{remark}
In case $M$ admits a basis (e.g., if $K$ is spherically complete), the proofs of \cite[Theorem~13.5.5, Theorem~13.5.6]{kedlaya-book} show that the exponent $A$ and the elements $\bv_{m,A,j}$ can be chosen uniformly in $J$. We will not need to use this fact in this paper.
\end{remark}

\begin{defn} \label{D:non-Liouville exponents}
We say that $M$ has \emph{$p$-adic non-Liouville exponents} if for some 
closed subinterval $J$ of $I$ of positive length not containing $0$, $M$ admits an exponent $A$ over $J$
which is $p$-adic non-Liouville. By Theorem~\ref{T:exponents}(b) and Corollary~\ref{C:weakly equivalent2},
this implies that every exponent of $M$ (over every $J$) is $p$-adic non-Liouville.

We say that $M$ has \emph{$p$-adic non-Liouville exponent differences} if $\End(M)$ has $p$-adic non-Liouville
exponents. For alternate characterizations, see Lemma~\ref{L:exponent differences} below.
\end{defn}

\begin{lemma} \label{L:exponent differences}
The following conditions are equivalent.
\begin{enumerate}
\item[(a)]
The module $M$ has $p$-adic non-Liouville exponent differences.
\item[(b)]
Some exponent of $M$ has $p$-adic non-Liouville differences.
\item[(c)]
Every exponent of $M$ has $p$-adic non-Liouville differences.
\end{enumerate}
Moreover, when these conditions hold, then any two exponents of $M$ are equivalent (not just weakly equivalent).
\end{lemma}
\begin{proof}
By Theorem~\ref{T:exponents}(a), (c) implies (b). 
By Remark~\ref{R:exponent functoriality},
(b) implies (a).

Suppose now that (a) holds. Let $A$ be any exponent for $M$, and let $B$ be an exponent for $\End(M)$
which is $p$-adic non-Liouville.
By Remark~\ref{R:exponent functoriality}, $A - A$ is an exponent for $\End(M)$, so
by Theorem~\ref{T:exponents}(b), $A-A$ and $B$ are weakly equivalent.
By Corollary~\ref{C:weakly equivalent3}, $A$ has $p$-adic non-Liouville differences, yielding (c).
Moreover, if $A'$ is another exponent for $M$, then $A$ and $A'$ are weakly equivalent by
Theorem~\ref{T:exponents}(b), so $A$ and $A'$ are equivalent by Corollary~\ref{C:weakly equivalent}.
\end{proof}

The primary structure theorem for differential modules satisfying the Robba condition is the decomposition
theorem of Christol-Mebkhout; see for instance \cite[Theorem~13.6.1]{kedlaya-book} and the errata to \cite{kedlaya-book}.
Here, we divide the statement into two parts in order to clarify the exposition and strengthen one of the two parts.
One of the two parts, which by itself is sufficient for many applications,
is the following structure theorem for modules admitting a singleton exponent.
\begin{theorem} \label{T:CM part 1}
Suppose that $M$ admits an exponent identically equal to some $\lambda \in \ZZ_p$.
Then for any closed subinterval $J$ of $I$ of positive length,
$M_J$ admits a basis on which $D$ acts via
a matrix over $K$ whose eigenvalues are all equal to $\lambda$.
\end{theorem}
\begin{proof}
We may assume $0 \notin I$ thanks to Remark~\ref{R:contains 0}.
By replacing $M$ with its twist $M_\lambda^\dual \otimes M$, we may
reduce the theorem to the special case $\lambda=0$. 
Let $J$ be any closed subinterval of $I$ of positive length;
by Lemma~\ref{L:exponent differences}, the zero $n$-tuple is an exponent for $M$
over $J$. Choose $\eta > 1$ and $\alpha, \alpha', \beta, \beta' \in I$ such that
$\alpha' < \beta'$, $\alpha' = \alpha \eta$, $\beta' = \beta/\eta$,
and $J \subseteq [\alpha', \beta']$.
Fix a basis of $M_J$ and define the matrices $S_{m,A}$ as in Definition~\ref{D:exponent}
for $A = \{0,\dots,0\}$.
Choose $\lambda \in (0,1)$ and $c>0$ so that $p^{10k} \eta^{-c} \leq \lambda$,
then choose $m_0 >0$ so that $p^{m} > cm$ for all $m \geq m_0$.
We will construct invertible matrices $R_m$ over $K$ for $m \geq m_0$ such that $R_{m_0} = I_n$
and
\[
\left|I_n - R_m S_{m,A}^{-1} S_{m+1,A} R_{m+1}^{-1}\right|_\rho \leq \lambda^m \qquad
(\rho \in [\alpha', \beta'], m \geq m_0).
\]
This will imply that for $m \geq m_0$ and $\rho \in [\alpha', \beta']$,
$\left|I_n - S_{m_0,A}^{-1} S_{m,A} R_m^{-1}\right|_\rho < 1$
and $\left|S_{m_0,A}^{-1} S_{m,A} R_m^{-1} - S_{m_0,A}^{-1} S_{m+1,A} R_{m+1}^{-1}\right|_\rho \leq \lambda^m$.
Consequently, the sequence $S_{m_0,A}^{-1} S_{m,A} R_m^{-1}$ for $m = m_0, m_0+1, \dots$ will converge to an invertible matrix $U$
over $R_{[\alpha',\beta']}$
such that $S_{m_0,A} U$ is the change-of-basis matrix to a basis of
$M_{[\alpha',\beta']}$ of the desired form. This will complete the proof.

The construction of the $R_m$ proceeds recursively as follows. Given $R_{m_0}, \dots, R_{m}$, we first verify that
\[
\left|R_{m}\right|, \left|R_{m}^{-1}\right| \leq p^{2km}.
\]
This is clear for $m = m_0$, so we may assume $m > m_0$. Choose any $\rho \in [\alpha', \beta']$.
As noted above, we have $\left|I_n - S_{m_0,A}^{-1} S_{m,A} R_{m}^{-1}\right|_\rho < 1$, so $\left|S_{m_0,A}^{-1} S_{m,A} R_{m}^{-1}\right|_\rho= \left|R_{m} S_{m,A}^{-1} S_{m_0,A}\right|_\rho = 1$.
We then deduce the claim from the bound $\left|S_{m,A}\right|_\rho, \left|S_{m,A}^{-1}\right|_\rho \leq p^{km}$.

Next, put $T_m = R_{m} S_{m,A}^{-1} S_{m+1,A}$; we then have
\[
\left|T_m\right|_{[\alpha,\beta]}, \left|T_m^{-1}\right|_{[\alpha,\beta]} \leq p^{4km+k}.
\]
Let $T_{m,0}$ be the constant coefficient of $T_{m}$. Since $T_m$ is a series in $t^{p^m}$,
Lemma~\ref{L:approximate using distance}(b) implies
\[
\left|T_m - T_{m,0}\right|_{[\alpha', \beta']} \leq p^{4km+k} \eta^{-p^m}.
\]
We may now take $R_{m+1} = T_{m,0}$, because
\begin{align*}
\left|I_n - R_{m+1} T_{m}^{-1}\right|_{[\alpha', \beta']} & \leq \left|T_m^{-1}\right|_{[\alpha', \beta']}  \cdot \left|T_m - T_{m,0}\right|_{[\alpha', \beta']} \\
&\leq p^{8km+2k} \eta^{-p^m} \\
&< p^{10km} \eta^{-cm} \leq \lambda^m < 1
\end{align*}
and so $\left|I_n - T_{m} R_{m+1}^{-1}\right|_{[\alpha', \beta']} \leq \lambda^m$. 
This completes the construction of the $R_m$ and thus the proof.
\end{proof}

\begin{remark}
Theorem~\ref{T:CM part 1} is sufficient to recover the full Christol-Mebkhout decomposition theorem in the case
of a differential module admitting an exponent contained in $\ZZ_p \cap \QQ$, by pulling back along the map 
$t \mapsto t^m$ for a suitably divisible integer $m \in \ZZ$.
\end{remark}

The second part is a splitting theorem for modules admitting an exponent with $p$-adic non-Liouville differences.
This may be generalized as follows.
\begin{theorem} \label{T:CM part 2}
Suppose that $M$ admits an exponent $A$ 
admitting the Liouville partition $A_1,\dots,A_k$. Then 
for any closed subinterval $J$ of $I$ of positive length,
there exists a unique direct sum decomposition $M_J = M_1 \oplus \cdots \oplus M_k$ such that for
$g=1,\dots,k$, $M_g$ admits an exponent over $J$ weakly equivalent to $A_g$.
\end{theorem}
\begin{proof}
We may assume $0 \notin I$ thanks to Remark~\ref{R:contains 0}.
We first verify uniqueness. 
Suppose to the contrary that there is a second decomposition $M_J = M'_1 \oplus \cdots \oplus M'_k$ of the desired
form for which there exist $g \neq h$ such that $M_{gh} = M_g \cap M'_h$ is nonzero.
Apply Theorem~\ref{T:exponents}(a) to produce exponents $B_1, B_2, B_3$ of
$M_{gh}$, $M_g/M_{gh}$, $M'_h/M_{gh}$.
By Remark~\ref{R:exponent functoriality} and Theorem~\ref{T:exponents}(b), $B_1 \cup B_2$ 
is weakly equivalent to $A_g$ and $B_1 \cup B_3$
is weakly equivalent to $A_h$.
We then obtain the desired contradiction by applying Proposition~\ref{P:weakly equivalent}(a).

We next verify existence. To simplify notation, we may reduce
to the case $k=2$. 
Let $J$ be any closed subinterval of $I$ of positive length;
by Theorem~\ref{T:exponents}(a,b) and Proposition~\ref{P:weakly equivalent}(b),
$M$ admits an exponent $A$ over $J$ of the specified form.
Choose an ordering $A = \{a_1,\dots,a_n\}$.
Choose $\eta > 1$ and $\alpha, \alpha', \beta, \beta' \in I$ such that
$\alpha' < \beta'$, $\alpha' = \alpha \eta$, $\beta' = \beta/\eta$,
and $J \subseteq [\alpha', \beta']$.
Fix a basis of $M_J$ and define the matrices $S_{m,A}$ as in Definition~\ref{D:exponent}.
Choose $\lambda \in (0,1)$ and $c>0$ so that $p^{9k} \eta^{-c} \leq \lambda$.
By hypothesis, there exists $m_0 > 0$ such that for $m \geq m_0$, 
$b_1 \in A_1$, $b_2 \in A_2$,
the congruence $h \equiv b_1 - b_2 \pmod{p^m}$ forces $\left|h\right| \geq cm$.

Let $\Pi_m$ be the projector onto the submodule of $M_J$ generated by
$\bv_{m,A,i}$ for those $i$ for which $a_i \in A_1$; then
\[
\left|\Pi_m\right|_{[\alpha, \beta]} \leq p^{2km}.
\]
For those $j$ for which $a_j \in A_1$, write 
$(\Pi_m - \Pi_{m+1})(\bv_{m,A,j}) = \sum_i a_{m,i} \bv_{m+1,A,i}$,
so that
\[
\left|a_{m,i}\right|_{[\alpha, \beta]}  \leq p^{4km+2k}.
\]
Since $(\Pi_m - \Pi_{m+1})(\bv_{m,A,j}) = (1 - \Pi_{m+1})(\bv_{m,A,j})$,
we have $a_{m,i} = 0$ when $a_i \in A_1$. On the other hand, when $a_i \in A_2$,
the coefficient of $t^h$ in $a_{m,i}$ can only be nonzero if $h \equiv a_i - a_j \pmod{p^m}$;
this implies
\[
\left|a_{m,i}\right|_{[\alpha', \beta']} \leq p^{4km+2k} \eta^{-cm} \qquad (m \geq m_0)
\]
by Lemma~\ref{L:approximate using distance}(a), and so 
\[
\left|(\Pi_m - \Pi_{m+1})(\bv_{m,A,j})\right|_{[\alpha', \beta']} \leq p^{5km+2k}\eta^{-cm} \qquad (m \geq m_0).
\]
Similarly, for those $j$ for which $a_j \in A_2$,
\[
\left|(\Pi_m - \Pi_{m+1})(\bv_{m,A,j})\right|_{[\alpha', \beta']} \leq p^{5km+3k}\eta^{-cm} \qquad (m \geq m_0)
\]
and so
\[
\left|(\Pi_m - \Pi_{m+1})\right|_{[\alpha', \beta']} \leq p^{6km+3k}\eta^{-cm} \leq \lambda^m \qquad (m \geq m_0).
\]
Therefore the $\Pi_m$ converge to an endomorphism of $M_J$, which is forced to be a projector defining the desired splitting.
\end{proof}

We may put Theorem~\ref{T:CM part 1} and Theorem~\ref{T:CM part 2} to separate integer exponents
from $p$-adic non-Liouville exponents.
\begin{cor} \label{C:cm1}
Suppose that $M$ has $p$-adic non-Liouville exponents. Then there exists a unique direct sum decomposition
$M \cong M_1 \oplus M_2$ with the following properties.
\begin{enumerate}
\item[(a)]
The module $M_1[t^{-1}]$ admits a basis on which $D$ acts via a nilpotent matrix over $K$. In particular, $M_1[t^{-1}]$
is unipotent (i.e., it is a successive extension of trivial differential modules over $R_I[t^{-1}]$).
\item[(b)]
No exponent of $M_2$ contains an integer or a $p$-adic Liouville number. 
\end{enumerate}
\end{cor}
\begin{proof}
We obtain the splitting $M \cong M_1 \oplus M_2$ using Theorem~\ref{T:CM part 2}.
We then obtain (a) using Theorem~\ref{T:CM part 2}
and (b) using Corollary~\ref{C:transmit integer}.
\end{proof} 

We recover as a corollary
the original decomposition theorem of Christol-Mebkhout, as stated in \cite[Theorem~13.6.1]{kedlaya-book}.

\begin{cor} \label{C:cm2}
Fix a set $S$ of coset representatives of $\ZZ$ in $\ZZ_p$.
Suppose that $M$ has $p$-adic non-Liouville exponent differences. 
Then there exists a unique direct sum
decomposition $M \cong \bigoplus_{\lambda \in S} N_\lambda$ in which $N_\lambda$ admits a basis on which $D$
acts via a matrix over $K$ with all eigenvalues equal to $\lambda$. In particular, $N_\lambda$
is isomorphic to a successive extension of copies of $M_\lambda$.
\end{cor}
\begin{proof}
We induct on $\rank(M)$. By twisting, we can force $M$ to admit an exponent containing $0$.
We may then split $M$ using Corollary~\ref{C:cm1} and continue.
\end{proof}

\begin{cor} \label{C:p-th power trivial}
If $M^{\otimes p}$ is unipotent (resp.\ trivial), then so is $M$.
\end{cor}
\begin{proof}
We first prove the unipotent case. Apply Theorem~\ref{T:exponents}(a) to construct an exponent $A$ for $M$.
By Remark~\ref{R:exponent functoriality}, the $p$-fold sum $A + \cdots + A$ is an exponent for $M^{\otimes p}$,
as by assumption is the zero tuple. Since the latter has $p$-adic non-Liouville differences,
Lemma~\ref{L:exponent differences} implies that $A + \cdots + A$ is equivalent (not just weakly equivalent) to 0.
In particular, $pa \in \ZZ$ for each $a \in A$; since $a \in \ZZ_p$, this is only possible when $a \in \ZZ$ for each $a \in A$. By Corollary~\ref{C:cm1}, $M$ is unipotent.

Suppose now that $M^{\otimes p}$ is trivial. By the previous paragraph, $M$ admits a basis on which $D$ acts via a nilpotent matrix $N$ over $K$; note that the conjugacy class of the matrix $N$ is uniquely determined by $M$
because $1$ is nonzero in the cokernel of $t\frac{d}{dt}$ on $R_I$.
In particular, since $M^{\otimes p}$ is trivial, the $p$-th tensor power of $N$ must be zero, but this is only possible if $N$ is itself zero.
Consequently, $M$ itself must be trivial.
\end{proof}

\begin{remark}
It is known that Theorem~\ref{T:CM part 2} does not extend to integer partitions. For example, if $\lambda$ is a 
$p$-adic Liouville number, then there exist nontrivial extensions of $M_0$ by $M_\lambda$, each of which 
admits $\{0,\lambda\}$ as an exponent by Remark~\ref{R:exponent functoriality}. 
In fact, we expect (although we have no examples)
that one can even find irreducible differential modules
of rank greater than 1 satisfying the Robba condition.
\end{remark}

\subsection{Frobenius antecedents and descendants}
\label{sec:antecedents descendants}

The construction of Frobenius antecedents and descendants can be generalized to differential modules
over power series. We record here some key facts from \cite[Chapter~10]{kedlaya-book} which we will use.

\begin{hypothesis} \label{H:Frobenius}
Throughout \S\ref{sec:discs annuli},
assume $p>0$,
fix a subinterval $I$ of $[0, +\infty)$,
and take $R = R_I$ or $R = R_I^{\an}$.
Let $(M,D)$ be a differential module of rank $n$ over $(R, \frac{d}{dt})$.
\end{hypothesis}

\begin{defn} \label{D:global Frobenius descendant}
Let $I^p$ be the subinterval of $[0, +\infty)$ consisting of $\gamma^p$ for all $\gamma \in I$.
In case $R = R_I$ (resp.\ $R = R_I^{\an})$, let $R'$ be a copy of $R_{I^p}$ (resp.\ $R_{I^p}^{\an}$)
in the variable $t^p$, identified with a subring of $R$. We may then view $(R', \frac{d}{dt^p})$ as a differential ring.

If $0 \notin I$, we may form the \emph{Frobenius descendant} $\varphi_* M$ as in \cite[Definition~10.3.4]{kedlaya-book}; that is, $\varphi_* M$ is a copy of $M$ viewed as an $R'$-module equipped with the derivation $D' = p^{-1} t^{1-p} D$.
For any $\rho \in I$, $(\varphi_* M) \otimes_{R'} F'_\rho$ may be naturally identified with the Frobenius descendant of $M_\rho$.
\end{defn}

\begin{prop} \label{P:global Frobenius antecedent}
Suppose that $f_i(M,r) < r - \log \omega$ for all $r \in -\log I$. 
 Then there exists a unique (up to unique isomorphism)
differential module $M'$ over $(R', \frac{d}{dt^p})$ such that for $\rho \in I \setminus \{0\}$,
$M' \otimes_{R'} F'_\rho$ is the Frobenius antecedent of $M_\rho$.
\end{prop}
\begin{proof}
See \cite[Theorem~10.4.4]{kedlaya-book}.
\end{proof}

We will also need to consider ``off-center Frobenius descendants'' as in \cite[\S 10.8]{kedlaya-book}.
\begin{defn}
Suppose that $R = R_{[0,1]}^{\an}$.
Let $R''$ be a copy of $R_{[0,1]}^{\an}$ in the variable $u$.
For $\rho \in (0,1]$, let $F''_\rho$ be a copy of $F_\rho$ in the variable $u$, so that $R''$ maps to $F''_\rho$.
Choose $\lambda \in K$ with $\left| \lambda \right| = 1$,
then identify $R''$ with a subring of $R$ by identifying $u$ with $(t+\lambda)^p - \lambda^p$.
\end{defn}

\begin{prop} \label{P:off-center descendant}
Suppose that $R = R_{[0,1]}^{\an}$.
Let $\psi_* M$ be a copy of $M$ viewed as a differential module over $(R'', \frac{d}{du})$. Then for $\rho \in (\omega,1]$, the multiset consisting of the intrinsic subsidiary radii of $(\psi_* M) \otimes_{R''} F''_{\rho^p}$
is the union of the multisets
\[
\begin{cases} \{s^p\} \cup \{ \omega^p \rho^{-p} \, \mbox{($p-1$ times)}\} & \mbox{if } s  > \omega\rho \\
\{p^{-1} s \rho^{1-p} \, \mbox{($p$ times)}\} & \mbox{if } s  \leq \omega/\rho
\end{cases}
\]
for $s$ running over the intrinsic subsidiary radii of $M_\rho$.
\end{prop}
\begin{proof}
By rescaling $t$, we may reduce to the case where $\lambda = 1$. In this case, see
\cite[Theorem~10.8.3]{kedlaya-book}.
\end{proof}

\subsection{Variation of intrinsic radii}
\label{sec:discs annuli}

We now consider differential modules not necessarily satisfying the Robba condition,
with an eye towards the variation of the intrinsic subsidiary radii. The results we report here are taken from \cite[Chapters~11--12]{kedlaya-book}; starting in \S \ref{sec:Berkovich discs annuli}, we will see how to make more definitive
statements in the language of Berkovich spaces. To facilitate this transition, we record a couple of important
direct corollaries of the results of \cite{kedlaya-book}.

\begin{hypothesis} \label{H:discs annuli modules}
Throughout \S\ref{sec:discs annuli},
fix a subinterval $I$ of $[0, +\infty)$,
and let $M$ be a differential module of rank $n$ over $(R_I^{\an}, \frac{d}{dt})$.
\end{hypothesis}

\begin{defn}
For $\rho \in I \setminus \{0\}$, put $M_\rho = M \otimes_{R_I^{\an}} F_\rho$.
For $r \in -\log I$ and $i=1,\dots,n$, define $f_i(M,r)$ so that the list of intrinsic subsidiary radii of $M_{e^{-r}}$ in increasing
order is
\[
\exp(r-f_1(M,r)), \dots, \exp(r-f_n(M,r)).
\]
Put $F_i(M,r) = f_1(M,r) + \cdots + f_n(M,r)$.
As observed in Definition~\ref{D:intrinsic radius}, the functions $f_i$ and $F_i$ are invariant under enlargement
of the constant field $K$.
\end{defn}

\begin{prop} \label{P:variation}
For $i=1,\dots,n$, we have the following.
\begin{enumerate}
\item[(a)] (Linearity)
The functions $f_i(M,r)$ and $F_i(M,r)$ are continuous and piecewise affine.
Moreover, these functions assume only finitely many different slopes
over any closed subinterval of $-\log I$ (even if $0 \in I$).
\item[(b)] (Integrality)
If $i=n$ or $f_i(M,r_0) > f_{i+1}(M,r_0)$, then 
the slopes of
$F_i(M,r)$ in some neighborhood of $r_0$ belong to $\ZZ$.
Consequently, the slopes of each $f_i(M,r)$ and $F_i(M,r)$ belong to
$\frac{1}{1} \ZZ \cup \cdots \cup \frac{1}{n} \ZZ$.
\item[(c)] (Subharmonicity)
Suppose that $K$ is algebraically closed, $\alpha < 1 < \beta$, and $f_i(M,0) > 0$.
Let $s_{\infty,i}(M)$ and
$s_{0,i}(M)$ be the left and right slopes of 
$F_i(M,r)$ 
at $r=0$. For $\overline{\mu} \in \kappa_K^\times$,
choose any $\mu \in \gotho_K$ lifting $\overline{\mu}$,
let $T_\mu$ denote the substitution $t \mapsto t + \mu$,
and let $s_{\overline{\mu},i}(M)$ be the right slope of
$F_i(T^*_\mu(M), r)$ at $r=0$. Then
\[
s_{\infty,i}(M) \leq \sum_{\overline{\mu} \in \kappa_K} 
s_{\overline{\mu},i}(M),
\]
with equality if either $i=n$ and $f_n(M,0) > 0$,
or $i<n$ and $f_i(M,0) > f_{i+1}(M,0)$.
\item[(d)] (Monotonicity)
Suppose that  $0 \in I$. Then for any point $r_0$ where $f_i(M,r_0) > r_0$,
the slopes of $F_i(M,r)$ are nonpositive in some neighborhood of $r_0$.
\item[(e)] (Convexity)
The function $F_i(M,r)$ is convex.
\end{enumerate}
\end{prop}
\begin{proof}
See \cite[Theorem~11.3.2]{kedlaya-book}.
\end{proof}

\begin{cor} \label{C:right slope 1}
Suppose that $0 \in I$ and $f_1(M,r_0) = r_0$ for some $r_0 \in -\log I$.
Then $f_1(M,r) = r$ for all $r \geq r_0$.
\end{cor}
\begin{proof}
By Proposition~\ref{P:variation}(a,d,e), the function $f_1(M,r)$ is piecewise affine and convex everywhere, and 
nonincreasing wherever it is greater than $r$. Since $f_1(M,r_0) = r_0$ and $f_1(M,r) \geq r$ everywhere,
all of the slopes of $f_1(M,r)$ for $r \geq r_0$ must be at least 1. However, none of them can be strictly
greater than 1 because this would force $f_1(M,r) > r$ for some $r$, and then $f_1(M,r)$ would be forced to
be nonincreasing. This proves the claim.
\end{proof}

\begin{cor} \label{C:constant initial}
Suppose that $0 \in I$ and for some $r_0 \in -\log I$, some $r_1 > r_0$,
and some $j \in \{0,\dots,n\}$,
the functions $f_1(M,r), \dots, f_j(M,r)$ are equal to some constant value $c$ for $r \in (r_0,r_1)$.
Then
\[
f_i(M,r) = \max\{r,c\} \qquad (r > r_0; \,i=1,\dots,j).
\]
\end{cor}
\begin{proof}
We prove that the claim holds for $f_1,\dots,f_i$ by induction on $i$, with base case $i=0$.
Given the induction hypothesis for $i-1$,
note that since $f_i(M,r) \geq r$ for all $r > r_0$, we must have $c > r_0$.
By Proposition~\ref{P:variation}(d,e) and the induction hypothesis, the function $f_i(M,r)$ is convex everywhere and nonincreasing wherever it is greater
than $r$. It follows that $f_i(M,r) = c$ for $r \in (r_0, c]$.
By Corollary~\ref{C:right slope 1}, we then have $f_i(M,r) = r$ for $r \geq c$.
\end{proof}

\begin{prop} \label{P:rational intercepts}
For $i=1,\dots,n$, on any interval where $f_i(M,r)$ is affine, it has the form $ar+b$ for some
$a \in \QQ$ and some $b$ in the divisible closure of $\log |K^\times|$.
\end{prop}
\begin{proof}
See \cite[Corollary~11.8.2]{kedlaya-book}.
\end{proof}

\begin{prop} \label{P:decompose disc}
Suppose that $0 \in I$ and that for some $i \in \{1,\dots,n-1\}$ and some $\gamma \in I \setminus \{0\}$, 
the following conditions hold.
\begin{enumerate}
\item[(a)]
The function $F_i(M,r)$ is constant for $r < -\log \gamma$.
\item[(b)]
We have $f_i(M,r) > f_{i+1}(M,r)$ for $r < -\log \gamma$.
\end{enumerate}
Then $M$ admits a unique direct sum decomposition separating the first $i$ intrinsic subsidiary radii of $M_\rho$
for all $\rho > \gamma$.
\end{prop}
\begin{proof}
See \cite[Theorem~12.5.1]{kedlaya-book}.
\end{proof}

\begin{cor} \label{C:decompose disc}
Suppose that $I = [0, \beta]$ or $I = [0,\beta)$ for some $\beta>0$ and put $r_0 = -\log \beta$.
Suppose that for some $i \in \{0,\dots,n\}$ and some $r_1 > r_0$,
the following conditions hold.
\begin{enumerate}
\item[(a)]
For $j=1,\dots,i$, the function $f_j(M,r)$ is constant for $r \in (r_0,r_1)$.
\item[(b)]
If $i<n$, then $\liminf_{r \to r_0^+} f_{i+1}(M,r) = r_0$.
\end{enumerate}
Then $M$ admits a direct sum decomposition $M_0 \oplus M_1 \oplus \cdots$ such that
\[
f_1(M_0,r) = \cdots = f_{\rank(M_0)}(M_0,r) = r \qquad (r > r_0)
\]
and for each $k>0$, there is a constant $c_k > r_0$ such that
\[
f_1(M_k,r) = \cdots = f_{\rank(M_k)}(M_k,r) = \max\{r, c_k\} \qquad (r > r_0).
\]
In particular, for $j=i+1,\dots,n$, $f_j(M,r) = r$ for $r > r_0$.
\end{cor}
\begin{proof}
We induct on $i$.
Suppose first that $i=0$. In this case, 
we cannot have $f_1(M,r_1) > r_1$ for any $r_1 > r_0$, as then
Proposition~\ref{P:variation}(d) would imply $f_1(M,r) > r_1$ for all $r \in (r_0,r_1)$
and hence $\liminf_{r \to r_0^+} f_1(M,r) \geq r_1$, violating condition (b).
Hence for all $r > r_0$ and all $j$, we have $r = f_1(M,r) \geq f_j(M,r) \geq r$,
proving the claim with $M = M_0$.

Suppose next that $i>0$. Let $c_1$ be the constant value of $f_1(M,r)$ for $r \in (r_0,r_1)$.
By Corollary~\ref{C:constant initial}, we have $f_1(M,r) = \max\{r,c_1\}$ for $r > r_0$.
Let $m$ be the largest value for which $f_1(M,r) = f_m(M,r)$
for $r$ in some right neighborhood $(r_0,r_1)$ of $r_0$. 
Split $M$ as $M_1 \oplus M_2$ as per Proposition~\ref{P:decompose disc}
so that $M_1$ accounts for the first $m$ intrinsic subsidiary radii of $M_\rho$ for $\rho > e^{-r_1}$.
For $r \geq c_1$, for all $j$ we have $r \leq f_j(M_1,r) \leq f_1(M_1,r) \leq f_1(M,r)= r$ and so
$f_j(M_1,r) = r$. By Proposition~\ref{P:variation}(d),
$F_{\rank(M_1)}(M_1,r)$ is convex; since it agrees with the constant function
$\rank(M_1) c_1$ for $r \in (r_0,r_1)$ and for $r=c_1$, we must have
$F_{\rank(M_1)}(M_1,r) = \rank(M_1) c_1$ for $r \in (r_0,c_1]$.
Since in addition $f_j(M_1,r) \leq f_1(M_1,r) \leq f_1(M,r) = \max\{r,c_1\}$, we must have
$f_j(M_1,r) = c_1$ for $r \in (r_0,c_1]$.
Thus $M_1$ has all of the desired properties, so we may apply the induction hypothesis to $M_2$ to prove the claim.
\end{proof}

\begin{prop} \label{P:decompose annulus}
Suppose that $I = (\alpha,\beta)$ for some $\alpha,\beta>0$
and that for some $i \in \{1,\dots,n-1\}$, the following conditions hold.
\begin{enumerate}
\item[(a)]
The function $F_i(M,r)$ is affine for $r \in -\log I$.
\item[(b)]
We have $f_i(M,r) > f_{i+1}(M,r)$ for $r \in -\log I$.
\end{enumerate}
Then 
$M$ admits a unique direct sum decomposition separating the first $i$ intrinsic subsidiary radii of $M_\rho$
for every $\rho \in I$. 
\end{prop}
\begin{proof}
See \cite[Theorem~12.4.2]{kedlaya-book}.
\end{proof}

\subsection{Decompositions over open annuli}
\label{subsec:annuli refined}

We now embark on a deeper analysis of differential modules over open annuli than is found in \cite{kedlaya-book},
concentrating on spectral decompositions and on properties of refined modules. For the latter, we
incorporate some ideas of Xiao \cite{xiao-thesis,xiao-refined}. 
\begin{hypothesis}
Throughout \S\ref{subsec:annuli refined}, continue to retain Hypothesis~\ref{H:discs annuli modules},
but assume further that $n>0$ and $I = (\alpha,\beta)$ for some $\alpha, \beta > 0$.
\end{hypothesis}

\begin{defn}
We say that $M$ is \emph{pure} if the functions $f_1(M,r),\dots,f_n(M,r)$ 
 for $r \in -\log I$ are all equal to a single affine function.
 A \emph{spectral decomposition} of $M$ is a direct sum decomposition $M = \bigoplus_i M_i$ in which each summand
$M_i$ is pure and the values $f_1(M_i, r)$ are all distinct for each $r \in -\log I$.
If such a decomposition exists, it specializes to the spectral decomposition of $M_\rho$ for all $\rho \in I$;
in particular, a spectral decomposition is unique if it exists.
\end{defn}

\begin{lemma} \label{L:spectral decomposition exists}
Consider the following conditions.
\begin{enumerate}
\item[(a)]
The module $M$ admits a spectral decomposition.
\item[(b)]
For $i=1,\dots,n$, the function $f_i(M,r)$ is affine for $r \in -\log I$.
\item[(c)]
The functions $F_n(M,r)$ and $F_{n^2}(\End(M),r)$ are affine for $r \in -\log I$. (That is, $M$ is \emph{clean} in the sense of
\cite[Definition~12.8.2]{kedlaya-book}.)
\end{enumerate}
Then (a) and (b) are equivalent, and (c) implies both of them.
\end{lemma}
\begin{proof}
It is clear that (a) implies (b), and \cite[Theorem~12.8.3]{kedlaya-book} shows that (c) implies (b), so it remains to check that (b) implies (a). Given (b), for $i=1,\dots,n-1$, if there exists $r_0 \in -\log I$ for which $f_i(M,r_0) = f_{i+1}(M,r_0)$, then we must have
$f_i(M,r) = f_{i+1}(M,r)$ identically; otherwise, since $f_i(M,r)$ and $f_i(M,r)$ are both affine, the inequality $f_i(M,r) \geq f_{i+1}(M,r)$ would have to be violated on one side of $r_0$. In other words, for $i=1,\dots,n-1$, either $f_i(M,r) = f_{i+1}(M,r)$ for $r \in -\log I$
or $f_i(M,r) > f_{i+1}(M,r)$ for $r \in -\log I$.
This allows us to apply Proposition~\ref{P:decompose annulus} to obtain a spectral decomposition, yielding (a).
\end{proof}

\begin{defn}
Suppose that $M$ admits a spectral decomposition. By the \emph{Robba component} of $M$, we mean the summand
$M_1$ in the spectral decomposition of $M$ for which $f_1(M_1,r) = r$ for each $r \in -\log I$, or the zero submodule
if no such summand exists.
\end{defn}

\begin{lemma} \label{L:compare cohomology1}
Suppose that $M$ admits a spectral decomposition. Let $M_1$ be the Robba component of $M$.
Then the natural maps $H^i(M_1) \to H^i(M)$ are bijections for $i=0,1$.
\end{lemma}
\begin{proof}
Let $M_2$ be the summand complementary to $M_1$ in the spectral decomposition of $M$. It is clear that $H^0((M_2)_\rho) = 0$ for $\rho \in I$,
proving the desired bijectivity for $i=0$. For $i=1$, note that 
$f_1(M_2,r) > r$ for each $r \in -\log I$, so any extension
$0 \to R_I \to N \to M_2 \to 0$ splits by Lemma~\ref{L:spectral decomposition exists}.
\end{proof}

\begin{lemma} \label{L:compare cohomology2}
Suppose that $M$ admits a spectral decomposition. Let $M_1$ be the Robba component of $M$.
Assume either that $p=0$ or that $p>0$ and $M_1$ has $p$-adic non-Liouville exponents.
\begin{enumerate}
\item[(a)]
Let $M_2$ be the maximal unipotent submodule of $M_1$. Then
the natural maps $H^i(M_2) \to H^i(M)$ are bijections for $i=0,1$.
\item[(b)]
The composition $H^0(M) \times H^1(M^\dual) \to H^1(R_I) \to K$
in which the first map is induced by the natural pairing $M \times M^\dual \to R_I$ and the second map is the residue
map is a perfect pairing of
finite-dimensional $K$-vector spaces.
\item[(c)]
For any open subinterval $J$ of $I$, the map
\[
H^i(M) \to H^i(M_{J})
\]
is a bijection for $i=0,1$.
\end{enumerate}
\end{lemma}
\begin{proof}
To prove (a), we may replace $M$ by $M_1$ using Lemma~\ref{L:compare cohomology1}.
In case $p=0$, we may check the claim after replacing $K$
by a finite extension $K'$, since $M$ may be viewed as a direct summand of $M \otimes_K K'$.
After a suitable such extension, by Theorem~\ref{T:CM char 0}
we may decompose $M_1 = M_2 \oplus M_3$ in such a way that $M_3$ becomes a successive extension of copies of $M_\lambda$ for various $\lambda \in \gotho_K \setminus \ZZ$. 
To see that $H^i(M_3) = 0$ for $i=0,1$, we may use the snake lemma to reduce to the case
$M = M_\lambda$ for some $\lambda \in \gotho_K \setminus \ZZ$. In this case, vanishing of $H^0$
follows from the nontriviality of $M_\lambda$, while vanishing of $H^1$ follows from
Theorem~\ref{T:CM char 0} applied to an extension $0 \to R_I \to N \to M_\lambda \to 0$.

To prove (a) in case $p>0$,
apply Corollary~\ref{C:cm1} to decompose $M_1 = M_2 \oplus M_3$ where $M_3$
has an exponent containing no integer or $p$-adic Liouville number.
On one hand, $H^0(M_3) = 0$ because otherwise Remark~\ref{R:exponent functoriality} would
force $M$ to have an exponent containing $0$. On the other hand, $H^1(M_3) = 0$ because we may split
any extension $0 \to R_I \to N \to M_3 \to 0$ using Remark~\ref{R:exponent functoriality} and
Theorem~\ref{T:CM part 2}. 

To prove (b) and (c), we may use (a) to reduce to the case $M = M_2$. We may then
use the snake lemma to reduce to the case $M = R_I$, for which both claims are
easily verified.
\end{proof}

\begin{lemma} \label{L:compare cohomology3}
Suppose that $M$ admits a spectral decomposition. 
Assume either that $p=0$ or that $p>0$ and the Robba component of $M$ has $p$-adic non-Liouville exponent differences.
Then for any $\rho \in I$, the map $H^0(M) \to H^0(M_\rho)$ is a bijection.
\end{lemma}
\begin{proof}
In case $p=0$, this is immediate from Lemma~\ref{L:compare cohomology2}(a).
In case $p>0$, apply Corollary~\ref{C:cm2} to reduce to the case $M = M_\lambda$ for some $\lambda \in \ZZ_p$.
The claim then holds because by \cite[Proposition~9.5.3]{kedlaya-book},
$H^0(M_{\lambda,\rho}) = 0$ whenever $\lambda \notin \ZZ$.
\end{proof}

%\begin{defn}
%
%For $\lambda \in K$, let $M_\lambda$ denote the differential module over $R_I$ on a single generator $\bv$ satisfying $D(\bv) = \lambda t^{-1}\,\bv$.
%We say that $M$ is \emph{potentially refined}
%if there exist a positive integer $m$ not divisible by $p$ and a finite Galois tamely ramified extension $K'$ of $K$ 
%containing the group of $m$-th roots of unity such that
%$M' = M \otimes_{K[t]} K'[t^{1/m}]$ 
%splits as a direct sum of refined submodules on which $\Gal(K'(t^{1/m})/K(t))$ acts transitively.
%In this case, the projectors defining the refined decomposition of $M'$
%span a submodule of $\End(M')$ which descends to a differential submodule $P$ of $\End(M)$ itself;
%the module $P$ then splits uniquely as a direct sum $P = \oplus_{i=0}^{m-1} P_i$ with the property that 
%$M_{i/m}^\dual \otimes P_i$ is trivial for $i=0,\dots,m-1$.
%We call $P$ the \emph{refining module} of $M$.
%
%We say that two potentially refined modules $M_1, M_2$ over $R_I$ are \emph{equivalent} if
%after tensoring over $K[t]$ with some $K'[t^{1/m}]$ and splitting into refined
%submodules, each summand of $M_1$ is equivalent to a summand of $M_2$. This defines an equivalence relation.
%\end{defn}

\begin{defn}
We say that $M$ is \emph{refined} if $M$ is pure and moreover $f_1(M,r) > f_1(M^\dual \otimes M,r)$ for all $r \in -\log I$ (that is, $M$ is pure and $M_\rho$ is refined for all $\rho \in I$).
If $M_1, M_2$ are refined, we say they are \emph{equivalent} if $f_1(M_1^\dual \otimes M_2,r) < f_1(M_1,r), f_1(M_2,r)$
for all $r \in -\log I$. Note that if $M_1$ and $M_2$ are inequivalent, then 
by convexity (Proposition~\ref{P:variation}(e)) we must have $f_1(M_1^\dual \otimes M_2,r) = \max\{f_1(M_1,r), f_1(M_2,r)\}$
for all $r \in -\log I$.

A \emph{refined decomposition} of $M$ is a direct sum
decomposition in which each summand $M_i$ is either refined or satisfies the Robba condition,
at most one summand satisfies the Robba condition,
and any two distinct refined summands $M_i, M_j$ are inequivalent. 
Such a decomposition specializes to a refined
decomposition of $M_\rho$ for each $\rho \in I$, and hence is unique if it exists.
%A \emph{potentially refined decomposition} (resp.\ \emph{refined decomposition}) of $M$ is a direct sum
%decomposition in which each summand $M_i$ is either potentially refined (resp.\ refined) or satisfies
%$f_1(M_i,r) = r$ for all $r \in -\log I$,
%and any two distinct potentially refined (resp.\ refined) summands $M_i, M_j$ are inequivalent. 
%Such a decomposition specializes to a potentially refined (resp.\ refined)
%decomposition of $M_\rho$ for each $\rho \in I$, and hence is unique if it exists.
\end{defn}

It is easiest to obtain refined decompositions using the following construction of \emph{test modules}
(compare \cite[Example~1.3.20]{xiao-refined}).
\begin{defn} \label{D:test modules}
For any finite tamely ramified extension $K'$ of $K$,
any $\lambda \in K'$, any positive integer $m$ not divisible by $p$,
 any positive integer $e$ which is a power of $p$  
(which must be $1$ if $p=0$), 
and any integer $h$ coprime to $em$, let $N_{\lambda,h,e,m}$ be the differential module
over $(R_I \otimes_{K[t]} K'[t^{1/m}], \frac{d}{dt^{1/m}})$
on the generators $\bv_1,\dots,\bv_e$ given by
\[
D(\bv_1) = t^{-1/m} \bv_2, \,\dots, \,D(\bv_{e-1}) = t^{-1/m} \bv_e, \,D(\bv_e) = \lambda t^{-1/m+h/m} \bv_1.
\]
\end{defn}

\begin{lemma} \label{L:compute test modules}
With notation as in Definition~\ref{D:test modules}, for $\rho > 0$ we have
\[
\min\{\omega, IR((N_{\lambda,h,e,m})_\rho)\} = \min\{\omega, \omega \left|\lambda\right|^{-1/e} \rho^{-h/(em)}\}.
\]
\end{lemma}
\begin{proof}
This is immediate from Proposition~\ref{P:christol-dwork}.
\end{proof}

\begin{lemma} \label{L:test modules2}
Suppose that $M$ is pure and $f_1(M,r) > r - \log \omega$ for $r \in -\log I$.
Then for any $\rho \in I$, there exist a finite tamely ramified extension $K'$ of $K$ and a positive integer $m$
not divisible by $p$ such that $M_\rho \otimes_{K[t]} K'[t^{1/m}]$ admits a refined decomposition
in which for each summand $V$, there exist a scalar $\lambda \in K'$, a positive integer $e$ which is a power of $p$,
and an integer $h$ coprime to $em$
 such that $IR(M_\sigma) = IR((N_{\lambda,h,e,m})_\sigma)$ for $\sigma$
in some neighborhood of $\rho$ and $IR(V^\dual \otimes (N_{\lambda,h,e,m})_\rho) > IR(V)$.
\end{lemma}
\begin{proof}
We imitate the proof of \cite[Lemma~6.8.1]{kedlaya-book}.
Apply Corollary~\ref{C:cyclic vector} to produce $\bv \in M$ which is a cyclic vector in $M \otimes_{R_I} \Frac(R_I)$. Write $D^n(\bv) = a_0 \bv + \cdots + a_{n-1} D^{n-1}(\bv)$ with $a_0,\dots,a_{n-1} \in \Frac(R_I)$.
Factor the polynomial $P(T) = T^n - a_{n-1} T^{n-1} - \cdots - a_0$ over an algebraic closure of $\Frac(R_I)$ within an algebraic closure of $F_\rho$.
For each root $\alpha$, we can find $\lambda,h,e,m$ such that $\left|\alpha - m^{-1} \lambda^{1/e} t^{-1+h/(em)}\right|_\rho < \left|\alpha\right|_\rho$;
by Corollary~\ref{C:christol-dwork}, $IR(M_\sigma) = IR((N_{\lambda,h,e,m})_\sigma)$ for $\sigma$ in a neighborhood of $\rho$ and one of the intrinsic subsidiary radii of $M_\rho^\dual \otimes  (N_{\lambda,h,e,m})_\rho$ is greater than $IR(M_\rho)$. 
Apply Proposition~\ref{P:field decomposition1} to construct a refined decomposition of $M_\rho \otimes_{F_\rho} E$
for some finite tamely ramified extension $E$ of $F_\rho$; then each summand is equivalent
to $(N_{\lambda,h,e,m})_\rho$ for some $\lambda,h,e,m$, and in particular is stable under
$\Gal(E'/F')$ for $F' = F_\rho \otimes_{K[t]} K'[t^{1/m}]$ and
$E'$ a compositum of $E$, $F'$, and $K(\mu_m)$. We thus obtain a refined decomposition of
$M_\rho \otimes_{K[t]} K'[t^{1/m}]$ with the desired property.
\end{proof}

\begin{theorem} \label{T:refined decomposition}
Suppose that $M$ is pure. Then
%
%\begin{enumerate}
%\item[(a)]
%The module $M$ admits a potentially refined decomposition.
%\item[(b)]
there exist a finite tamely ramified extension $K'$ of $K$ and a positive integer $m$ not divisible by $p$ such that
$M \otimes_{K[t]} K'[t^{1/m}]$ 
admits a refined decomposition.
%\end{enumerate}
\end{theorem}
\begin{proof}
%It suffices to prove (b), as (a) then follows at once.
By virtue of the uniqueness of refined decompositions, we may work 
locally in a neighborhood of some $\rho \in (\alpha, \beta)$.
Suppose first that $IR(M_\rho) < \omega$.
To simplify notation, we may assume that the conclusion of Lemma~\ref{L:test modules2} holds with
$K' = K$ and $m=1$, so that $M_\rho$ admits a refined decomposition.
In addition, for each summand $V$ in the refined decomposition of $M_\rho$,
we can find a differential module $N$ over $R_I$ such that $IR(V^\dual \otimes N_\rho) > IR(V)$.
By continuity (Proposition~\ref{P:variation}(a)),
for $\sigma$ in a neighborhood of $\rho$, $M_\sigma^\dual \otimes N_\sigma$
has an intrinsic subsidiary radius strictly greater than $IR(M_\sigma) = IR(N_\sigma)$.
Apply Proposition~\ref{P:decompose annulus} to $N^\dual \otimes M$
to pull off a summand corresponding to the intrinsic subsidiary radii of $N_\rho^\dual \otimes M_\rho$ less than $IR(V)$,
then tensor with $N$ and project the decomposition from $N \otimes N^\dual \otimes M$ to $M$.
Repeating this process gives the desired decomposition.

Suppose next that $p>0$ and $IR(M_\rho) = \omega$. Let $M'$ be the global Frobenius descendant of $M$
(Definition~\ref{D:global Frobenius descendant}). By Proposition~\ref{P:descendant},
$IR(\varphi_* M_\rho) = \omega^p$, so we may apply the previous paragraph to exhibit
a finite tamely ramified extension $K'$ of $K$ and a positive integer $m$ not divisible by $p$ such that
$M' \otimes_{K[t^p]} K'[t^{p/m}]$ admits a refined decomposition. 
To simplify notation, we may assume that
$K' = K$ and $m = 1$, i.e., that $M'$ itself admits a refined decomposition.
In particular, $\varphi_* M_\rho$ admits a refined decomposition.
By Remark~\ref{R:refined construction}, if we group summands of $\varphi_* M$ into $\ZZ/p\ZZ$-orbits,
the resulting decomposition descends to a decomposition specializing to a refined decomposition of $M$.

Suppose finally that $p>0$ and $IR(M_\rho) > \omega$. Using Frobenius antecedents (Proposition~\ref{P:global Frobenius antecedent}), we may reduce to one of the previous cases.
\end{proof}

\begin{theorem} \label{T:cyclic type}
Suppose that either:
\begin{enumerate}
\item[(a)]
$M$ is refined and $\rank(M)$ is not divisible by $p$; or
\item[(b)]
$p>0$ and $M$ is of cyclic type.
\end{enumerate}
Then the slopes of $f_1(M,r)$ are in $\ZZ$.
\end{theorem}
\begin{proof}
Using Proposition~\ref{P:variation}(a), $f_1(M,r)$ is piecewise affine. It thus suffices to compute its slope
on a closed subinterval $J$ of $I$ on which $f_1(M,r)$ is affine. 
We may assume that this slope is not equal to 0 or 1,
as otherwise there is nothing left to check.

Suppose first that we are in case (a) with $p=0$.
Choose a generator $\bv$ of $\wedge^n M_J$,
define $c \in R_J$ by the formula
$D(\bv) = c\bv$, and let $N$ be the differential module over 
$R_J$
on a single generator $\bv$ given by $D(\bw) = (c/n)\bv$. We then have
$N^{\otimes n} \cong \wedge^n M$ and
so $f_1(N^\dual \otimes M, r) < f_1(M, r)$ for $r \in I$
by \cite[Proposition~6.8.4]{kedlaya-book}.
In particular, in some range we have
$f_1(M, r) = f_1(N, r)$, whereas $f_1(N,r)$ has integer slopes
by Proposition~\ref{P:variation}(b). This proves the claim in this case.

Suppose next that we are in case (a) with $p>0$. 
Since we assumed that the slope of $f_1(M,r)$ is neither 0 nor 1, we may
shrink $J$ to ensure that $f_1(M,r) \neq r - p^{-j} \log \omega$ for all $r \in J$
and all nonnegative integers $j$. We may then use Frobenius antecedents
 (Proposition~\ref{P:global Frobenius antecedent})
to reduce to the case where $f_1(M,r) > r - \log \omega$ for all $r \in J$,
and then argue as in (a).

Suppose finally that we are in case (b).
We may again assume that $f_1(M,r) > r - \log \omega$ for all $r \in J$;
we may also assume that $K$ is algebraically closed.
Pick any $r_0 \in J$ and 
apply Lemma~\ref{L:test modules2} to construct
$\lambda,h,e,m$ for which $IR(M_\rho^\dual \otimes (N_{\lambda,h,e,m})_\rho) > IR(M_\rho)$ for $\rho = e^{-r_0}$;
by continuity (Proposition~\ref{P:variation}(a)), 
the same inequality holds for $\rho$ in a neighborhood of $e^{-r_0}$.
For $\mu \in 1 + \gothm_K$, apply Corollary~\ref{C:christol-dwork}(a) to $\mu^* N_{\lambda,h,e,m}^\dual \otimes
N_{\lambda,h,e,m}$; it implies that there exists $a>0$ for which
$f_1(\mu^* M^\dual \otimes M,r) = f_1(M,r) + a \log |\mu-1|$
for $|\mu-1|$ sufficiently close to 1 and $r$ sufficiently close to $r_0$.
By Lemma~\ref{L:projector}, there exists a rank $1$ submodule $Q_\mu$ of $\mu^* M^\dual \otimes M$.
Since $\mu^* M^\dual \otimes M$ is of cyclic type,
we have $f_1(Q_\mu, r) = f_1(\mu^* M^\dual \otimes M,r) = f_1(M,r) + a \log |\mu-1|$ for suitable $\mu, r$.
Since $f_1(Q_\mu, r)$
has integer slopes by Proposition~\ref{P:variation}(b) again, so then does $M$ in a neighborhood of $r_0$;
this proves the claim in this case.
\end{proof}

\begin{remark}
Theorem~\ref{T:cyclic type}(b) is new to this paper. It was known previously that
if $p>0$, $M$ is of cyclic type, and $\End(M)$ has $p$-adic non-Liouville exponent differences,
then $M$ is a successive extension of differential modules of rank $1$ over $R_I$;
namely, this is an easy consequence of Corollary~\ref{C:cm2}. That previous result figures in the proofs of
the $p$-adic local monodromy theorem given by Andr\'e \cite{andre-monodromy} and Mebkhout \cite{mebkhout-monodromy};
see Remark~\ref{R:local monodromy}.
\end{remark}

The following refinement of Lemma~\ref{L:test modules2} will be used in the study of solvable modules
in \S \ref{sec:solvable}.
\begin{lemma} \label{L:rank 1 test modules}
Choose $\gamma, \delta$ with $\alpha < \gamma < \delta < \beta$. Suppose 
that $p>0$, $K$ is algebraically closed, 
$M$ is refined, and there exists a nonnegative integer $b$ such that 
$IR(M_\rho) = (\alpha/\rho)^b < \omega$ for $\rho \in [\gamma,\delta]$. Then there exists
a differential module $N$ over $R_I$ which is free of rank $1$ with $IR(N_\rho) = (\alpha/\rho)^b$ for $\rho \in (\alpha,\delta]$
and $IR((N^\dual \otimes M)_\rho) < IR(M_\rho)$ for $\rho \in [\gamma,\delta]$.
\end{lemma}
\begin{proof}
We may rescale to reduce to the case $\rho = \alpha = 1$.
Using Lemma~\ref{L:test modules2}, we may replace $M$ with $N_{\lambda,h,e,m}$; note that the fact that
$b \in \ZZ$ forces $e = m= 1$. After making the substitution $t \mapsto t^{-1}$,
we may perform the construction from the proof of \cite[Theorem~12.7.2]{kedlaya-book} to obtain the desired $N$.
\end{proof}

\subsection{Solvable modules}
\label{sec:solvable}

We continue in the vein of \cite{kedlaya-book}, next treating differential modules over rings of convergent power
series on an open annulus which are \emph{solvable at a boundary}. This gives a uniform statement of the
classical Turrittin-Levelt-Hukuhara decomposition as well as a strong $p$-adic analogue.

Note that for differential modules on an open annulus,
one can equally well discuss solvability at the inner boundary or the outer boundary.
In \cite{kedlaya-book} and other literature, it is typical to consider outer boundaries because one has in mind
the boundary of a residue disc. However, in this paper we will mostly need to consider inner boundaries
(see \S\ref{subsec:solvable}),
so we will set notation to address that case.

\begin{hypothesis}
Throughout \S\ref{sec:solvable}, 
fix $\alpha > 0$ and put
\[
\calR_\alpha = \bigcup_{\beta > \alpha} R_{(\alpha,\beta)},
\]
viewed as a differential ring for the derivation $d = \frac{d}{dt}$.
Let $M$ be a differential module over $\calR_\alpha$ which is \emph{solvable at $\alpha$} in the sense of 
Definition~\ref{D:solvable} below.
\end{hypothesis}

\begin{convention}
The functions $f_i(M,r)$ and $F_i(M,r)$ are not well-defined for any particular $r < -\log \alpha$;
however, the germs of these functions in left neighborhoods of $- \log \alpha$ may be interpreted 
unambiguously. We will use these germs frequently in what follows.
\end{convention}

\begin{defn} \label{D:solvable}
The module $M$ is \emph{solvable at $\alpha$} if 
\[
\lim_{r \to (-\log \alpha)^-} f_1(M, r) = -\log \alpha.
\]
By Proposition~\ref{P:variation} plus an extra argument (see \cite[Lemma~12.6.2]{kedlaya-book}),
this implies that there exist nonnegative rational numbers $b_1(M) \geq \cdots \geq b_n(M)$ such that
at the level of germs, we have
\begin{equation} \label{eq:slope condition}
f_i(M,r) = r + b_i(M)(-\log \alpha - r) \qquad (i=1,\dots,n).
\end{equation}
We call the $b_i(M)$ the \emph{slopes} of $M$.
\end{defn}

\begin{defn}
Suppose that $M \neq 0$.
We say $M$ satisfies the \emph{Robba condition} if $b_1(M) = 0$.
We say $M$ is \emph{pure} if $b_1(M) = \cdots = b_{\rank(M)}(M)$.
We say $M$ is \emph{refined} if $b_1(M) > b_1(\End(M))$; this implies that $M$ is pure.
We say $M$ is of \emph{cyclic type} if $b_1(\End(M)) = 0$; this implies that $M$ either is refined or satisfies the Robba condition.

By \eqref{eq:slope condition} plus 
Lemma~\ref{L:spectral decomposition exists}, $M$ admits a unique decomposition
$\bigoplus_j M_j$ into pure summands of distinct slopes; we call this the \emph{spectral decomposition} of $M$. By the \emph{Robba component} of $M$, we mean the summand of the spectral decomposition of slope $0$,
or the zero submodule of $M$ if no such summand exists.
\end{defn}

\begin{defn} \label{D:irreducible solvable relation}
Define a binary relation on irreducible solvable differential modules
over $\calR_\alpha$ by declaring that $M \sim N$ if at least one slope of $M^\dual \otimes N$ is nonzero. This relation is evidently reflexive and symmetric; it is also transitive by Lemma~\ref{L:relation transitive} below.
\end{defn}

\begin{lemma} \label{L:relation transitive}
With notation as in Definition~\ref{D:irreducible solvable relation}, the relation $\sim$ is transitive.
\end{lemma}
\begin{proof}
Suppose $M \sim N$ and $N \sim P$. Let $S,T$ be the Robba components of $M^\dual \otimes N$, $N^\dual \otimes P$. Since $N$ is irreducible and $S,T \neq 0$, the images of the elements of $S \subseteq \Hom_{\calR_{\alpha}}(M, N)$ span $N$ and the kernels of the elements of $T \subseteq \Hom_{\calR_{\alpha}}(N,P)$ have zero intersection in $N$. It follows that the image of $S \otimes T$ under the contraction map $M^\dual \otimes N \otimes N^\dual \otimes P \to M^\dual \otimes P$ is nonzero; this image satisfies the Robba condition. This proves the claim.
\end{proof}

\begin{lemma} \label{L:no H1}
If $b_{\rank(M)}(M) > 0$, then $H^1(M) = 0$.
\end{lemma}
\begin{proof}
Each element of $H^1(M)$ corresponds to an extension $0 \to M \to N \to \calR_\alpha \to 0$ of differential modules, but any such extension is split by the spectral decomposition of $N$.
\end{proof}

\begin{prop} \label{P:solvable decomposition}
There exists a unique direct sum decomposition $M = \bigoplus_i M_i$ such that for any irreducible subquotients $P,Q$ of $M_i, M_j$, we have $P \sim Q$ in the sense of Definition~\ref{D:irreducible solvable relation} if and only if $i=j$.
\end{prop}
\begin{proof}
It suffices to check that if $P,Q$ are inequivalent irreducible solvable differential modules over $\calR_{\alpha}$, then $H^1(P^\dual \otimes Q) = 0$. But this is immediate from Lemma~\ref{L:no H1}.
\end{proof}

\begin{lemma} \label{L:cyclic type2}
Suppose either that:
\begin{enumerate}
\item[(a)]
$p=0$ and $M$ is refined; or
\item[(b)]
$p>0$, $K$ is algebraically closed, $M$ is refined, and $\dim(M)$ is not divisible by $p$; or
\item[(c)]
$p>0$, $K$ is algebraically closed, $M$ is of cyclic type, and $b_1(M) > 0$.
\end{enumerate}
Then there exists a differential module $N$ over $\calR_\alpha$ which is free of rank $1$, is solvable at $\alpha$, and satisfies
$b_1(N^\dual \otimes M) < b_1(M)$.
\end{lemma}
\begin{proof}
Realize $M$ as a refined differential module over $R_{(\alpha,\beta)}$ for some $\beta > \alpha$.
By Theorem~\ref{T:cyclic type}, $b_1(M)$ is a positive integer.
We may thus imitate the proof of \cite[Theorem~12.7.2]{kedlaya-book} as follows.

In case (a), we may apply Lemma~\ref{L:test modules2} to construct $N_{\lambda,h,e,m}$ with
$IR((N_{\lambda,h,e,m}^\dual \otimes M)_\rho) < IR(M_\rho)$ for $\rho$ in some interval;
because $b_1(M) \in \ZZ$, we are forced to take $e = m = 1$.
By Lemma~\ref{L:compute test modules}, $N_{\lambda,h,1,1}$ is solvable at $\alpha$.
It remains to check that we may take $\lambda$ in $K$, not just in a finite extension of $K$; for this,
we argue as in Proposition~\ref{P:Turrittin analogue}.
Put $n = \rank(M)$.
Choose a generator $\bv$ of the restriction of $\wedge^n M$ to $R_I$ for some closed interval $I$,
and write $D(\bv) = a \bv$ with $a \in R_I$. 
Let $M'$ be the differential module over $R_I$ on the single generator $\bw$ with
$D(\bw) = (a/n) \bw$; then $(M')^{\otimes n}$ is isomorphic to the restriction of $\wedge^n M$ to $R_I$.
It follows that $\left|a/n - \lambda t^{h-1}\right|_\rho < \left|a/n\right|_\rho = 
\left| \lambda t^{h-1} \right|_\rho$ for $\rho \in I$, so there must exist $\lambda' \in K$
with $\left| \lambda - \lambda'\right| < \left| \lambda \right| = \left| \lambda' \right|$.
We may thus replace $N_{\lambda,h,1,1}$ with $N_{\lambda',h,1,1}$ without affecting the preceding arguments.

In cases (b) and (c), 
by taking global Frobenius antecedents (Proposition~\ref{P:global Frobenius antecedent})
as needed, we can ensure that
there exist $\gamma,\delta$ with $\alpha < \gamma < \delta < \beta$ such that
$IR(M_\rho) > \omega$ for $\rho \in [\gamma,\delta]$.
By Lemma~\ref{L:rank 1 test modules}, we obtain the desired module $N$.
\end{proof}

\begin{cor} \label{C:cyclic type2}
Suppose either that:
\begin{enumerate}
\item[(a)]
$p=0$ and $M$ is indecomposable and refined; or
\item[(b)]
$p=0$ and $M$ is of cyclic type; or 
\item[(c)]
$p>0$, $K$ is algebraically closed, $M$ is indecomposable and refined, and $\dim(M)$ is not divisible by $p$; or
\item[(d)]
$p>0$, $K$ is algebraically closed, $M$ is of cyclic type, and $b_1(M) > 0$.
\end{enumerate}
Then there exists a factorization $M \cong N \otimes P$
in which $N$ is free of rank $1$ and $b_1(P) = 0$. In particular, $M$ is of cyclic type.
\end{cor}
\begin{proof}
This follows by repeated application of Lemma~\ref{L:cyclic type2}.
Note that since $b_1(M) \in \ZZ$
by Theorem~\ref{T:cyclic type}, only finitely many iterations are needed before $b_1(M)$ is reduced to 0.
\end{proof}

When $p=0$, the structure of solvable modules is relatively simple.
\begin{theorem} \label{T:turrittin0}
Assume $p=0$. Then there exist a finite extension $K'$ of $K$ and a positive integer $m$ such that
$M \otimes_{K[t]} K'[t^{1/m}]$ admits a direct sum decomposition in which
each summand is of cyclic type.
\end{theorem}
\begin{proof}
This follows from Theorem~\ref{T:refined decomposition} and Corollary~\ref{C:cyclic type2}.
\end{proof}

\begin{remark}
By taking $K = \CC$ with the trivial norm,
we may deduce from Theorem~\ref{T:turrittin0} the usual Turritin-Levelt-Hukuhara decomposition theorem
for differential modules over $\CC((t))$ \cite[Theorem~7.5.1]{kedlaya-book}.
\end{remark}

\begin{defn}
Put $F = \Frac(\calR_\alpha)$.
Let $[M]$ denote the Tannakian subcategory generated by $M$ within the category of differential modules
over $\calR_\alpha$, equipped with the fibre functor $\omega$ taking each $N \in [M]$ to the $F$-vector space $N \otimes_{\calR_\alpha} F$.
Note that the objects of $[M]$ are all solvable at $\alpha$.

Let $G(M)$ be the automorphism group of $\omega$.
For $r \geq 0$, let $G^r(M)$ denote the subgroup of $G(M)$ which acts trivially
on $\omega(N)$ for each nonzero $N \in [M]$ for which $b_1(N) < r$.
Also put $G^{r+}(M) = \bigcup_{s>r} G^s(M)$.
\end{defn}

\begin{remark} \label{R:torus}
As in Remark~\ref{R:Tannakian easy}, we may use Theorem~\ref{T:turrittin0} to deduce that when $p=0$,
the group $G^{0+}(M)$ is a torus. The structure of $G^{0+}(M)$ in case $p>0$ will be clarified by 
Theorem~\ref{T:finite levels} below; this will imply that for any $p$ and any $r \geq 0$,
$G^{r+}(M)$ equals the subgroup of $G(M)$ which
acts trivially on $\omega(N)$ for each nonzero $N \in [M]$ for which $b_1(N) \leq r$.
\end{remark}

\begin{lemma} \label{L:cyclic type1}
If $p>0$ and $M$ is of cyclic type, then there exists a nonnegative integer $h$ such that
$b_1(M^{\otimes p^h}) = 0$.
\end{lemma}
\begin{proof}
If $b_1(M) > 0$,
then by Proposition~\ref{P:refined power},
we have $b_1(M^{\otimes p}) < b_1(M)$.
Since $b_1(M)$ and $b_1(M^{\otimes p})$ are nonnegative integers by Theorem~\ref{T:cyclic type},
this proves the claim.
\end{proof}

\begin{theorem} \label{T:finite levels}
If $p>0$, then $G^{0+}(M)$ is a finite $p$-group.
\end{theorem}
\begin{proof}
This follows from Proposition~\ref{P:finite Tannakian} using Remark~\ref{R:finite Tannakian} as follows.
Replace the category of differential modules over $\calR_\alpha$ with the direct limit of the categories of differential
modules over $\calR_\alpha \otimes_{K[t]} K'[t^{1/m}]$ over all finite extensions $K'$ of $K$ and all
positive integers $m$ not divisible by $p$; this does not change the groups $G^{r}(M)$ except for a base extension.
We may then deduce conditions (i), (ii), (iii) of Remark~\ref{R:finite Tannakian} using
Theorem~\ref{T:refined decomposition}, Proposition~\ref{P:refined power}, Lemma~\ref{L:cyclic type1}, respectively.
\end{proof}
\begin{cor}
There exist a finite extension $K'$ of $K$ and a positive integer $m$ such that
for all nonnegative integers $g,h$,
$(M^\dual)^{\otimes g} \otimes M^{\otimes h} \otimes_{K[t]} K'[t^{1/m}]$ admits a refined decomposition.
\end{cor}
\begin{proof}
This is apparent from Theorem~\ref{T:turrittin0} if $p=0$. If $p>0$,
for each pair $(g,h)$ we may choose a suitable $m$ by Theorem~\ref{T:refined decomposition},
so we need only check that $m$ may be chosen uniformly. But this follows from
Theorem~\ref{T:finite levels}: it is enough to list each of the finitely many
isomorphism classes of irreducible representations $\tau$ of $G^{0+}(M)$
and, for each $\tau$, ensure that $m$ works for one pair $g,h$ such that $\tau$ appears in $(M^\dual)^{\otimes g} \otimes M^{\otimes h}$.
\end{proof}

\begin{cor} \label{C:cyclic type constituent}
If $p>0$ and $b_1(M) > 0$, then there exist
a finite extension $K'$ of $K$, a positive integer $m$
and an object $N \in [M \otimes_{K[t]} K'[t^{1/m}]]$ of cyclic type such that
$b_1(N) > 0$ but $b_1(N^{\otimes p}) = 0$.
\end{cor}
\begin{proof}
This follows from Lemma~\ref{L:isolate character} plus the proof of Theorem~\ref{T:finite levels} (in which it is shown that the conditions of Remark~\ref{R:finite Tannakian} are satisfied).
\end{proof}

\begin{lemma} \label{L:order p case}
Suppose that $p>0$, $K$ contains a primitive $p$-th root of unity,
$M$ is free of rank $1$, and $b_1(M^{\otimes p}) = 0$.
Then there exists another differential module $N$ over $\calR_\alpha$ which is solvable on $\alpha$,
is free on a single 
generator $\bv$ such that $D(\bv) = P'(t)$ for some $P(t) \in K[t]$
with $|P(t)|_\alpha = \omega$,
and satisfies $b_1(N^\dual \otimes M) = 0$.
\end{lemma}
\begin{proof}
This follows from \cite[Theorem~17.1.6, Remark 17.1.7]{kedlaya-book}.
\end{proof}

\begin{defn} \label{D:extend Robba}
Let $\calR_\alpha^{\bd}$ be the subring of $\calR_\alpha$ consisting of germs of 
bounded analytic functions. This ring is henselian but not complete for the $\alpha$-Gauss norm;
let $\calR_\alpha^{\inte}$ denote the valuation subring.

If $S$ is a connected finite \'etale cover, it makes sense to impose the Robba condition on 
$M \otimes_{\calR^{\inte}_\alpha} S$ provided that $S$ can be identified with
a ring of the form $\calR^{\inte}_{\alpha}$ in a suitable power series coordinate;
the resulting condition will not depend on the choice of this identification. Such an identification
can always be made if $\kappa_K$ is algebraically closed.
\end{defn}

\begin{theorem} \label{T:p-adic turrittin}
If $p>0$, then there exists a connected finite \'etale cover $S$ of $\calR^{\inte}_\alpha$ such that
$M_S = M \otimes_{\calR^{\inte}_\alpha} S$ satisfies the Robba condition in the sense of
Definition~\ref{D:extend Robba}.
\end{theorem}
\begin{proof}
Since $G^{0+}(M)$ is finite by Theorem~\ref{T:finite levels} and is trivial if and only if $M$ satisfies
the Robba condition, it suffices to produce a cover that decreases $G^{0+}(M)$.
This may be achieved as follows. We may assume from the outset that $K$ contains
an element $\pi$ with $\pi^{p-1} = -p$; this also forces $K$ to contain a primitive $p$-th root of unity.
Pick out an object $N \in [M \otimes_{K[t]} K'[t^{1/m}]]$ for some $K',m$ as in 
Corollary~\ref{C:cyclic type constituent}.
Apply Corollary~\ref{C:cyclic type2}
to produce a free rank 1 object $N'\in [M \otimes_{K[t]} K'[t^{1/m}]]$ for some $K',m$ such that $N^\dual \otimes N'$ satisfies the Robba condition.
By Lemma~\ref{L:order p case},
we may choose $N'$ to be free on one generator $\bv$ satisfying $D(\bv) = 
P'(t)$ for some $P \in K[t]$ with $\left|P(t)\right|_\alpha = \omega$.
We may then trivialize $N'$ by extending scalars from
$\calR^{\inte}_\alpha$ to $\calR^{\inte}_\alpha[z]/(z^p - z - \pi^{-1} P(t))$ and recalling that the power series
$\exp(\pi(z^p - z))$ in $z$ has radius of convergence strictly greater than $1$
(see for example \cite[Example~9.9.3]{kedlaya-book}).
\end{proof}

\begin{cor} \label{C:turrittin ramification}
Assume that $p>0$, $\kappa_K$ is algebraically closed, and $\alpha=1$.
\begin{enumerate}
\item[(a)]
There is a unique minimal choice of $S$ satisfying the conclusion of Theorem~\ref{T:p-adic turrittin}.
\item[(b)]
The residue field of $S$ is a finite Galois extension of $\kappa_K((t))$ whose highest ramification break
is equal to $b_1(M)$.
\end{enumerate}
\end{cor}
\begin{proof}
This follows from Theorem~\ref{T:p-adic turrittin} as in the proof of \cite[Theorem~5.23]{kedlaya-overview}
(see also \cite[Theorem~19.4.1]{kedlaya-book}).
\end{proof}

\begin{cor} \label{C:turrittin factorization}
Assume that $p>0$,
and decompose $M = \bigoplus M_i$ as in Proposition~\ref{P:solvable decomposition}. Then for each $i$, there exists an isomorphism $M_i \cong N \otimes P$ for some solvable differential modules $N,P$ over $\calR_\alpha$ such that $N$ is irreducible, $N_S$ is trivial for some connected finite \'etale cover $S$ of $\calR_\alpha^{\inte}$, and $P$ satisfies the Robba condition.
\end{cor}
\begin{proof}
Let $Q$ be an irreducible subquotient of $M_i$.
By Theorem~\ref{T:p-adic turrittin}, we may choose $S$ so that 
$M_i \otimes_{\calR^{\inte}_{\alpha}} S$ satisfies the Robba condition,
as then does $Q \otimes_{\calR^{\inte}_{\alpha}} S$.
Let $T$ be the restriction of scalars of $\calR \otimes_{\calR^{\inte}_{\alpha}} S$ to $\calR$, viewed as a solvable differential module; then $Q \otimes T$ has a nontrivial Robba component, so $Q$ is equivalent to some irreducible subquotient $N$ of $T$. Let $P$ be the Robba component of $N^\dual \otimes M_i$; by construction, there is a natural map $N \otimes P \to M_i$ 
factoring through the contraction $N^\dual \otimes N \otimes M_i \to M_i$.
We may check that this map is an isomorphism by induction on the length of a Jordan-H\"older filtration of $M_i$.
\end{proof}

\begin{cor} \label{C:successive extension}
Suppose that $p>0$ and that the Robba component of $\End(M)$ has $p$-adic non-Liouville exponents.
Then for $S$ as in Theorem~\ref{T:p-adic turrittin},
$M_S$ splits as a direct sum, each summand of which is a successive extension of copies of $M_\lambda$ for some $\lambda \in \ZZ_p$.
\end{cor}
\begin{proof}
We may assume that $M$ is indecomposable. By Corollary~\ref{C:turrittin factorization}, we may write $M \cong N \otimes P$ where $N$ is irreducible, $N_S$ is trivial, and $P$ satisfies the Robba condition. We then have $\End(M) \cong \End(N) \otimes \End(P)$
and hence $\End(M_S) \cong \End(N_S) \otimes \End (P_S)$. 
Choose an exponent $A$ of $P$; then $A-A$ is an exponent of $\End(P)$, and the multiset obtained from $A-A$ by multiplying
each multiplicity by $\rank(N)^2$ is an exponent of $\End(M_S)$.
On the other hand, since $\End(N)$ contains a nontrivial Robba component (namely the trace component), $\End(P)$ is isomorphic to a submodule of the Robba component of $\End(M)$. Therefore $\End(P)$ has $p$-adic non-Liouville exponents, as then does $\End(M_S)$.
By Corollary~\ref{C:cm2}, $M$ has the desired form.
\end{proof}
\begin{remark}
In Corollary~\ref{C:successive extension}, it is not true in general that the differences between the different values of $\lambda$ are $p$-adic Liouville numbers. That is because if $M$ splits nontrivially 
as in Proposition~\ref{P:solvable decomposition}, then $M_i^\dual \otimes M_j$ has no Robba component and thus imposes no restriction on the exponents of $(M_i^\dual \otimes M_j)_S$. For instance, 
choose inequivalent irreducible solvable differential modules $N_i, N_j$
over $\calR_{\alpha}$ with $N_{i,S}, N_{j,S}$ trivial, and choose  $\lambda, \mu \in \ZZ_p$ which differ by a $p$-adic Liouville number. Then 
\[
M = (N_i \otimes M_{\lambda}) \oplus (N_j \otimes M_{\mu})
\]
satisfies the hypothesis of Corollary~\ref{C:successive extension}
but $\End(M)$ admits an exponent containing $\lambda - \mu$.
\end{remark}

\begin{remark} \label{R:local monodromy}
Theorem~\ref{T:p-adic turrittin} includes a result variously known as the \textit{$p$-adic Turrittin theorem} (the implicit analogy being perhaps most clear from Corollary~\ref{C:turrittin factorization})
and the \emph{$p$-adic local monodromy theorem}. That result, due to Andr\'e \cite{andre-monodromy},
Mebkhout \cite{mebkhout-monodromy}, and the author \cite{kedlaya-monodromy},
assumes the existence of a \emph{Frobenius structure} on $M$
(see \cite[Chapter~17]{kedlaya-book}); in addition, $K$ must be discretely valued
and $\beta$ must equal 1. 

The methods of Andr\'e and Mebkhout can be used to derive Theorem~\ref{T:p-adic turrittin} also in the case
where all of the objects in $[M]$ have $p$-adic non-Liouville exponent differences.
In these arguments, the non-Liouville condition is needed to ensure that irreducible objects satisfying the Robba condition are all of rank 1. The proof of Theorem~\ref{T:p-adic turrittin} provides a workaround
in cases where advance information about exponents is not available.
\end{remark}

\section{Berkovich discs}
\label{sec:Berkovich discs annuli}

We are at last ready to shift language and perspective towards Berko\-vich's nonarchimedean analytic spaces.
In this section, we introduce the topological spaces which play the role of discs in Berkovich's theory,
and consider radii of convergence of local horizontal
sections of differential modules on such spaces.
This draws heavily on the results of
\S \ref{sec:rings discs annuli}, but some additional maneuvering is needed.
In addition, the behavior of differential
modules around points of type 4 requires some extra work.

\subsection{Underlying topological spaces}

We begin by defining the Gel'\-fand spectrum of a Banach ring. 
For now, we just consider the resulting topological space; we postpone discussion of the 
analytic space structure to \S\ref{sec:berkovich curves}.

\begin{defn}
For $R$ a ring equipped with a submultiplicative norm (e.g., a commutative Banach algebra over $K$), the \emph{Gel'fand spectrum} $\calM(R)$ is defined as the set
of bounded (by the given norm) multiplicative seminorms on $R$, topologized as a subset of the product $\RR^R$.
Note that $\calM(R)$ may also be viewed as a closed subset of a product of bounded closed intervals,
and hence is compact; it is also nonempty provided that $R \neq 0$ \cite[Theorem~1.2.1]{berkovich1}.
For $x \in \calM(R)$, let $\calH(x)$ denote the completion of $\Frac(R/\ker(x))$ for the multiplicative
norm induced by $x$.
\end{defn}

\begin{remark}
Any bounded homomorphism $R \to S$ of commutative Banach algebras over $K$ defines a continuous restriction map
$\calM(S) \to \calM(R)$. If this map is surjective, then it is a quotient map because the source and target are
compact: the induced map from the quotient space is a continuous bijection from a quasicompact space to a Hausdorff space, hence a bijection \cite[\S 9, No. 4, Corollaire~2]{bourbaki-top}.

For example, suppose that $R$ is a commutative Banach algebra over $K$ and that $K'$ is a complete field extension of $K$.
Then the completed tensor product $R' = R \widehat{\otimes}_K K'$ is a Banach algebra over $K'$ and the restriction map $\calM(R') \to \calM(R)$ is always surjective
\cite[Lemma~1.20]{kedlaya-witt}.

In the previous paragraph, if $K'$ is the completion of an algebraic Galois extension of $K$
(such as $\CC$), we can say more: not only is the restriction map $\calM(R') \to \calM(R)$ surjective,
but the group of continuous automorphisms of $K'$ over $K$ acts transitively on the fibres of the restriction map.
See \cite[Corollary~1.3.6]{berkovich1}.
\end{remark}

\begin{defn} \label{D:point type}
Let $R$ be a commutative Banach algebra over $K$, and put $R' = R \widehat{\otimes}_K \CC$.
For $x \in \calM(R)$, choose any lift $\tilde{x} \in \calM(R')$ of $x$, and define the \emph{signature}
of $x$ as the triple
\[
\left( \dim (\ker(\tilde{x})), \rank (|\calH(x)^\times|/|K^\times|), \trdeg(\kappa_{\calH(x)}/\kappa_K) \right).
\]
Note that one can have $\dim(\ker(\tilde{x})) > \dim(\ker(x))$.
\end{defn}

\subsection{Discs}
\label{sec:geometry discs}

We now specialize the previous discussion to rings of convergent power
series on discs. Due to the increasing prevalence of such rings and their associated Gel'fand spectra
in various branches of mathematics, numerous expositions of this material can be found in the literature;
among these, perhaps the most comprehensive is
the book of Baker and Rumely \cite[Chapter~1]{baker-rumely}.
However, that treatment assumes that the ground field $K$ is algebraically closed, which we prefer not to
do here; to avoid imposing this condition, we refer also to \cite[\S 2]{kedlaya-witt}.
\begin{defn}
For $\beta > 0$, the space $\calM(R_{[0,\beta]})$ is called the 
\emph{Berkovich closed disc of radius $\beta$} with coordinate $t$ over $K$, and also denoted
$\DD_{\beta,K}$.
For $z \in \CC$ with $|z| \leq \beta$ and $\rho \in [0,\beta]$, 
the restriction to $R_I \cong K \langle t/\beta \rangle$ of the $\rho$-Gauss
norm on $\CC\langle (t-z)/\beta\rangle$ defines a point $\zeta_{z,\rho} \in \DD_{\beta,K}$; the point $\zeta_{0,\beta}$ is called the 
\emph{Gauss point} of $\DD_{\beta,K}$. For $\beta' > \beta$, the natural map $R_{[0,\beta']} \to R_{[0,\beta]}$
induces an inclusion $\DD_{\beta,K} \to \DD_{\beta',K}$; the direct limit of the $\DD_{\beta,K}$ along these maps
is called the \emph{Berkovich affine line} over $K$.
\end{defn}

\begin{lemma} \label{L:disc points quotient}
The restriction map $\DD_{\beta,\CC} \to \DD_{\beta,K}$ identifies $\DD_{\beta,K}$ with the quotient
of $\DD_{\beta,\CC}$ by the action of
the group of continuous automorphisms of $\CC$ over $K$.
\end{lemma}
\begin{proof}
See \cite[Proposition 1.3.5]{berkovich1}.
\end{proof}

\begin{prop} \label{P:homotopy}
For $\beta > 0$, $x \in \DD_{\beta,K}$ and $\rho \in [0,\beta]$, define
\begin{equation} \label{eq:homotopy}
H(x,\rho)(f) = \max \left\{ \rho^i x \left( \frac{1}{i!} \frac{d^i}{dt^i}(f) \right): i = 0,1,\dots \right\}
\end{equation}
with the interpretation that $\rho^0 = 1$ even for $\rho=0$.
\begin{enumerate}
\item[(a)]
The formula \eqref{eq:homotopy}
defines a continuous map 
\[
H: \DD_{\beta,K} \times [0, \beta] \to \DD_{\beta,K}.
\]
\item[(b)]
For $x \in \DD_{\beta,K}$, $H(x,0) = x$ and $H(x,\beta) = \zeta_{0,\beta}$.
\item[(c)]
For $x \in \DD_{\beta,K}$ and $\rho,\sigma \in [0,\beta]$, 
\[
H(H(x,\rho),\sigma) = H(x, \max\{\rho,\sigma\}).
\]
\item[(d)]
For $z \in \CC$ with $|z| \leq \beta$ and $\rho \in [0,\beta]$, $H(\zeta_{z,0},\rho) = \zeta_{z,\rho}$.
\item[(e)]
For $x,y \in \DD_{\beta,K}$, $y$ dominates $x$ (that is, $y(f) \geq x(f)$ for all $f \in R_I$)
if and only if $y = H(x,\rho)$ for some $\rho \in [0,\beta]$.
\end{enumerate}
\end{prop}
\begin{proof}
Ssee \cite[Remark~6.1.3(ii)]{berkovich1} or \cite[Lemma~2.3]{kedlaya-witt} for (a)-(d)
and \cite[Theorem~2.11]{kedlaya-witt} for (e).
\end{proof}

\begin{defn} \label{D:radius}
For $\beta>0$ and $x \in \DD_{\beta,K}$, define the \emph{diameter} of $x$, denoted $\rho(x)$, to be the maximum
$\rho \in [0,\beta]$ for which $H(x,\rho) = x$. 
Beware that the diameter is stable under base extension from $K$ to $\CC$
(see Proposition~\ref{P:Berkovich classification}), but 
not under general base extensions (see Remark~\ref{R:canonical base extension}).
It is also stable under increasing $\beta$.
\end{defn}

\begin{remark} \label{R:canonical base extension}
For $\beta>0$ and  $x \in \DD_{\beta,K}$, let $t_x \in \calH(x)$ be the image of $t$ under the natural map
$R_{[0,\beta]} \to \calH(x)$. We may then realize $x$ as the restriction of the seminorm $\zeta_{t_x,0} \in \calM(R_{I,\calH(x)})$ of radius 0.
\end{remark}

At the other extreme, we have the following.
\begin{lemma}\label{L:same radius}
For $\beta > 0$, $x \in \DD_{\beta,K}$, and $K'$ an analytic field containing $K$, there exists $y \in \DD_{\beta,K'}$ lifting $x$ with $\rho(y) = \rho(x)$.
\end{lemma}
\begin{proof}
In case $K' = \CC$, this will follow from Proposition~\ref{P:Berkovich classification} below. In case $K = \CC$, then the tensor product
norm on $\calH(x) \otimes_K K'$ is itself multiplicative
(see for instance \cite[3.14]{poineau}) and hence defines a point $y$ of the desired form.

In the general case, note that any lift $y$ satisfies $\rho(y) \leq \rho(x)$, so it suffices to check the claim after enlarging $K'$. We may thus ensure that $\CC \subseteq K'$ and then check the claim in two steps using the previous paragraph.
\end{proof}

In terms of the intrinsic radius function, Berkovich's classification of points of $\calM(R_I)$
reads as follows.
\begin{prop} \label{P:Berkovich classification}
For $\beta > 0$, every point of $\DD_{\beta,K}$ is of exactly one of the following types
(called \emph{types 1,2,3,4} hereafter).
\begin{enumerate}
\item[1.] Points of signature $(1,0,0)$. These are the points of the form $\zeta_{z,0}$ for some $z \in \CC$. The diameter of such a point is $0$.
\item[2.] Points of signature $(0,1,0)$. These are the points of the form $\zeta_{z,\rho}$ for some $z \in \CC$ and some $\rho \in (0,\beta] \cap |\CC^\times|$.
The diameter of such a point is $\rho>0$.
\item[3.] Points of signature $(0,0,1)$. These are the points of the form $\zeta_{z,\rho}$ for some $z \in \CC$ and some $\rho \in (0,\beta] \setminus |\CC^\times|$.
The diameter of such a point is $\rho >0$.
\item[4.] Points of signature $(0,0,0)$. The diameter of such a point $x$
is the infimum of those values of $\rho$
for which the seminorm $x$ is dominated by some $\zeta_{z,\rho}$; it belongs to the interval
$(0,\beta)$.
\end{enumerate}
Moreover, the points that are minimal under domination are precisely those of types 1 and 4.
\end{prop}
\begin{proof}
For $K = \CC$, see \cite[1.4.4]{berkovich1}. For the general case, see
\cite[Theorem~2.26]{kedlaya-witt}.
\end{proof}
This can be used to recover a version of the Zariski-Abhyankar inequality. For a more traditional
variant, see for instance \cite[Th\'eor\`eme~9.2]{vaquie}.

\begin{cor} \label{C:Abhyankar}
Let $R$ be the completion of $K[T_1,\dots,T_d]$ for the Gauss norm. Then the signature of each point in $\calM(R)$
consists of three nonnegative integers whose sum is at most $d$.
\end{cor}
\begin{proof}
Choose $x \in \calM(R)$ and let $x_i$ be the restriction of $R$ to the completion of $K[T_1,\dots,T_i]$.
Then the difference between the signatures of $x_{i+1}$ and $x_i$ is itself the signature of a point
in $\DD_{1,\calH(x_i)}$. The claim thus follows from Proposition~\ref{P:Berkovich classification}.
\end{proof}

We next make the topology of $\DD_{\beta,K}$ more explicit.
\begin{defn}
For $\beta > 0$ and $x \in \DD_{\beta,K}$, 
the \emph{root path} of $x$ is the subspace $\{H(x, \rho): \rho \in [0,\beta]\}$ of $\DD_{\beta,K}$.
It is homeomorphic to the interval $[\rho(x), \beta]$ via the map $H(x, \cdot)$.

A \emph{rooted skeleton} in $\DD_{\beta,K}$ is a subset of the form
\[
\bigcup_{i=1}^m \{H(x_i,\rho): \rho \in [\rho(x_i),\beta]\}
\]
for some nonempty finite subset $\{x_1,\dots,x_m \} \subseteq \DD_{\beta,K}$;
we sometimes say that this skeleton is 
\emph{generated} by $x_1,\dots,x_m$. A 
\emph{strict rooted skeleton} is a rooted skeleton generated by a set of points of type 2.

For any rooted skeleton $S$ of $\DD_{\beta,K}$, define the map $\pi_S: \DD_{\beta,K} \to S$ taking each
$x \in \DD_{\beta,K}$ to $H(x,\rho)$ for $\rho$ the least value in $[0,\beta]$ for which $H(x,\rho) \in S$.
By Proposition~\ref{P:homotopy}, $\pi_S$ is a deformation retract.
\end{defn}

\begin{prop} \label{P:inverse limit of skeleta}
Form the inverse system consisting of the rooted skeleta of $\DD_{\beta,K}$ with morphisms given as follows:
for every pair of rooted skeleta $S, S'$ with $S \subseteq S'$, include a morphism $S' \to S$ given by the restriction
of $\pi_S$. Define a map from $\DD_{\beta,K}$ to this inverse system whose projection onto $S$ is given by
$\pi_S$. Then this map is a homeomorphism of topological spaces.
\end{prop}
\begin{proof}
The map is injective because every pair of points can be found in some rooted skeleton. 
The map is surjective because $\DD_{\beta,K}$ is compact and surjects onto each rooted skeleton.
The map is a homeomorphism because any continuous bijection from a quasicompact space to a Hausdorff space is a
homeomorphism. (See also \cite[Proposition~1.13]{baker-rumely} for an alternate treatment in case
$K$ is algebraically closed and $\beta=1$.)
\end{proof}

\begin{remark}
Proposition~\ref{P:inverse limit of skeleta}
is a special case of the general phenomenon that Berkovich analytic spaces can be described 
as inverse limits of tropical spaces (see for example \cite{payne}).
For Berkovich curves, this inverse limit presentation is also closely related to semistable models;
we will return to this point in \S \ref{sec:berkovich curves}.
\end{remark}

\begin{defn}
For $x \in \DD_{\beta,K}$, a \emph{branch} of $\DD_{\beta,K}$ at $x$ is a path-connected component of $\DD_{\beta,K} \setminus \{x\}$.
If $x$ is not the Gauss point, then there is a branch containing the Gauss point, called the \emph{upper branch}
of $\DD_{\beta,K}$ at $x$. By Proposition~\ref{P:inverse limit of skeleta},
additional branches (called \emph{lower branches}) exist according to the type of $x$ as follows.
\begin{enumerate}
\item[1.] No lower branches.
\item[2.] Infinitely many lower branches.
\item[3.] Exactly one lower branch.
\item[4.] No lower branches. 
\end{enumerate}
For $S$ a rooted skeleton of $\DD_{\beta,K}$ and $x \in S$, a \emph{branch} of $S$ at $x$ is a branch of $X$
at $x$ meeting $S$. There are only finitely many such branches at any $x$.
\end{defn}

\begin{defn} \label{D:subdivision}
Let $S$ be a rooted skeleton of $\DD_{\beta,K}$. 
By a \emph{subdivision} of $S$, we will mean a graph (in the combinatorial sense) with underlying topological space
$S$.

We
equip $S$ with the piecewise linear structure 
characterized as follows: a function $f: S \to \RR$ is piecewise affine (with integral slopes) if and only if for each
$x \in S$, the function $r \mapsto f(H(x,e^{-r}))$ is piecewise affine (with integral slopes)
and constant for $r$ sufficiently large. Then for any piecewise affine function $f: S \to \RR$, there exists
a subdivision of $S$ such the restriction of $f$
to any edge of the subdivision is affine. We call such a subdivision a \emph{controlling graph} of $f$.

It is meaningful to refer to the \emph{slope} of a piecewise affine function $f: S \to \RR$
along a branch of $S$ at a point $x$. Explicitly, the slope along the upper branch is the left slope
of $r \mapsto f(H(x,e^{-r}))$ at $r_0 = -\log \rho(x)$ (or $0$ in case $\rho(x) = 0$),
while the slope along the lower branch containing $y \in S$ is the right slope of $r \mapsto f(H(y,e^{-r}))$
at $r_0$.
\end{defn}

\begin{defn}
By the \emph{Berkovich open unit disc of radius $\beta$ over $K$}, denoted $\DD_{\beta,K}^{\circ}$,
we will mean the branch of $\DD_{\beta,K}$
at the Gauss point containing $\zeta_{0,0}$.
\end{defn}

\subsection{Radii of convergence}
\label{sec:radii discs}

We now define the radii of optimal convergence for differential modules on discs,
following Baldassarri \cite{baldassarri}.

\begin{hypothesis} \label{H:radii discs}
Throughout \S\ref{sec:radii discs}, fix $\beta > 0$ and (except in 
Definition~\ref{D:open unit disc} and Lemma~\ref{L:variation0 convexity})
let $M$ be a differential module
of rank $n \geq 0$ over $R_{[0,\beta]}$.
\end{hypothesis}

\begin{defn} \label{D:radii of convergence}
For $x \in \DD_{\beta,K}$, put 
\begin{align*}
M_{x,0} &= M \otimes_{R_{[0,\beta]}} \calH(x)\llbracket t-t_x \rrbracket, \\
\qquad 
M_{x,\rho} &= M \otimes_{R_{[0,\beta]}} \calH(x)\langle (t-t_x)/\rho\rangle
\quad (\rho \in (0,\beta);
\end{align*}
these can be viewed as differential modules as well.
By a standard argument
(see for instance \cite[Theorem~7.2.1]{kedlaya-book}),
the natural map 
\[
M_{x,0}^{D=0} \otimes_{\calH(x)} \calH(x)\llbracket t-t_x \rrbracket \to 
M_{x,0}
\]
is an isomorphism. Define the sequence $s_i(M, x)$ of \emph{radii of optimal convergence}
of $M$ at $x$ as follows: for $i =1,\dots,n$, put
\[
s_i(M, x) = \sup\{\rho \in [0,\beta): \dim_{\calH(x)} (M_{x,0}^{D=0} \cap M_{x,\rho}) \geq n-i+1\}.
\]
In other words, $s_i(M, x)$ is the radius of the maximal open disc around $t_x$ on which there exist
$n-i+1$ linearly independent horizontal sections of $M$.
For $M \neq 0$, we refer to $s_1(M, x)$ also as the \emph{radius of convergence}
of $M$ at $x$.
\end{defn}

\begin{lemma} \label{L:base change radii}
Let $K'$ be an analytic field containing $K$, and suppose $y \in \DD_{\beta,K'}$ restricts to $x \in \DD_{\beta,K}$.
Then 
\[
s_i(M, x) = s_i(M \otimes_{R_{[0,\beta],K}} R_{[0,\beta],K'}, y) \qquad (i=1,\dots,n).
\]
\end{lemma}
\begin{proof}
By replacing $K$ with $\calH(x)$, we may reduce to the case $x = \zeta_{0,0}$.
The claim then comes down to the fact that formation of the kernel of the bounded $K$-linear endomorphism
of the Banach space $M \otimes_{R_{[0,\beta]}} R_{[0,\rho]}$
commutes with formation of the completed tensor product over $K$ with $K'$.
This in turn reduces formally to the case where $K'$ is the completion of a countably generated field extension of $K$, in which case the claim is clear because $K'$ admits a Schauder basis over $K$ (see \cite[Proposition~2.7.2/3]{bgr} or \cite[Lemma~1.3.8]{kedlaya-book}).
\end{proof}

\begin{remark} \label{R:grow sections}
The intuition behind Definition~\ref{D:radii of convergence} is that the
elements of $M_{x,0}^{D=0}$ are the formal horizontal sections of $M$ centered at $x$.
In the language of \cite{kedlaya-book} and preceding literature on $p$-adic differential equations, one would think of $x$
as the generic point of a certain subdisc of $\DD_{\beta,K}$.

Following this intuition, one observes that for $y = H(x,\sigma)$ for some $\sigma > \rho(x)$,
the discs of radius $\rho$
centered at $x$ and $y$ coincide for all $\rho \in (\sigma,\beta)$. Formally, for any field $L$
containing both $\calH(x)$ and $\calH(y)$, we obtain a natural isomorphism
$L \langle (t-t_x)/\rho \rangle \cong L \langle (t-t_y)/\rho \rangle$.
One consequence is that for $i \in \{1,\dots,n\}$, if $s_i(M,x) > \rho(x)$, then $s_i(M,x) = s_i(M,H(x,\rho))$ for all 
$\rho < s_i(M,x)$.
\end{remark}

The relationship between radii of optimal convergence and intrinsic subsidiary radii (due in its original
form to Young) is the following.
\begin{defn} \label{D:differential field}
For $x \in \DD_{\beta,K}$ not of type $1$,
let $F_x$ be a copy of $\calH(x)$ viewed as a differential field for the derivation $\frac{d}{dt}$.
\end{defn}

\begin{lemma} \label{L:same spectral norm}
For any $x \in \DD_{\beta,K}$ not of type $1$, any analytic field $K'$ containing $K$, and any $y \in \DD_{\beta,K'}$ lifting $x$ with
$\rho(y) = \rho(x)$ (which exists by Lemma~\ref{L:same radius}),
the spectral norms of $\frac{d}{dt}$ on $F_x$ and $F_y$ coincide.
\end{lemma}
\begin{proof}
Since $F_x \subseteq F_y$, the spectral norm of $\frac{d}{dt}$ on $F_x$ is no greater than that on $F_y$. To prove the reverse inequality, 
by Lemma~\ref{L:same radius} we are free to enlarge $K'$.
We may thus reduce to the cases where $K = \CC$ and where $K' = \CC$.
In the former case, we have $F_y = F_x \widehat{\otimes}_K K'$ with the tensor product norm (see the proof of Lemma~\ref{L:same radius}), so the desired inequality is clear.

To treat the latter case, it is sufficient to instead consider the case where $K'$ is a finite extension of $K$. In this case, $F_y$ is a direct summand of $F_x \otimes_K K'$, so the desired inequality is again clear.
\end{proof}

\begin{prop} \label{P:compare radii}
For $x \in \DD_{\beta,K}$ not of type 1, the intrinsic subsidiary radii of $M \otimes_{R_{[0,\beta]}} F_x$ are given by
\[
\min\{1, s_i(M,x) / \rho(x)\} \qquad (i=1,\dots,n).
\]
\end{prop}
\begin{proof}
By Lemma~\ref{L:base change radii} and Lemma~\ref{L:same spectral norm}, we are free to lift $x$ as long as we do not change its diameter. This lifting being possible by Lemma~\ref{L:same radius},
we may reduce to the case $x = \zeta_{0,\rho}$ for some $\rho \in (0,\beta]$. In this case, the claim follows from
\cite[Theorem~11.9.2]{kedlaya-book}.
\end{proof}

One can also interpret Dwork's transfer theorem in this language.
\begin{prop} \label{P:transfer}
For $M$ nonzero, for all $x \in \DD_{\beta,K}$ and $\rho \in [0,\beta]$,
\[
s_1(M, H(x,\rho)) \leq s_1(M,x).
\]
\end{prop}
\begin{proof}
Using Lemma~\ref{L:base change radii}, we may reduce to the case where $x = \zeta_{0,0}$,
in which case the claim asserts that $s_1(M, \zeta_{0,\rho}) \leq s_1(M, \zeta_{0,0})$ for any $\rho \in [0,\beta]$.
If $s_1(M, \zeta_{0,\rho}) > \rho$, then this follows from Remark~\ref{R:grow sections}.
If $s_1(M, \zeta_{0,\rho}) \leq \rho$, then by Proposition~\ref{P:compare radii}, 
$\rho^{-1} s_1(M, \zeta_{0,\rho})$ equals the intrinsic radius of $M \otimes_{R_{[0,\beta]}} F_\rho$,
so we may apply \cite[Theorem~9.6.1]{kedlaya-book} to conclude.
\end{proof}

\begin{remark} \label{R:positive radius}
The radius of convergence of $M$ at any $x \in \DD_{\beta,K}$ is always positive. This can be deduced either from
Proposition~\ref{P:transfer} or from Clark's $p$-adic Fuchs theorem \cite[Theorem~13.2.3]{kedlaya-book};
the latter also covers the case of a regular singularity with $p$-adic non-Liouville exponent differences.
\end{remark}

\begin{remark} \label{R:dual convergence}
For $M$ nonzero, the properties of the intrinsic radius described in Definition~\ref{D:intrinsic radius} carry over
to the radius of convergence, as follows.
\begin{enumerate}
\item[(a)]
We have $s_1(M^\dual, x) = s_1(M, x)$.
\item[(b)]
For any short exact sequence $0 \to M_1 \to M \to M_2 \to 0$, $s_1(M, x) = \min\{s_1(M_1,x), s_1(M_2,x)\}$.
\item[(c)]
For any $M_1, M_2$, $s_1(M_1 \otimes M_2,x) \geq \min\{s_1(M_1,x), s_1(M_2,x)\}$,
with equality if $s_1(M_1,x) \neq s_1(M_2,x)$.
\end{enumerate}
However, unlike for intrinsic subsidiary radii, these properties do not propagate to
radii of optimal convergence despite the validity of Proposition~\ref{P:compare radii}. 
The difficulty already appears in (a): the existence of a horizontal section of $M$ on a large open disc
does not imply the same for $M^\dual$. A similar difficulty arises for (b) unless we restrict consideration to
split exact sequences. No such difficulty arises for (c).
\end{remark}

So far we have consider only radii of convergence on closed discs, but one can make similar definitions
for open discs.
\begin{defn} \label{D:open unit disc}
For $M$ a differential module over $R_{[0,\beta)}$, we may similarly define $s_i(M,x)$ for $x \in \DD_{\beta,K}^{\circ}$. Then Proposition~\ref{P:transfer}
implies that for $M$ nonzero, for all $x \in \DD_{\beta,K}^\circ$,
\[
\limsup_{\rho \to \beta^-} s_1(M, H(x,\rho)) \leq s_1(M,x).
\]
\end{defn}

For open discs, we have the following key lemma.
\begin{lemma} \label{L:variation0 convexity}
Let $M$ be a differential module over $R_{[0,\beta)}$.
Suppose that for some $\gamma \in (0,\beta]$,
there exists $m \in \{0,\dots,n\}$ satisfying the following conditions.
\begin{enumerate}
\item[(a)]
For $i=1,\dots,m$, $s_i(M,\zeta_{0,\rho})$ is constant and less than $\gamma$ for $\rho$ in some punctured left neighborhood of $\gamma$.
\item[(b)]
For $i=m+1,\dots,n$, $\limsup_{\rho \to \gamma^-} s_i(M,\zeta_{0,\rho}) \geq \gamma$.
\end{enumerate}
Then the restrictions of the functions $s_i(M, \cdot)$ to $\DD_{\gamma,K}^{\circ}$ are constant for $i=1,\dots,n$.
\end{lemma}
\begin{proof}
%Put $\alpha = \rho(x) > 0$. 
Using Proposition~\ref{P:compare radii} to see that the appropriate hypotheses are satisfied,
we may decompose $M \otimes_{R_{[0,\beta)}} R_{[0,\gamma)} = M_0 \oplus M_1 \oplus \cdots$  as per Corollary~\ref{C:decompose disc}.

Consider any $k>0$.
By Corollary~\ref{C:decompose disc} and Proposition~\ref{P:compare radii},
for $i \in \{1,\dots,\rank(M_k)\}$, for all $y \in \DD_{\gamma,K}^{\circ}$ we have $\min\{\rho(y), s_i(M_k,y)\} = 
\min\{\rho(y), e^{-c_k}\}$. 
For those $y$ with $\rho(y) > e^{-c_k}$, we have 
\[
e^{-c_k} = \min\{\rho(y),e^{-c_k}\} = \min\{\rho(y), s_i(M_k,y)\}
\]
and the right side cannot equal $\rho(y)$, so we must have $s_i(M_k,y) = e^{-c_k}$.
For those $y$ with $\rho(y) \leq e^{-c_k}$, we cannot have $s_i(M_k,y) > e^{-c_k}$: otherwise, we could choose
$\delta \in (e^{-c_k}, s_i(M_k,y))$ and apply 
Remark~\ref{R:grow sections} to see that $s_i(M_k,H(y,\delta)) = s_i(M_k,y) > e^{-c_k}$,
contradicting the previously established equality $s_i(M_k,H(y,\delta)) = e^{-c_k}$.
We thus have $s_i(M_k,y) \leq e^{-c_k}$; on the other hand, for any $\delta \in (e^{-c_k},\gamma)$
we may apply Proposition~\ref{P:transfer} to obtain $e^{-c_k} = s_1(M_k,H(y,\delta)) \leq s_1(M_k,y) \leq s_i(M_k,y)$.
We conclude that $s_i(M_k,y)$ is constant for $y \in \DD_{\gamma,K}^{\circ}$.

For $i=1,\dots,n$ and $y \in \DD_{\gamma,K}^{\circ}$, we have
\begin{equation} \label{eq:truncate radius}
s_i(M \otimes_{R_{[0,\beta)}} R_{[0,\gamma)}, y) = \min\{s_i(M,y), \gamma\}.
\end{equation}
For $i=1,\dots,m$, we must have $s_i(M,y) < \gamma$ or else Remark~\ref{R:grow sections} would lead
to a violation of hypothesis (a);
moreover, from \eqref{eq:truncate radius} and the previous paragraph, $\min\{s_i(M,y), \gamma\}$ is constant on
$\DD_{\gamma,K}^{\circ}$.
We are thus done in case $m=n$,
so we may assume $m < n$ hereafter.

By Corollary~\ref{C:decompose disc}, Proposition~\ref{P:compare radii}, and Proposition~\ref{P:transfer}
(applied as in Definition~\ref{D:open unit disc}),
we have $s_1(M_0,y) \geq \gamma > \rho(y)$ for all $y \in \DD_{\gamma,K}^{\circ}$.
From this inequality plus \eqref{eq:truncate radius}, it follows that
for $i=m+1,\dots,n$, we have $s_i(M,y) \geq \gamma$ for all $y \in \DD_{\gamma,K}^{\circ}$.
If there exists $y \in \DD_{\gamma,K}^{\circ}$ for which $s_i(M,y) > \gamma$,
then by Remark~\ref{R:grow sections}, $s_i(M,y)$ is constant on $\DD_{\gamma,K}^{\circ}$;
otherwise, $s_i(M,y)$ is evidently equal to the constant value $\gamma$ on $\DD_{\gamma,K}^{\circ}$.
This completes the proof.
\end{proof}

\subsection{Solvable modules}
\label{subsec:solvable}

If one views solvability of a differential module on an annulus as a question about what happens
as one approaches the generic point of the inner boundary, one is then led to an analogous concept
in which one approaches an arbitrary point of a Berkovich disc. For points of type 2,
this amounts to a cosmetic revision of \S\ref{sec:solvable}, but at points of other types one has more precise
results. The case of type 4 points is especially critical in order to eliminate such points from the controlling graph of $M$ (see Theorem~\ref{T:strict skeleton}).

\begin{defn}
Choose $x \in \DD_{\beta,K}$ with $x \neq \zeta_{0,\beta}$, so that $\rho(x) < \beta$.
Put $r_0 = -\log \rho(x)$. 

For $\rho \in (\rho(x), \beta]$, the points $H(x,\rho)$ are all of types 2 and 3 (because they are not minimal).
Moreover, as $\rho \to \rho(x)^+$ these points form 
a net converging to $x$.

For any $\gamma,\delta$ with $\rho(x) < \gamma \leq \delta \leq \beta$, the subset of $\DD_{\beta,K}$
consisting of points dominated by $H(x,\rho)$ for some $\rho \in [\gamma,\delta]$ has the form
$\calM(R_{x,[\gamma,\delta]})$ for some Banach algebra $R_{x,[\gamma,\delta]}$ over $K$.
More precisely, this subset is an \emph{affinoid subdomain} of $\DD_{\beta,K}$ in the sense of
Definition~\ref{D:spaces}.
Even more precisely, if $K = \CC$ and $\gamma,\delta \in \left| \CC^\times \right|$, the set in question is an annulus.

For $\delta \in (\rho(x), \beta]$, define
\[
R_{x,(\rho(x), \delta]} = \bigcap_{\gamma \in (\rho(x), \delta]} R_{x,[\gamma,\delta]};
\]
it is equivalent to run the intersection over
$\gamma \in (\rho(x),\delta] \cap \left| \CC^\times \right|$.
Define the \emph{Robba ring at $x$} as the ring
\[
\calR_x = 
\bigcup_{\delta \in (\rho(x),\beta]} R_{x,(\rho(x),\delta]};
\]
it is equivalent to run the union over $\delta \in (\rho(x), \beta] \cap \left| \CC^\times \right|$.
All of these rings may be viewed as differential rings for the derivation $\frac{d}{dt}$.
\end{defn}

\begin{defn}
For $N$ a differential module of rank $n>0$ over $\calR_x$, the germ of the function
$-\log s_i(N, H(x,e^{-r}))$ in a left neighborhood of $r_0$ is well-defined for $i=1,\dots,n$.
We may thus say that $N$ is \emph{solvable at $x$} if 
\[
\limsup_{r \to r_0^-}
 -\log s_1(N, H(x,e^{-r}))-r \leq 0.
\]
In this case, as in Definition~\ref{D:solvable},
there exist
nonnegative rational numbers $b_1(N,x) \geq \dots \geq b_n(N,x)$ such that for $i=1,\dots,n$,
at the level of germs we have
\[
\max\{r, -\log s_i(N,H(x,e^{-r}))\} = r + b_i(N,x)(r_0 -r).
\]
\end{defn}

\begin{remark} \label{R:type 1 2 3}
For $x$ of type 2, after making a finite extension of $K$ to force $K$ to be integrally closed in $\calH(x)$,
we may obtain an isomorphism $\calR_x \cong \calR_\alpha$ for $\alpha = \rho(x)$
by translating $x$ to $\zeta_{0,\alpha}$. We may thus transfer statements about $\calR_\alpha$, such as
Theorem~\ref{T:p-adic turrittin}, directly
to the setting of solvable modules over $\calR_x$.

For $x$ of other types, the behavior of a solvable module over $\calR_x$ is much more restricted,
especially in the case of a module obtained by base extension from $R_{[0,\beta]}$.
It is most convenient to postpone discussion of this point until after we have Theorem~\ref{T:strict skeleton}
in hand; see \S\ref{subsec:more solvable}. However, one key case is needed for the proof of Theorem~\ref{T:strict skeleton}, so we include it here; see Lemma~\ref{L:type 4 constant}.
\end{remark}

\begin{lemma} \label{L:type 4 constant1}
Assume $p>0$.
Let $x \in \DD_{\beta,K}$ be a point of type 4 for which $\rho(x) \in \left| \CC^\times \right|$.
Let $N$ be a differential module over $\calR_x$ of rank $n$ which is solvable at $x$.
Then $b_i(N,x) \in [0,1]$ for $i=1,\dots,n$.
\end{lemma}
\begin{proof}
Using Lemma~\ref{L:base change radii} and the fact that Berkovich's classification is preserved
by passage from $K$ to $\CC$ (see Proposition~\ref{P:Berkovich classification}),
we may assume without loss of generality that $K = \CC$ and $\rho(x) = 1$.
We may also assume $n>0$.

Let $L_1, L_2$ be two copies of $\calH(x)$, and let $L_3$ be a complete extension of both
(obtained by choosing an element of $\calM(L_1 \widehat{\otimes}_K L_2)$).
Let $t_1, t_2$ be the copies of $t_x$ in $L_1, L_2$.
For $i=1,2$, let $\calR_{(i)}$ be a copy of $\calR_1$ (that is, the ring $\calR_\alpha$ with $\alpha = 1$)
over $L_i$ in the variable $t-t_i$.
Let $\calR_{(3)}$ be a copy of $\calR_1$ over $L_3$ in the variable $t-t_1$,
and equip $\calR_{(3)}$ with the map from $\calR_{(1)}$ sending $t - t_1$ to $t-t_1$ and the map from $\calR_{(2)}$ sending $t-t_2$ to
$t-t_1 + (t_1-t_2)$.
For $i=1,2$, we may identify the residue field of $\calR_{(i)}$ with $\kappa_K((u_i))$ for $u_i = (t-t_i)^{-1}$,
and then apply Theorem~\ref{T:p-adic turrittin} and Corollary~\ref{C:turrittin ramification} to produce
the minimal finite \'etale extension $S_i$ of $\calR_{(i)}^{\inte}$ over which
$N \otimes_{\calR_{(i)}^{\inte}} S_i$ satisfies the Robba condition.
By the uniqueness in Corollary~\ref{C:turrittin ramification}, we must then have an isomorphism
\begin{equation} \label{eq:type 4 constant}
S_1 \otimes_{\calR_{(1)}^{\inte}} \calR_3^{\inte} \cong S_2 \otimes_{\calR_{(2)}^{\inte}} \calR_3^{\inte}
\end{equation}
which commutes with the cocycle condition. This implies (e.g., by faithfully flat descent)
that $S_1$ admits an action of the group  
of $\kappa_K$-linear substitutions on $\kappa_K((u_1))$ of the form $t \mapsto t + c$ with $c \in \kappa_K$.
Let $L$ be the residue field of $S_1$; applying Proposition~\ref{P:local fields}, we may deduce that the highest upper numbering ramification break of $L$ as a finite extension of $\kappa_K((u_1))$ 
is at most $1$. 
By Corollary~\ref{C:turrittin ramification},
this implies that $b_1(N,x) \leq 1$ and hence $b_i(N,x) \leq 1$ for $i=1,\dots,n$.
\end{proof}

\begin{lemma} \label{L:type 4 constant}
Assume $p>0$.
Let $M$ be a differential module over $R_{[0,\beta]}$ of rank $n$.
Let $x$ be a point of type 4 for which $\rho(x) \in \left| \CC^\times \right|$.
Put $N = M \otimes_{R_{[0,\beta]}} \calR_x$.
If $N$ is solvable at $x$, then $b_i(N,x) \in \{0,1\}$ for $i=1,\dots,n$.
\end{lemma}
\begin{proof}
Let $j \in \{0,\dots,n\}$ be any index for which $b_1(N,x), \dots, b_j(N,x) > 0$.
For $r$ in some left neighborhood of $-\log \rho(x)$, the function 
\[
\sum_{i=1}^j -\log s_i(M,H(x,e^{-r}))
\]
is affine with nonpositive slope by Proposition~\ref{P:variation}(a,d).
However, this slope is equal to
\[
\sum_{i=1}^j (1 - b_i(N,x)),
\]
each summand of which is nonnegative by Lemma~\ref{L:type 4 constant1}. This proves the claim.
\end{proof}

\begin{remark}
It is tempting to argue directly that the isomorphism \eqref{eq:type 4 constant}
from the proof of Lemma~\ref{L:type 4 constant} implies
by faithfully flat descent that $S_1$ descends to a finite \'etale algebra over $\calR_x^{\inte}$.
One obstruction to this approach is that it is unclear whether the maps $\calR_x^{\inte} \to \calR_{(i)}^{\inte}$ are flat.
\end{remark}

\subsection{Controlling graphs for radii of convergence}
\label{subsec:controlling graphs}

Using Proposition~\ref{P:compare radii}, we can give a partial translation of Proposition~\ref{P:variation}
into the language of radii of optimal convergence.
This reproduces and improves a result of Pulita 
\cite[Theorem~4.7]{pulita-poineau};
see Remark~\ref{R:pulita}.
Throughout \S\ref{subsec:controlling graphs}, retain Hypothesis~\ref{H:radii discs}.

\begin{defn}
For $x \in \DD_{\beta,K}$, put
\[
f_i(M, x) = -\log s_i(M, x) \qquad ( i \in \{1,\dots,n\})
\]
and $F_i(M,x) = f_1(M,x) + \cdots + f_i(M,x)$.
Note that unlike the functions $f_i(M,r)$ considered in \S\ref{sec:discs annuli}, the function $f_i(M,x)$
may take values less than $-\log \rho(x)$. We are thus led to define the truncated functions
\begin{align*}
\overline{s}_i(M,x) &= \min\{\rho(x), s_i(M,x)\} \\
\overline{f}_i(M, x) &= -\log \overline{s}_i(M, x).
\end{align*}
\end{defn}

\begin{remark} \label{R:translation of old variation}
By Proposition~\ref{P:compare radii},
the function $f_i(M,r)$ of \S\ref{sec:discs annuli}
coincides with $\overline{f}_i(M,H(\zeta_{0,0},e^{-r}))$. This will allow us to apply Proposition~\ref{P:variation} to obtain
information about the functions $f_i$.
\end{remark}

\begin{prop} \label{P:divisible closure}
For any $x \in \DD_{\beta,K}$, 
$s_1(M,x)$ belongs to the divisible closure of $\left| \calH(x)^\times \right|$.
\end{prop}
\begin{proof}
By making the canonical base extension as in
Remark~\ref{R:canonical base extension},
 we may reduce to the case where $x$ is of type 1.
By Lemma~\ref{L:base change radii}, we may further reduce to the case where $K = \CC$ 
and $x = \zeta_{0,0}$.

By Remark~\ref{R:translation of old variation} and Proposition~\ref{P:rational intercepts}, the function 
$\overline{f}_1(M,H(x,e^{-r}))$
is piecewise of the form $ar+b$ with $a \in \QQ$ and $b \in \log |K^\times|$.
Put $r_0 = -\log s_1(M,x)$. By 
Proposition~\ref{P:transfer}, $r_0$ is the smallest value 
for which $\overline{f}_1(M,H(x,e^{-r})) = r$ for all $r \geq r_0$.
We thus have $r_0 = -b/a$ for some $a \in \QQ$ and $b \in \log |K^\times|$, proving the claim.
\end{proof}

\begin{lemma} \label{L:grow sections}
For $i=1,\dots,n$, if $s_i(M,x) > \rho(x)$ for some $x$, then $s_i(M,x)$ is constant on some neighborhood of $x$.
\end{lemma}
\begin{proof}
This is immediate from Remark~\ref{R:grow sections}.
\end{proof}

\begin{lemma} \label{L:piecewise affine restriction}
For $i=1,\dots,n$, the restriction of $f_i(M, \cdot)$ to any skeleton of $\DD_{\beta,K}$ 
is piecewise affine.
\end{lemma}
\begin{proof}
It suffices to check that for any $x \in \DD_{\beta,K}$ not of type 1,
the function $g_i$ given by $g_i(r) = f_i(M, H(x,e^{-r}))$ is piecewise affine
(the same then holds for points of type 1 by Lemma~\ref{L:grow sections}
and Remark~\ref{R:positive radius}).
We first verify that $\max\{r, g_i(r)\} = \overline{f}_i(M,H(x,e^{-r}))$ is piecewise affine.
Using Lemma~\ref{L:base change radii}, we may reduce to the case where $x = \zeta_{0,\alpha}$ for some
$\alpha > 0$, in which case the claim follows from Proposition~\ref{P:variation}(a).

Given that $\max\{r, g_i(r)\}$ is piecewise affine, it follows that $g_i$ is piecewise affine at 
any $r_0$ for which $g_i(r_0) > r_0$. At a value $r_0$ where $g_i(r_0) < r_0$, by 
Lemma~\ref{L:grow sections}, $g_i$ is constant
in a neighborhood of $r_0$. It thus suffices to check piecewise affinity at an arbitrary value $r_0$ at which
$g_i(r_0) = r_0$.

We first consider a right neighborhood of $r_0$. If the values of $r$ in this neighborhood
for which $g_i(r) < r$ fail to accumulate at $r_0$, then in some smaller neighborhood we have
$g_i(r) = r$ identically.
Otherwise, for each value $r_1$ at which $g_1(r_1) < r_1$, by the previous paragraph $g_i$ is constant
for $r \geq r_1$. It follows that $g_i$ is constant for $r > r_0$ and the constant value must be at most $r_0$.
If it were strictly less than $r_0$, we would have $g_i(r_0) < r_0$ by Remark~\ref{R:grow sections},
contrary to hypothesis;
we thus have $g_i(r) = r_0$ for $r \geq r_0$.
This proves affinity to the right of $r_0$.

We next consider a left neighborhood of $r_0$. If there exists any $r_1$ in this neighborhood
for which $g_i(r_1) < r_1$, then as above, $g_i$ would be constant for $r \geq r_1$.
But then we would have $r_0 = g_i(r_0) = g_i(r_1) < r_1 < r_0$, a contradiction.
Hence $g_i(r) = r$ identically in this neighborhood.
This proves affinity to the left of $r_0$.
\end{proof}
 
\begin{remark}
By Lemma~\ref{L:piecewise affine restriction}, it makes sense to refer to the slopes
of $f_i(M, \cdot)$ or $\overline{f}_i(M, \cdot)$ along any branch of $\DD_{\beta,K}$. 
\end{remark}

\begin{defn}
For $x \in \DD_{\beta,K}$, define the \emph{spectral cutoff} of $M$ at $x$ to be the largest value
$m(x) \in \{0,\dots,n\}$ such that $s_i(M,x) < \rho(x)$ for $i=1,\dots,m(x)$.
\end{defn}

\begin{lemma} \label{L:variation1 convexity}
Let $U$ be a lower branch of $\DD_{\beta,K}$ at some point $x$.
Suppose that for $i=1,\dots,m(x)$, the slope of $\overline{f}_i(M,\cdot)$ along $U$
(which exists by Lemma~\ref{L:piecewise affine restriction})
is equal to $0$. Then 
\[
s_i(M,x) = s_i(M,y) \qquad (y \in U; \, i=1,\dots,n).
\]
\end{lemma}
\begin{proof}
By Lemma~\ref{L:piecewise affine restriction}, for any $y \in U$ we have $s_i(M,H(y,\rho)) \to s_i(M,x)$
as $\rho \to \rho(x)^-$. This implies on one hand that the conditions of Lemma~\ref{L:variation0 convexity}
are satisfied, so $s_i(M,y)$ is constant for $y \in U$, and on the other hand that this constant value
is equal to $s_i(M,x)$.
\end{proof}

\begin{lemma} \label{L:variation1 convexity2}
Let $S$ be a rooted skeleton of $\DD_{\beta,K}$. Let $T$ be the interior of an edge in a subdivision of $S$.
Suppose that for $i=1,\dots,n$, $\overline{f}_i(M,\cdot)$ is affine on $T$. Then
\[
s_i(M,y) = s_i(M, \pi_S(y)) \qquad (i=1,\dots,n; y \in \pi_S^{-1}(T)).
\]
\end{lemma}
\begin{proof}
For $x \in T$ and $i=1,\dots,m(x)$, by Proposition~\ref{P:variation}(c,d) and Remark~\ref{R:translation of old variation},
the slope of $\overline{f}_i(M, \cdot)$ along any lower branch of $x$ other than the
one meeting $T$ is equal to $0$. The claim thus follows from Lemma~\ref{L:variation1 convexity}.
\end{proof}

\begin{lemma} \label{L:branches with slope 0}
For any $x \in \DD_{\beta,K}$, along all but finitely many lower branches of $\DD_{\beta,K}$ at $x$, the slope of
$\overline{f}_i(M,\cdot)$ is $0$ for $i=1,\dots,m(x)$.
\end{lemma}
\begin{proof}
This is immediate from Proposition~\ref{P:variation}(c).
\end{proof}

\begin{lemma} \label{L:local controlling graph}
For any $x \in \DD_{\beta,K}$, there exist a skeleton $S$ of $\DD_{\beta,K}$ and an open neighborhood $I$ of $\pi_S(x)$
such that the restrictions of $s_1(M, \cdot), \dots, s_n(M, \cdot)$ to $\pi_S^{-1}(I)$
factor through $\pi_S$. Moreover, we may choose $S$ to have no generators of type 3.
\end{lemma}
\begin{proof}
By Lemma~\ref{L:branches with slope 0}, along all but finitely many lower branches of $\DD_{\beta,K}$ at $x$, the slope of
$\overline{f}_i(M,\cdot)$ is $0$ for $i=1,\dots,m(x)$. Choose $S$ to pass through $x$ and
meet each of the remaining lower branches
of $X$ at $x$; this can always be done without using generators of type 3 because any point of type 3 dominates
some points of type 2  by Proposition~\ref{P:Berkovich classification}.
By Lemma~\ref{L:piecewise affine restriction}, we can find a subdivision of $S$ such that 
for $i=1,\dots,n$, $\overline{f}_i(M, \cdot)$ is affine on each edge of the subdivision meeting $x$.
Let $I$ be the union of the interiors of these edges, together with $x$. For $y \in \pi_S^{-1}(I)$, we have
$s_i(M,y) = s_i(M,\pi_S(y))$ by Lemma~\ref{L:variation1 convexity} (if
$\pi_S(y) = x$) or Lemma~\ref{L:variation1 convexity2} (if $\pi_S(y) \neq x$).
\end{proof}

\begin{lemma} \label{L:local controlling graph2a}
For $x \in \DD_{\beta,K}$ of type 4, for $i=1,\dots,n$,
in some left neighborhood of $\rho(x)$, the function
\[
\rho \mapsto \min\{\omega \rho, s_i(M,H(x,\rho))\}
\]
is either constant or identically equal to $\omega \rho$.
\end{lemma}
\begin{proof}
Apply Corollary~\ref{C:cyclic vector} to construct $\bv \in M$ which is a cyclic vector in
$M \otimes_{R_{[0,\beta]}} \Frac(R_{[0,\beta]})$, and write
$D^n(\bv) = a_0 \bv + \cdots + a_{n-1} D^{n-1}(\bv)$ for some $a_0,\dots,a_{n-1} \in \Frac(R_{[0,\beta]})$. 
Since $x$ is of type 4, for $i=0,\dots,n-1$, the function $y \mapsto y(a_i)$ is constant in some neighborhood
of $x$. By Proposition~\ref{P:christol-dwork}, this yields the desired result.
\end{proof}

\begin{lemma}\label{L:local controlling graph2b}
For $x \in \DD_{\beta,K}$ of type 4, for $i=1,\dots,n$, 
in some left neighborhood of $\rho(x)$, the function $\rho \mapsto \overline{s}_i(M, H(x,\rho))$
is either constant or identically equal to $\rho$.
\end{lemma}
\begin{proof}
This is immediate from Lemma~\ref{L:local controlling graph2a} if $p=0$, so we may assume $p>0$;
we may also assume $K=\CC$. Let $h$ be the smallest nonnegative integer for which
$s_i(M,x) \notin (\omega^{p^{-h-1}} \rho(x), \rho(x))$ for $i=1,\dots,n$.
We proceed by induction on $h$.
 
Put $r_0 = - \log \rho(x)$; since $x$ is of type 4, we have $r_0 > -\log \beta$.
Let $j \in \{0,\dots,n\}$ be the largest value for which $s_i(M,x) \leq \omega \rho(x)$ for $i=1,\dots,j$.
Since the functions $r \mapsto \overline{f}_i(M, H(x,e^{-r}))$ are continuous by
Lemma~\ref{L:piecewise affine restriction}, we may apply Lemma~\ref{L:local controlling graph2a}
to produce $r_1 \in (-\log \beta, r_0)$ such that for $i=1,\dots,j$, the function $r \mapsto f_i(M, H(x,e^{-r}))$
is constant for $r \in [r_1,r_0]$. By moving $r_1$ towards $r_0$, we may also ensure that
$\rho(x) > \omega e^{-r_1}$ and $s_i(M, H(x,e^{-r_1})) > \omega e^{-r_1}$ for $i>j$.
By rescaling $t$, we may further ensure that $r_1 < 0 < r_0$.

By Proposition~\ref{P:Berkovich classification},
we can find $z \in \CC$ such that $H(x,1) = \zeta_{z,1}$. There is no harm in applying a translation
on the disc to reduce to the case $z = 0$. If we put $\beta' = e^{-r_1}$,
then by Proposition~\ref{P:decompose disc}, the restriction of $M$
to $\DD_{\beta',K}^{\circ}$ splits as a direct sum $M_1 \oplus M_2$ with $\rank(M_1) = j$
and $f_i(M, e^{-r}) = f_i(M_1,e^{-r})$ for $i=1,\dots,j$ and $r \in (r_1,0]$.
By Corollary~\ref{C:constant initial}, the original claim holds with $M$ replaced by the restriction of $M_1$ to
$\DD_{1,K}$.

Let $N$ be the restriction of $M_2$ to $\DD_{1,K}$;
it now suffices to prove the original claim with $M$ replaced by $N$.
We may assume $j < n$, as otherwise there is nothing to check.
We first check the claim for $N$ in case $s_{i+1}(M, x) \geq \rho(x)$,
which in particular will cover the base case $h=0$ of the induction.
If $\rho(x) \notin \left| \CC^\times \right|$, then Proposition~\ref{P:divisible closure} implies that
$s_i(N,x) > \rho(x)$ for all $i$, so the desired result follows by Lemma~\ref{L:grow sections}.
If instead $\rho(x) \in \left| \CC^\times \right|$, then the desired result follows by 
Lemma~\ref{L:type 4 constant}.

We next check the claim for $N$ in case $s_{i+1}(M,x) < \rho(x)$;
note that by construction we also have $\omega \rho(x) < s_{i+1}(M,x)$.
Let $\psi: \DD_{1,K}^{\circ} \to \DD_{1,K}^{\circ}$ be the map for which 
$\psi^*(t) = (t+1)^p - 1$. Put $y = \psi(x)$; it is a point of type 4 with $\rho(y) = \rho(x)^p$.
Let $N'$ be the off-center Frobenius descendant of $N$
in the sense of Proposition~\ref{P:off-center descendant} with $\lambda = 1$.
By that proposition, $\overline{s}_{(p-1)(n-j)+i}(N', z) = \overline{s}_i(N,z)^p$ for $i=1,\dots,n-j$
and $z \in \DD_{1,K}$ with $\rho(z) > \omega$.
Since we assumed that $\rho(x) > \omega \beta' > \omega$,
we have $\overline{s}_{(p-1)(n-j)+i}(N', H(y,\rho^p)) = \overline{s}_i(N,H(x,\rho))^p$ for $i=1,\dots,n-j$
and $\rho \in [\rho(x), 1]$.
We may thus deduce the claim for $N$ from the corresponding claim for $N'$, to which we may apply the induction
hypothesis because we have decreased the value of $h$.
\end{proof}

\begin{lemma} \label{L:local controlling graph2}
For $x \in \DD_{\beta,K}$ of type 1 or 4, 
for $i=1,\dots,n$, the function $s_i(M, \cdot)$ is constant on some neighborhood of $x$.
\end{lemma}
\begin{proof}
For $x$ of type 1, the claim follows from Remark~\ref{R:grow sections} and Remark~\ref{R:positive radius}.
For $x$ of type 4, Lemma~\ref{L:local controlling graph2b} implies that the hypothesis of
Lemma~\ref{L:variation1 convexity} holds for some open disc containing $x$, yielding the claim in this case.
\end{proof}

\begin{theorem} \label{T:strict skeleton}
\begin{enumerate}
\item[(a)]
There exists a strict skeleton $S$ of $\DD_{\beta,K}$
such that $s_1(M, \cdot), \dots, s_n(M, \cdot)$ factor through $\pi_S$.
\item[(b)]
For $i=1,\dots,n$, $f_i(M,\cdot)$ is piecewise affine with slopes in
$\frac{1}{1} \ZZ \cup \cdots \cup \frac{1}{n} \ZZ$. Moreover, $F_n(M,\cdot)$ has integral slopes.
\item[(c)]
There is a unique minimal graph $G$ in $\DD_{\beta,K}$ which is a controlling
graph for all of the functions $f_i(M,\cdot)$.
Moreover, the vertices of $G$ are all of type 2 or 3.
(We call $G$ the \emph{controlling graph} of $M$.)
\end{enumerate}
\end{theorem}
\begin{proof}
For each $x \in \DD_{\beta,K}$, apply Lemma~\ref{L:local controlling graph} to construct a skeleton $S_x$ of $\DD_{\beta,K}$ and
and an open neighborhood $I_x$ of $\pi_S(x)$
such that the restrictions of $s_1(M, \cdot), \dots, s_n(M, \cdot)$ to $\pi_{S_x}^{-1}(I_x)$
factor through $\pi_{S_x}$. Since $\pi_{S_x}^{-1}(I_x)$ is open in the compact space $\DD_{\beta,K}$, we can choose
finitely many points $x_i \in \DD_{\beta,K}$ such that, if we relabel $S_x, I_x$ as $S_i, I_i$, then the
open sets $\pi_{S_i}^{-1}(I_i)$ cover $\DD_{\beta,K}$. Let $S$ be the union of the $S_i$; for $y \in \pi_{S_i}^{-1}(I_i)$, we have
$\pi_{S_i}(y) = \pi_{S_i}(\pi_S(y))$ and so $s_i(M,y) = s_i(M, \pi_{S_i}(y)) = s_i(M, \pi_S(y))$.
This proves (a) except that $S$ might include some generators of types 1 or 4
(generators of type 3 are excluded by Lemma~\ref{L:local controlling graph}).
However, by Lemma~\ref{L:local controlling graph2}, if $x$ is a generator of type 1 or 4,
then the functions $s_i(M, \cdot)$ are constant in a neighborhood of $x$, so we may replace $x$
with a point of type 2 in this neighborhood which dominates $x$.
We thus deduce (a). 

{}From (a), we deduce piecewise affinity using Lemma~\ref{L:piecewise affine restriction}.
To deduce integrality of slopes, we apply Proposition~\ref{P:variation}(b) at points $x$ where
$s_i(M,x) < \rho(x)$ and Lemma~\ref{L:grow sections} at points $x$ where $s_i(M,x) > \rho(x)$.
This fails to account for segments where $s_i(M,x) =\rho(x)$ identically, but on any such segment
$f_i(M,x)$ has slope 1. We thus deduce (b).

Using (a) and (b), we deduce the existence of the minimal controlling graph $G$
and the fact that none of its vertices is of type $1$ or $4$. 
This yields (c).
%
%Using Proposition~\ref{P:divisible closure},
%we further deduce that $G$ has no vertices of type $3$, thus yielding (c).
\end{proof}

\begin{remark} \label{R:pulita}
The weaker form of Theorem~\ref{T:strict skeleton} in which strictness of the skeleton is not asserted
is the essential content of \cite[Theorem~4.7]{pulita-poineau} applied to a disc: 
more precisely, parts (i) (finiteness) and (ii) (integrality)
of that result are included in Theorem~\ref{T:strict skeleton}. The proof in \cite{pulita-poineau} is a bit different, making use of a combinatorial criterion for piecewise affinity.
A proof in terms of $p$-adic potential theory is given in \cite{pulita-poineau3}. Another proof, essentially a streamlined version of the above arguments, is given in \cite{baldassarri-kedlaya}.
Neither the analysis in \cite{pulita-poineau} nor in \cite{pulita-poineau3} nor in \cite{baldassarri-kedlaya} includes any 
special study of type 4 points, as these are treated
by base extension to convert them into other types. Consequently, the techniques of those papers cannot by
themselves exclude vertices of type 4 from the controlling graph, which here is made possible by the 
analysis in \S\ref{subsec:solvable}.

Note that \cite[Theorem~4.7]{pulita-poineau} gives a finer description of the controlling graph
than appears here. It also 
includes weak analogues of the convexity, subharmonicity, and monotonicity assertions from
Proposition~\ref{P:variation} (although with a change of sign convention, so convexity becomes
concavity and subharmonicity becomes superharmonicity). 
In \cite[Theorem~4.7]{pulita-poineau} these statements are used in an essential
way to prove finiteness; however, given Theorem~\ref{T:strict skeleton}, they can be 
deduced directly from Proposition~\ref{P:variation}.

Note also that \cite{baldassarri-kedlaya, pulita-poineau2, pulita-poineau3, pulita-poineau} consider not just discs
but more general curves. We will return to this more general case in \S\ref{sec:berkovich curves}.
\end{remark}

\subsection{More on solvable modules}
\label{subsec:more solvable}

With Theorem~\ref{T:strict skeleton} in hand, we can now fill out the discussion of solvable modules over
$\calR_x$ initiated in \S\ref{subsec:solvable}. We also point out a link with our previous work
on semistable reduction for overconvergent $F$-isocrystals \cite{kedlaya-semistable4}.

\begin{hypothesis} \label{H:special solvable}
Throughout \S\ref{subsec:more solvable},
let $M$ be a differential module over $R_{[0,\beta]}$ of rank $n$,
choose $x \in \DD_{\beta,K}$,
put $M_x = M \otimes_{R_{[0,\beta]}} \calR_x$,
and let $N$ be a subquotient of $M_x$ which is solvable at $x$.
\end{hypothesis}

\begin{remark} \label{R:type 1 3}
For $x$ of type 3, Theorem~\ref{T:strict skeleton} forces $N$ to satisfy the Robba condition;
if $N = M_x$, then $N$ is forced to be trivial by
Proposition~\ref{P:transfer}.
For $x$ of type 1, we can say even more: Proposition~\ref{P:transfer} and Theorem~\ref{T:strict skeleton}
together imply that $M_x$ itself is a trivial differential module, as then is $N$.
\end{remark}

For $x$ of type 4, we have the following refinements of Lemma~\ref{L:type 4 constant}.

\begin{prop} \label{P:type 4 constant}
Suppose that $x$ is of type 4.
\begin{enumerate}
\item[(a)]
If $\rho(x) \in \left| \CC^\times \right|$, then $b_i(N,x) \in \{0,1\}$ for all $i$.
\item[(b)]
If $\rho(x) \notin \left| \CC^\times \right|$, then $N$ is trivial, so $b_i(N,x) = 0$ for all $i$.
\end{enumerate}
\end{prop}
\begin{proof}
By Theorem~\ref{T:strict skeleton} (or Lemma~\ref{L:local controlling graph2}),
the functions $s_i(M, \cdot)$ are constant in a neighborhood of $x$. This immediately implies (a).
To deduce (b), note that we must have $s_1(M, x) \neq \rho(x)$ by Proposition~\ref{P:divisible closure}.
Since the $s_i(M, \cdot)$ are constant, we may apply Proposition~\ref{P:decompose disc} to decompose $M$ in a neighborhood of $x$ as a direct sum $M' \oplus M''$ with $s_i(M', x) < \rho(x)$ for all $i$
and $s_i(M'', x) \geq \rho(x)$ for all $i$; by applying  Proposition~\ref{P:divisible closure} again,
we see that in fact $s_i(M'', x) > \rho(x)$ for all $i$.
In particular, $M''$ is trivial on some neighborhood of $x$;
moreover, the projection of $N$ onto $M' \otimes \calR_x$ must be zero.
It follows that $N$ is trivial, yielding (b).
\end{proof}

\begin{theorem} \label{T:turrittin type 4}
Assume that $K = \CC$, $x$ is of type 4, and $\rho(x) = 1$. 
For each $c \in \kappa_K$, choose a lift $\tilde{c}$ of $c$ to $\gotho_K$,
and let $Q_c$ be the differential module over $\calR_x$ free on one generator $\bv$
such that $D(\bv) = \tilde{c} \bv$.
\begin{enumerate}
\item[(a)]
For each irreducible subquotient $P$ of $N$, there exists $c \in \kappa_K$ such that 
$P \otimes Q_c$ satisfies the Robba condition.
\item[(b)]
There exists a finite \'etale extension
$S$ of $\calR_x$ of the form
\[
S = \calR_x[z_1,\dots,z_m]/(z_1^p - z_1 - a_1 t, \dots, z_m^p - z_m - a_m t)
\]
for some nonnegative integer $m$ and some $a_1,\dots,a_m \in \gotho_K^\times$ such that
$N \otimes_{\calR_x} S$ is trivial.
\end{enumerate}
\end{theorem}
\begin{proof}
By Proposition~\ref{P:type 4 constant} we have $b_1(P) \in \{0,1\}$. If
$b_1(P) = 0$ we take $c=0$; otherwise, by \cite[Theorem~12.7.2]{kedlaya-book},
we can choose $c$ so that $b_1(P \otimes Q_c) < 1$, and then by Proposition~\ref{P:type 4 constant} again
we have $b_1(P \otimes Q_c) = 0$. This proves (a).

Given (a), to prove (b), Theorem~\ref{T:strict skeleton} and Proposition~\ref{P:decompose disc} allow us
to reduce to the case where $M_x$ itself is solvable at $x$;
we may then further reduce to the case where $N = M_x$.
In this case, the proof of Theorem~\ref{T:p-adic turrittin} provides $S$ such that $N \otimes_{\calR_x} S$
satisfies the Robba condition. However, by induction on $m$, we see that there is an isomorphism
\begin{equation} \label{eq:isomorphism of artin-schreier}
R_{[0,\beta^{p^{-m}}]} \cong R_{[0,\beta]}[z_1,\dots,z_m]/(z_1^p-z_1 - a_1 t, \dots, z_m^p - z_m - a_m t)
\end{equation}
sending $t$ to $z_m$. 
This isomorphism gives rise to a map $\psi: \DD_{\beta^{p^{-m}},K} \to \DD_{\beta,K}$
by mapping $R_{[0,\beta]}$ into the right side of \eqref{eq:isomorphism of artin-schreier} and then
crossing to the left side. The inverse image of $x$ under this map is a single point $y$. By
construction, $N \otimes_{\calR_x} S \cong 
\psi^*M \otimes_{R_{[0,\beta^{p^{-m}}]}} \calR_y$ satisfies the Robba condition.
By Proposition~\ref{P:transfer}, $\psi^* M$ is trivial in a neighborhood of $y$,
so $N \otimes_{\calR_x} S$ is also trivial.
\end{proof}

\begin{cor} \label{C:turrittin type 4}
Assume that $x$ is of type 4. Then any subquotient of $N$ satisfying the
Robba condition is trivial, and hence admits the zero tuple as an exponent.
\end{cor}
\begin{proof}
If $\rho(x) \notin \left| \CC^\times \right|$, then $N$ is trivial
by Proposition~\ref{P:type 4 constant}(b), so any subquotient of $N$ satisfying the Robba condition is
also trivial and hence admits the zero tuple as an exponent.
If $\rho(x) \in \left| \CC^\times \right|$, we may assume that $K = \CC$ and $\rho(x) = 1$.
Set notation as in the proof of Theorem~\ref{T:turrittin type 4}(b),
again reducing to the case where $N = M_x$.
In this case, the Tannakian category of differential modules over $\calR_x$ generated by $N$ admits a fibre functor
computing horizontal sections over $S$, for which the automorphism group is an elementary abelian
$p$-group. In particular, $N$ splits as a direct sum of irreducible submodules whose $p$-th tensor powers are trivial. Consequently, to check that a subquotient of $N$ satisfying the Robba condition is trivial, 
it suffices to check the case of a irreducible submodule $P$ for which $P^{\otimes p}$ is trivial;
this case follows from Corollary~\ref{C:p-th power trivial}.
\end{proof}

\begin{remark}
Note that the isomorphism in \eqref{eq:isomorphism of artin-schreier} depends critically on having linear powers of $t$ on the right side; otherwise, we would end up with something other than a disc, so Dwork's transfer theorem
(Proposition~\ref{P:transfer}) would not apply. This is why it is necessary to invest the hard work to first
prove $b_i(N,x) \in \{0,1\}$ in order to deduce Corollary~\ref{C:turrittin type 4}.
\end{remark}

\begin{remark} \label{R:semistable}
The above arguments, including the proof of Lemma~\ref{L:type 4 constant}, are loosely inspired by the arguments made in \cite[\S 5]{kedlaya-semistable4}.
However, the correspondence turns out to be somewhat less close than we had originally expected, primarily
because the process of transposing the arguments exposed an error in \cite{kedlaya-semistable4}. We now describe this error and how it may be remedied using results from this paper.

The error appears in the second sentence of the proof of \cite[Lemma~5.6.2]{kedlaya-semistable4}:
it is not the case that the property of being terminally presented is stable under tame alterations. 
That is because the tame alteration $x \mapsto x^m$ is ramified along the segment joining 0 to the Gauss point;
consequently, after pulling back a terminally presented module along a tame alteration, one encounters a change of
slope at the point where one branches off from the ramification locus.
In the continuation of the proof, the tame alteration is erroneously used to force the group 
$\tau(I_1)$, which initially is the semidirect
product of the $p$-group $\tau(W_1')$ with a cyclic group of order prime to $p$,
to become equal to $\tau(W_1')$. 

To correct the proof, it suffices to establish that the equality
$\tau(I_1) = \tau(W_1')$ holds initially, so that no tame alteration is needed and the rest of 
the argument may proceed unchanged. To verify this,
choose $\rho$ as in \cite[Lemma~4.7.4]{kedlaya-book}; by that lemma, $\left|\cdot\right|_{\rho^\alpha,s_0}$
defines a point of $\calM(\ell \langle x \rangle)$ of type 4. We may thus apply
Corollary~\ref{C:turrittin type 4} to deduce that any subquotient of the cross-section $M_\rho$ which
satisfies the Robba condition admits the zero tuple as an exponent. This implies that $\tau(W_1')$ has
no nontrivial quotient of prime-to-$p$ order, and so $\tau(I_1) = \tau(W_1')$ as desired.

One might prefer to incorporate some of the intermediate arguments from this paper into the proof method of
\cite{kedlaya-semistable4}, but this seems difficult. The plan of attack in \cite{kedlaya-semistable4}
is to pick out an Artin-Schreier extension that reduces the image of the monodromy representation,
which requires tame ramification to be ruled out first.
By contrast, the method here is to use Artin-Schreier extensions only to lower the ramification numbers;
only when this stops being possible is the presence of tame ramification ruled out.

A more satisfying resolution would be to use additional results of this paper, especially Theorem~\ref{T:turrittin type 4}, to shortcut many of the complicated proofs in \cite[\S 5]{kedlaya-semistable4}. We leave this as an exercise for the interested reader.
\end{remark}

\section{Berkovich curves}
\label{sec:berkovich curves}

To conclude, we globalize our setup to include more
general Berkovich curves. We now adopt the full language of Berkovich analytic spaces, as in
\cite{berkovich1, berkovich2}.

\subsection{Analytic spaces}

\begin{defn} \label{D:spaces}
A \emph{strictly affinoid algebra} (resp.\ an \emph{affinoid algebra}) over $K$ is a commutative Banach algebra over $K$
isomorphic to a quotient of the completion of some polynomial ring $K[T_1,\dots,T_n]$ for the Gauss norm
(resp.\ the $(r_1,\dots,r_n)$-Gauss norm for some $r_1,\dots,r_n > 0$).

Let $A$ be a (strictly) affinoid algebra over $K$.
A \emph{(strictly) affinoid subdomain} of $\calM(A)$
is a closed subset $U$ for which the category of bounded $K$-linear homomorphisms $A \to B$ of (strictly) affinoid
algebras whose restriction maps carry $\calM(B)$ into $U$ has an initial element. Any such initial homomorphism $A \to B$ is then flat and induces a homeomorphism $\calM(B) \cong U$ \cite[Proposition~2.2.4]{berkovich1};
in particular, a strictly affinoid subdomain is also an affinoid subdomain. 

Note that $\calM(A)$ admits
a neighborhood basis of affinoid subdomains, e.g., because any rational subdomain is an affinoid subdomain.
For $x \in \calM(A)$, define the local $A$-algebra $A_x$ as the direct limit of the
representing homomorphisms $A \to B$ over all affinoid subdomains of $\calM(A)$ which are neighborhoods of $x$.
We define the \emph{structure sheaf} $\calO$ on $\calM(A)$ so that for $U$ an open subset of $\calM(A)$,
$\calO(U)$ consists of the functions $f: U \mapsto \coprod_{x \in \calM(A)} A_x$ such that for each $x \in U$,
there exist a homomorphism $A \to B$ and an element $g \in B$ such that:
\begin{itemize}
\item
the map $A \to B$ represents an affinoid subdomain of $\calM(A)$ contained
in $U$ and containing a neighborhood of $x$; 
\item
for each $y \in U$, $f(y)$ is the image of $g$ in $A_y$.
\end{itemize}
By Tate's theorem, the natural map $A \to \Gamma(\calM(A), \calO)$ is a bijection.
By Kiehl's theorem, coherent sheaves over $\calO$ correspond to finite $A$-modules
via the functor of global sections.
\end{defn}

\begin{defn}
A \emph{good (strictly)
$K$-analytic space} is a locally ringed space which is locally isomorphic to an open subspace of
the Gel'fand spectrum of a (strictly) affinoid algebra over $K$. 
These are the analytic spaces considered in \cite{berkovich1}; they have the property that any
point has a neighborhood basis consisting of affinoid spaces.
\end{defn}

\begin{remark} \label{R:good}
In \cite{berkovich2}, the more general notion of a \emph{(strictly) $K$-analytic space} is considered,
in which it is only required that each point have a neighborhood basis consisting of a finite union of
affinoid spaces (glued in a suitable way). In this paper, we can get away with considering only good
spaces because any curve over $K$ is good \cite[Corollary~3.4]{dejong}. 
\end{remark}

\subsection{Curves and triangulations}

We next introduce some of the the combinatorial structure of a Berkovich analytic curve over $K$.
One way to explain this is using semistable models, as in \cite{baldassarri} and \cite{baldassarri-kedlaya}.
Here, we take an alternate approach using triangulations introduced by Ducros \cite{ducros}, 
so as to avoid leaving the realm of analytic spaces; this follows the example of \cite{pulita-poineau, pulita-poineau2,pulita-poineau3}. There is also a link to tropicalization; see Remark~\ref{R:tropicalization}.

\begin{defn} \label{D:rig-smooth curve}
For $K'$ an analytic field containing $K$ and $X$ a good $K$-analytic space, 
let $X_{K'}$ denote the base extension of $K$ to $K'$. For $X = \calM(A)$, we have $X_{K'} = \calM(A \widehat{\otimes}_K K')$.

Let $\Omega_X$ denote the sheaf of continuous K\"ahler differentials on $X$. We say that $X$ is 
\emph{rig-smooth of pure dimension $n$} if for every analytic field $K'$ containing $K$,
$\Omega_{X_{K'}}$ is locally free of rank $n$.

By a \emph{curve} over $K$, we will mean
a good $K$-analytic space $X$ which is separated (i.e., the diagonal morphism is a closed immersion)
and rig-smooth of pure dimension 1. In particular, $X$ is paracompact.
\end{defn}

\begin{defn}
Let $X$ be a curve over $K$.
For $x \in X$, we declare $x$
to be of \emph{type $1,2,3,4$} if the signature of $x$ is respectively
$(1,0,0)$, $(0,1,0)$, $(0,0,1)$, $(0,0,0)$. These cases are exhaustive by
Proposition~\ref{P:Berkovich classification} plus Noether normalization for strictly affinoid algebras
\cite[Corollary~6.1.2/2]{bgr}.
\end{defn}

\begin{defn}
An \emph{open disc} over $K$ is a $K$-analytic space isomorphic to $\bigcup_{\gamma \in (0,\beta]} \calM(R_{[0,\gamma]})$ for some $\beta > 0$.
An \emph{open annulus} over $K$ is a $K$-analytic space isomorphic to $\bigcup_{\alpha < \gamma \leq \delta < \beta}
\calM(R_{[\gamma,\delta]})$ for some $0 < \alpha < \beta$. 

A \emph{virtual open disc} (resp. \emph{virtual open annulus}) 
is a connected $K$-analytic space whose base extension to $\CC$ is a disjoint union
of open discs (resp.\ open annuli).
By the \emph{skeleton} of a virtual open annulus over $K$,
we mean the set of points not contained in a virtual open disc.
For the standard open annulus $\bigcup_{\alpha < \gamma \leq \delta < \beta}
\calM(R_{[\gamma,\delta]})$ within $\DD_{\beta,K}$, the skeleton is the set $\{\zeta_{0,\rho}: \rho \in 
(\alpha,\beta)\}$; in general, the skeleton of a virtual open annulus is an open segment.
\end{defn}

\begin{defn}
Let $X$ be a curve over $K$.
A \emph{weak strict triangulation} (resp.\ \emph{weak triangulation})
of $X$ is a locally finite subset $S$ of $X$
consisting of points of type 2 (resp.\ of types 2 or 3)
such that any connected component of $X \setminus S$ is a virtual open disc
or a virtual open annulus. The union of the skeleta of the connected components of $X \setminus S$ which 
are virtual open annuli forms a locally finite graph $\Gamma_S$, called the \emph{skeleton} of the weak
triangulation. The points of $\Gamma_S$ are all of types 2 or 3.
\end{defn}

\begin{remark} \label{R:retract}
The definition of weak triangulation used here is the same as in \cite{pulita-poineau} but is somewhat more permissive
than the one used in \cite{ducros}, in which it is required that $X \setminus S$ be relatively compact.
Omitting this condition makes it possible for $\Gamma_S$ to fail to meet some connected components of $X$, 
e.g., if there is a component which is itself a virtual open disc. If $\Gamma_S$ does meet every connected component of $X$, then 
there is a natural continuous retraction $\pi_S: X \to \Gamma_S$ taking any $x \in \Gamma_S$
to itself and taking any $x \in X \setminus \Gamma_S$ to the unique point of $\Gamma_S$ in the closure of 
the connected component of $X \setminus \Gamma_S$ containing $x$.
\end{remark}

\begin{theorem} \label{T:triangulation}
Any (strictly) analytic curve over $K$ admits a weak (strict) triangulation.
\end{theorem}
\begin{proof}
See \cite[Th\'eor\`eme~5.1.14]{ducros}.
\end{proof}

\begin{remark} \label{R:tropicalization}
There is also an approach to the structure theory of analytic curves via tropicalization, i.e., consideration
of the projections defined by evaluation at finitely many functions on the curve. For discussion of the case
$K = \CC$, including a proof of Theorem~\ref{T:triangulation} in that context, see \cite[\S 5]{bpr}.
\end{remark}

\begin{defn}
Let $X$ be a curve and let 
$x \in X$ be a point of type 2. Then the residue field $\kappa_{\calH(x)}$ is the function field of an
algebraic curve over $\kappa_K$; we denote the genus of this function field by $g(x)$ and call it the
\emph{genus} of $x$. For any weak triangulation $S$ of $X$, the type 2 points of $X \setminus S$
are all of genus 0; by Theorem~\ref{T:triangulation}, it follows that the type 2 points of $X$ of positive
genus form a locally finite set.
\end{defn}

\begin{defn} \label{D:branches}
Let $X$ be a curve. By a \emph{branch} of $X$ at a point $x \in X$, we mean a
local path-connected component of $X \setminus \{x\}$ at $x$. Depending on the type of $x$, branches exist as follows.
\begin{enumerate}
\item[1.] Exactly one branch.
\item[2.] Infinitely many branches, corresponding to all but finitely many places of the function field
$\kappa_{\calH(x)}$.
\item[3.] Either zero, one, or two branches.
\item[4.] Exactly one branch.
\end{enumerate}
Given a weak triangulation $S$ of $X$, we say a branch $U$ of $X$ at $x$ is \emph{skeletal}
(or \emph{$S$-skeletal} in case of ambiguity) if
the closure of $U \cap \Gamma_S$ contains $x$; such branches can only exist if $x \in \Gamma_S$.

We say that $x \in X$ is \emph{external} if it is of type 2 and its branches do not correspond to all of the
places of $\kappa_{\calH(x)}$ or if it is type 3 and it has fewer than two branches; otherwise,
we say that $x$ is \emph{internal}.
For any weak triangulation $S$ of $X$,
every point of $X \setminus \Gamma_S$ is internal, as is every point of $\Gamma_S$ lying in the interior of
an edge; by Theorem~\ref{T:triangulation}, it follows that the external points of $X$ form a locally finite set.
\end{defn}

\begin{example} \label{exa:Shilov boundary}
For $X$ an affinoid space, the external points of $X$ are precisely the points of the \emph{Shilov boundary},
the minimal subset of $X$ for which the maximal modulus principle holds.
\end{example}

\begin{remark}
If $X$ is a strictly affinoid space, then the points of the Shilov boundary are all of type 2. Consequently, for any strictly analytic curve $X$, the external points of $X$ are all of type 2, so
every point of type 3 has exactly two branches. 
However, for more general analytic spaces, a point of type 3 may have one or even zero branches. For instance, take $X$ to be the annulus
$\calM(R_{[\alpha,\beta]})$ for some $\alpha, \beta \in (0, +\infty) \setminus \left| \CC^\times \right|$. If $\alpha < \beta$, then
$\zeta_{0,\alpha}$ and $\zeta_{0,\beta}$ are points of type 3 each with only one branch. If $\alpha = \beta$, then $\calM(R_{[\alpha,\alpha]})$ consists only of a single point $\zeta_{0,\alpha}$ of type 3, which in particular has zero branches.
\end{remark}

\subsection{Convergence of local horizontal sections}

We next study the convergence of local horizontal sections on analytic curves.
As in the case of discs, we end up with
a global statement about the behavior of radii of convergence of differential modules
on analytic curves; this statement recovers the main results of \cite{pulita-poineau,
pulita-poineau2, pulita-poineau3}.

\begin{hypothesis}
For the remainder of the paper, let $X$ be a curve over $K$ equipped with a weak triangulation $S$
and let $M$ be a vector bundle over $X$ of constant rank $n>0$ equipped with a connection.
(Since $X$ is of dimension $1$, the connection is automatically integrable.)
\end{hypothesis}

\begin{remark}
One interesting case excluded by our hypotheses is that where $X$ is an affine line
and $S$ is empty. In this case, the radii of convergence should be allowed to be infinite, but we do not want to worry about this.
For a more comprehensive treatment, see for instance \cite{baldassarri-kedlaya}.
\end{remark}

In order to define analogues of the radii of optimal convergence, one must make reference to the chosen
triangulation. This has the same effect as the choice of a semistable model in \cite{baldassarri}. 
\begin{defn}
For $x \in \Gamma_S$, define $s_1(M,S,x), \dots, s_n(M,S,x)$ as the intrinsic subsidiary radii of $M$ in order,
and put $\rho_S(x) = 1$.

For $x \in X \setminus \Gamma_S$, lift $x$ to a point $y \in X_{\CC}$,
identify the connected component of $(X \setminus \Gamma_S)_{\CC}$ containing $x$
with an open disc of some radius $R$, then define $s_1(M,S,x),\dots,s_n(M,S,x)$ as the functions
$s_1(M,y)/R,\dots,s_n(M,y)/R$ as in Definition~\ref{D:open unit disc}.
We use the same identification (again dividing by $R$) to define the diameter $\rho_S(x)$.
These definitions do not depend on the choice of $y$ or $R$, and are stable under enlarging $K$.

For $x \in X$, define the \emph{spectral cutoff} of $M$ as the largest $m(x) \in \{0,\dots,n\}$ such that
$s_i(M,S,x) < \rho_S(x)$ for $i =1,\dots,m(x)$.
\end{defn}

In order to analyze these functions, it will be useful to consider them first along individual branches.
\begin{defn} \label{D:slope along branch}
Choose $x \in \Gamma_S$ of type $2$, let $U$ be a branch of $X$ at $x$, and let $v$ be the corresponding place
of $\kappa_{\calH(x)}$. Choose $t \in \calO_{X,x}$ with $x(t) = 1$ whose image $\overline{t}$ in $\kappa_{\calH(x)}$
is a uniformizer of $v$ (i.e., its $v$-valuation is the positive generator of the value group).
Then for $\beta \in (0,1)$ sufficiently close to 1, $t$ defines an isomorphism between the space of $y \in U$ with $y(t) \in (\beta,1)$ and the open annulus $\beta < |t| < 1$ in the $t$-line.
We can use this isomorphism to define the class of functions $f: X \to \RR$ which are affine along $U$ in a neighborhood of $x$,
and to associate to each such function a slope (in the direction away from $x$); neither of these definitions depends on the choice of $t$.
\end{defn}

\begin{lemma} \label{L:affine at higher genus}
Set notation as in Definition~\ref{D:slope along branch}.
Then for $i=1,\dots,m(x)$, the function $\log s_i(M,S,\cdot)$ is affine along $U$
and its limit at $x$ (approaching from within $U$) equals $\log s_i(M,S,x)$.
\end{lemma}
\begin{proof}
For $x \notin S$ this is immediate from Proposition~\ref{P:variation}(a).
For $x \in S$ with $g(x) = 0$, we may also apply Proposition~\ref{P:variation}(a)
over the ring $R_{(\alpha,1)}^{\an}$. For $x \in S$ with $g(x) \neq 0$, we obtain a differential
module over a ring $S$ which can be written as a finite \'etale algebra over $R_{(\alpha,1)}^{\an}$
of some degree $d>0$
such that $S \otimes_{R_{(\alpha,1)}^{\an}} F_1 \cong \calH(x)$ is a finite \emph{unramified}
extension of $F_1$. If we restrict scalars from $S$ to $R_{(\alpha,1)}^{\an}$, the multiset of intrinsic subsidiary radii does not change except that each multiplicity gets multiplied by $d$.
We may thus apply Proposition~\ref{P:variation}(a) in this case also.
\end{proof}

We have the following analogue of Proposition~\ref{P:variation}(c). A more detailed exposition of the geometry used in this argument will be given in 
\cite{baldassarri-kedlaya}.
\begin{theorem} \label{T:global subharmonic}
Choose $x \in X$ of type 2.
Let $c(x)$ be the number of skeletal branches of $X$ at $x$.
(Note that if $x \notin \Gamma_S$, then $g(x) = c(x) = 0$.)
\begin{enumerate}
\item[(a)]
For $i=1,\dots,m(x)$, 
the function $\log s_i(M,S,\cdot)$ is affine of slope $0$ along all but finitely many branches of $X$ at $x$.
In particular, we may form the sum $\mu_i$ of the slopes of the function $\sum_{j=1}^i \log s_j(M,S,\cdot)$ along all of the branches of $X$ at $x$ (in the directions away from $x$).
\item[(b)]
If $x \notin \Gamma_S$, then $\mu_i \leq 0$ for $i=1,\dots,m(x)$.
\item[(c)]
If $x \in \Gamma_S$ is internal,
then $\mu_i \leq (2g(x) -2 + c(x))i$ for $i=1,\dots,m(x)$.
\item[(d)]
In (b) and (c), equality holds if $i = m(x)$. Equality also holds if $i < n$ and $s_i(M,S,x) < s_{i+1}(M,S,x)$.
\end{enumerate}
\end{theorem}
\begin{proof}
We may assume $K = \CC$, so that $\kappa_K$ is algebraically closed. 
If $x \notin \Gamma_S$, by rescaling we may reduce the claims to an instance of Proposition~\ref{P:variation}(c),
so we may assume hereafter that $x \in \Gamma_S$.

Suppose first that $X$ is contained in the affine line; in this case, we may follow the proof of 
\cite[Theorem~11.3.2(c)]{kedlaya-book}. Namely, using Frobenius pushforwards as in
Definition~\ref{D:global Frobenius descendant}
(and using both Proposition~\ref{P:descendant} and Proposition~\ref{P:off-center descendant}),
we may reduce to the case where $s_i(M,S,x) < \omega \rho_S(x)$.
In this case, the claims follow by first using 
Corollary~\ref{C:cyclic vector} to choose an element of $M_x$ which is a cyclic
vector for $M_x \otimes_{\calO_{X,x}} \Frac(\calO_{X,x})$ for the derivation $\frac{d}{dt}$, then applying
Proposition~\ref{P:christol-dwork}. 

We now treat the case of general $X$.
Let $C$ be a
smooth projective connected curve over $\kappa_K$ with function field
$\kappa_{\calH(x)}$.
Choose a nonconstant $\overline{f} \in \kappa_{\calH(x)}$ of degree $d>0$, then
choose $f \in \calO_{X,x}$ with $x(f)=1$ lifting $\overline{f}$
which is unramified at each point corresponding to a branch named in (a).
Note that removing part of $X$ contained in a branch adds $i$ to both sides of the desired inequality
and is thus harmless; we can thus ensure that $f$ defines a finite \'etale map $X \to X'$ for $X'$
a subspace of the affine line.
Put $x' = f(x')$ and let $S'$ be the image of $S$. 
For each branch $U'$ of $X'$ at $x$, the slope of $\sum_{j=1}^{di} \log s_j(f_* M,S',U')$ can be computed
as follows. Let $P'$ be the point of $C$ corresponding to $U'$. For each point $P \in \overline{f}^{-1}(P')$
with multiplicity $m$ and ramification number $e$ (so that $e=m$ if the ramification at $P$ is tame),
let $U$ be the corresponding branch of $X$ at $x$; we then get a contribution of 
$1-e$ plus the slope of $\sum_{j=1}^i \log s_{j}(M,S,U)$
(as may be verified using Frobenius descendants).
We thus deduce the claim from the previous case plus the Riemann-Hurwitz formula.
\end{proof}

To show that the functions $s_i(M,S,\cdot)$ can be computed using some triangulation, we use the following criterion.
\begin{lemma} \label{L:sufficient triangulation}
Let $T$ be a triangulation containing $S$ with the following properties.
\begin{enumerate}
\item[(a)]
The set $\Gamma_T$ meets every connected component of $X \setminus \Gamma_S$. In particular, the retraction
$\pi_T$ exists (see Remark~\ref{R:retract}).
\item[(b)]
Along each edge of $\Gamma_T$, the functions $\log s_i(M,S,\cdot)$ are affine for $i=1,\dots,n$.
\item[(c)]
For each $x \in T$, for $i=1,\dots,m(x)$, the slope of $\log s_i(M,S,\cdot)$ along any nonskeletal branch
of $X$ at $x$ is $0$.
\end{enumerate}
Then for $i=1,\dots,n$, $\log s_i(M,S,\cdot)$ factors as the retraction $\pi_{T}$ followed by a piecewise affine
function on $\Gamma_{T}$. 
\end{lemma}
\begin{proof}
Note that (b) implies that (c) holds also for $x \in \Gamma_T$ by Proposition~\ref{P:variation}(c,d).
We may thus deduce the claim using Lemma~\ref{L:variation0 convexity} (applied after enlarging $K$ to turn
a virtual open disc into a true open disc) and Lemma~\ref{L:affine at higher genus}.
\end{proof}

We then obtain the following generalization of Theorem~\ref{T:strict skeleton},
which recovers the main results of \cite{pulita-poineau2, pulita-poineau3, pulita-poineau}.
\begin{theorem} \label{T:global triangulation}
There exists a triangulation $T$
containing $S$
such that $\Gamma_T$ meets every connected component of $X \setminus \Gamma_S$
(so the retraction $\pi_T$ exists by Remark~\ref{R:retract})
and each function $\log s_i(M,S,\cdot)$ factors as $\pi_{T}$ followed by a piecewise affine
function on $\Gamma_{T}$. In particular, the functions $s_1(M,S,x), \dots, s_n(M,S,x)$ on $X$ are continuous. 
\end{theorem}
\begin{proof}
Since $X$ is locally compact, it suffices to check the claim locally around some $x \in X$.
If $x \notin \Gamma_S$, the claim follows from Theorem~\ref{T:strict skeleton},
so we need only consider $x \in \Gamma_S$.
By Theorem~\ref{T:global subharmonic}(a),
we can choose $T$ so that for $i=1,\dots,m(x)$,
the slope of $\log s_i(M,S,\cdot)$ is 0 along each $T$-nonskeletal branch of $X$
at $x$. By Proposition~\ref{P:variation}(a), we may
draw an open star in $\Gamma_{T}$ around $x$ such that on each edge,
the functions $\log s_i(M, S,\cdot)$ are affine for $i=1,\dots,n$.
On this star, the conditions of Lemma~\ref{L:sufficient triangulation} are satisfied,
so the desired result follows.
\end{proof}

One can also change the functions to match the new triangulation without disturbing the conclusion.
\begin{defn}
We say that a triangulation $T$ is \emph{controlling} for $M$ if the functions $s_i(M,T,\cdot)$
also factor as the retraction $\pi_T$ followed by some piecewise affine functions on $\Gamma_T$.
That is, we must be able to take $T = S$ in the conclusion of Theorem~\ref{T:global triangulation}.
\end{defn}

\begin{cor} \label{C:global triangulation}
In the notation of Theorem~\ref{T:global triangulation}, the triangulation $T$ is controlling.
\end{cor}
\begin{proof}
This follows from Theorem~\ref{T:global triangulation} and the fact that conditions (a,b)
of Lemma~\ref{L:sufficient triangulation} can be stated in terms of intrinsic subsidiary radii, 
and so remain valid if we replace $S$ by $T$.
\end{proof}
\begin{cor}
Let $T$ be a triangulation containing $S$ with the following properties.
\begin{enumerate}
\item[(a)]
The set $\Gamma_T$ meets every connected component of $X \setminus \Gamma_S$.
\item[(b)]
Along each edge of $\Gamma_T$, the functions $\sum_{i=1}^n \log s_i(M,S,\cdot)$ and $\sum_{i=1}^{n^2} \log s_i(\End(M),S,\cdot)$ are affine for $i=1,\dots,n$.
\item[(c)]
For each $x \in T$, the slope of $\sum_{i=1}^{m(x)}\log s_i(M,S,\cdot)$ along any nonskeletal branch
of $X$ at $x$ is $0$.
\end{enumerate}
Then for $i=1,\dots,n$, $\log s_i(M,S,\cdot)$ factors as the retraction $\pi_{T}$ followed by a piecewise affine
function on $\Gamma_{T}$. In particular, by Corollary~\ref{C:global triangulation}, $T$ is controlling.
\end{cor}

\begin{proof}
It suffices to verify the conditions of Lemma~\ref{L:sufficient triangulation}. Condition (a) is true by hypothesis. Condition (b) holds by Lemma~\ref{L:spectral decomposition exists}.
To check condition (c), note that
for $i=1,\dots,m(x)$,
the slope of $\log s_i(M,S,\cdot)$ at $x$ is nonnegative by
Proposition~\ref{P:transfer}, but the sum of these slopes is 0 so each slope individually must equal 0.
\end{proof}

\begin{cor}
There exists a strict triangulation $T$ which is controlling for $M$.
\end{cor}
\begin{proof}
We construct an increasing sequence of triangulations $T_0,\dots,T_n$ such that for $i=0,\dots,n$, the retraction $\pi_{T_i}$ exists and for $j = 1,\dots,i$, the functions $\log s_j(M,T_i,\cdot)$ factor as $\pi_{T_i}$ followed by a piecewise affine function on $\Gamma_{T_i}$.
To begin, let $T_0$ be any strict triangulation of $M$ for which the retraction $\pi_{T_0}$ exists.
Given $T_i$ for some $i \in \{0,\dots,n-1\}$, by Theorem~\ref{T:global triangulation} there exists a triangulation $T_{i+1}$ containing $T_i$ such that $\log s_{i+1}(S,T_i, \cdot)$ factors as $\pi_{T}$ followed by a piecewise affine function on $\Gamma_T$. If $i=0$, then 
Proposition~\ref{P:divisible closure} ensures that $T_{i+1}$ can be chosen to be strict. If $i>0$, we may make the same argument after applying Proposition~\ref{P:decompose disc} to separate the first $i-1$ radii in each disc.
\end{proof}

\begin{remark}
The methods of \cite{baldassarri-kedlaya, pulita-poineau2, pulita-poineau3, pulita-poineau},
when considered without reference to this paper,
can only prove a weaker version of Theorem~\ref{T:global triangulation}: they only provide a controlling
triangulation over a sufficiently large analytic field $K'$ containing $K$. As in Remark~\ref{R:pulita}, the problem is that 
this triangulation may involve vertices which project to type 4 points of the original curve,
which our methods are able to rule out. In the language of \cite{baldassarri},
we are able to exhibit a controlling strictly semistable model already over $\CC$, 
whereas the methods of \cite{pulita-poineau2, pulita-poineau3, pulita-poineau} provide such
a model only over a possibly larger algebraically
closed analytic field containing $\CC$.
\end{remark}

\begin{remark}
One can also give a variant of Theorem~\ref{T:global triangulation} for meromorphic (possibly irregular) connections; in this case, one must allow triangulations to have vertices at type 1 points (namely the poles of the connection). This result is described in \cite{baldassarri-kedlaya}.
\end{remark}

\subsection{Clean decompositions}
\label{subsec:clean}

One has an analogue of the spectral decomposition for the stalk of $M$ at a point $x \in X$.
Using Theorem~\ref{T:global triangulation}, we can extend this decomposition to specific subspaces of $X$.
\begin{lemma} \label{L:pointwise decomposition}
Choose $x \in X$ of type 2 or 3.
\begin{enumerate}
\item[(a)]
There exists a unique direct sum decomposition $M_x = \bigoplus_i N_i$ whose base extension to $\calH(x)$
is the spectral decomposition.
\item[(b)]
There exists a finite \'etale extension $S$ of $\calO_{X,x}$ such that $M_x \otimes_{\calO_{X,x}} S$
admits a direct sum decomposition whose base extension to $\calH(x) \otimes_{\calO_{X,x}} S$ is a
refined decomposition.
\end{enumerate}
\end{lemma}
\begin{proof}
Part (a) follows by using
the pushforward argument from the proof of Theorem~\ref{T:global subharmonic}
to reduce to the case where $X$ is contained in the affine line over $K$; this cases is the Dwork-Robba decomposition theorem
\cite[\S 4, Theorem, p. 20]{dwork-robba}, or can alternatively be derived by following the proof of \cite[Theorem~12.3.2]{kedlaya-book}.
Part (b) follows similarly upon noting that the local ring $\calO_{X,x}$ is henselian.
\end{proof}

\begin{theorem} \label{T:clean decomposition}
Let $T$ be a controlling triangulation for $M$.
\begin{enumerate}
\item[(a)]
For $x \notin \Gamma_T$, let $U$ be the branch of $\pi_T(x)$ containing $x$.
Then the restriction of $M$ to $U$ splits as a direct sum in which for each summand $N$,
there exists a constant $c>0$ such that $s_i(N,T,y) = c$ for $i=1,\dots,\rank(N)$ and $y \in U$.
\item[(b)]
For $x \in \Gamma_T$,
let $E$ be the open star around $x$ (i.e., the union of $x$ with the interiors of the edges of $\Gamma_T$ incident
upon $x$) and put $U = \pi_T^{-1}(E)$. 
Then there exists a unique direct sum decomposition of $M$ whose base extension to $\calH(x)$ is the spectral decomposition.
\end{enumerate}
\end{theorem}
\begin{proof}
Part (a) is immediate from Proposition~\ref{P:decompose disc}. To obtain (b), first apply
Lemma~\ref{L:pointwise decomposition} to obtain a decomposition over an uncontrolled open neighborhood $V$
of $x$. Note that $V$ already contains all but finitely many branches of $X$ at $x$. For each remaining branch $W$, it remains to construct a second decomposition which agrees with the first one on $V \cap W$. If $W$ is $T$-nonskeletal, this is immediate from (a). If $W$ is $T$-skeletal, 
we may apply Lemma~\ref{L:spectral decomposition exists} to obtain a decomposition in which each summand has a unique spectral radius along $E \cap W$. In each summand has a unique limiting spectral radius at $x$. If we group summands by limiting spectral radius, the resulting decomposition
agrees with the original one on $V \cap W$, as desired.
\end{proof}

\begin{remark}
The conclusion of Theorem~\ref{T:clean decomposition}(b) is best possible in certain senses. For one, one cannot ensure that the base extension of the decomposition to $\calH(y)$ is the spectral decomposition at any $y \in E \setminus \{x\}$, because of the coarsening step in the proof of Theorem~\ref{T:clean decomposition}.
Similarly, one cannot extend the decomposition to another vertex of $\Gamma_T$.
\end{remark}

\begin{remark}
The decompositions appearing in Theorem~\ref{T:clean decomposition} are analogues of the \emph{good formal structures} for formal meromorphic connections described in \cite{kedlaya-goodformal1,
kedlaya-goodformal2}. Additional analogues in the $p$-adic setting also appear in
\cite{kedlaya-xiao}. The decompositions given here can be used to obtain a global index formula
for connections on analytic curves, in the style of the work of Robba
\cite{robba-index1, robba-index2, robba-index3, robba-index4} and
Christol and Mebkhout \cite{cm1, cm2, cm3, cm4}. Such a formula will appear in a forthcoming paper of Baldassarri and the author.
\end{remark}

\begin{remark}
Using these results, it is tempting to look for a more global version of Theorem~\ref{T:p-adic turrittin}.
When $p>0$, one might even guess that every connection \'etale-locally satisfies the Robba condition.
However, this guess is incorrect as shown by Remark~\ref{R:bad tannakian example},
and it is not immediately obvious to us how to salvage the statement.

One motivation for doing so would be to show that the behavior of radii of convergence
for connections arising from discrete representations of the geometric fundamental group,
which can be explained in terms of Faber's Berkovich-theoretic ramification locus \cite{faber1, faber2}, is in fact
completely representative of the general case.
\end{remark}

\end{document}